\documentclass[a4paper, 10pt, reqno]{amsart}
\usepackage[a4paper,left=2.5cm,right=2.5cm,top=3cm,bottom=3cm]{geometry}
\usepackage[T1]{fontenc}
\usepackage{lmodern}

\usepackage[foot]{amsaddr}
\usepackage{amsmath,amssymb,graphicx, mathtools, amsthm} 
\usepackage{enumerate} 
\usepackage{enumitem}
\usepackage[colorlinks, linkcolor=red, urlcolor=blue, citecolor=blue]{hyperref}
\usepackage{url}    
\numberwithin{equation}{section}
\usepackage{tikz}

\usepackage[giveninits=true,url=false,doi=false,isbn=false,eprint=true,datamodel=mrnumber,sorting=nyt,sortcites=false,maxcitenames=4,maxbibnames=99,backref=false,block=space,backend=biber,style=phys, biblabel=brackets]{biblatex} 
\AtEveryBibitem{\clearfield{month}}
\AtEveryCitekey{\clearfield{month}} 
\renewbibmacro{in:}{}
\DeclareFieldFormat[article]{title}{\emph{#1}} 
\DeclareFieldFormat{mrnumber}{\ifhyperref{\href{http://www.ams.org/mathscinet-getitem?mr=#1}{\nolinkurl{MR#1}}}{\nolinkurl{#1}}}
\DeclareFieldFormat{pmid}{\ifhyperref{\href{https://www.ncbi.nlm.nih.gov/pubmed/#1}{\nolinkurl{PMID#1}}}{\nolinkurl{#1}}}
\DeclareFieldFormat{eprint}{\ifhyperref{\href{https://arxiv.org/abs/#1}{\nolinkurl{arXiv:#1}}}{\nolinkurl{#1}}}
\renewbibmacro*{doi+eprint+url}{%
  \iftoggle{bbx:doi}{\printfield{doi}}{}
  \newunit\newblock%
  \printfield{mrnumber}%
  \newunit\newblock%
  \printfield{pmid}%
  \newunit\newblock%
  \printfield{eprint}%
  \iftoggle{bbx:url}{\usebibmacro{url+urldate}}{}}

\bibliography{refs3}
\usepackage{caption}
\usepackage{float}
\usepackage{subcaption}

\DeclareMathOperator{\tr}{Tr}

\DeclareMathOperator{\var}{Var}

\newcommand{\R}{\mathbb{R}} 
\newcommand{\E}{\mathbb{E}} 
\newcommand{\N}{\mathbb{N}}  

\theoremstyle{plain}
\newtheorem{thm}{Theorem}[section]

\newtheorem{lem}[thm]{Lemma}
\newtheorem{prop}[thm]{Proposition}

\theoremstyle{definition}
\newtheorem{defn}[thm]{Definition}
\newtheorem{ex}[thm]{Example}
\newtheorem{hyp}[thm]{Assumption}
\theoremstyle{remark}
\newtheorem{rmk}[thm]{Remark}

\makeatletter
\renewcommand\subsection{\@startsection{subsection}{2}%
  \z@{-.5\linespacing\@plus-.7\linespacing}{.5\linespacing}%
  {\normalfont\scshape}}
\renewcommand\subsubsection{\@startsection{subsubsection}{3}%
  \z@{.5\linespacing\@plus.7\linespacing}{-.5em}%
  {\normalfont\scshape}}
\makeatother

\makeatletter
\@namedef{subjclassname@2020}{%
  \textup{2020} Mathematics Subject Classification}
\makeatother

\usepackage[format=plain,
            labelfont=sc]{caption}
\usepackage[toc,page]{appendix} 
\usepackage{scalerel,stackengine}
\stackMath%
\newcommand\reallywidehat[1]{%
\savestack{\tmpbox}{\stretchto{%
  \scaleto{%
    \scalerel*[\widthof{\ensuremath{#1}}]{\kern.1pt\mathchar"0362\kern.1pt}%
    {\rule{0ex}{\textheight}}
  }{\textheight}%
}{2.4ex}}%
\stackon[-6.9pt]{#1}{\tmpbox}%
}

\makeatletter
\newsavebox\myboxA
\newsavebox\myboxB
\newlength\mylenA
\newcommand*\xoverline[2][0.75]{%
    \sbox{\myboxA}{$\m@th#2$}%
    \setbox\myboxB\null
    \ht\myboxB=\ht\myboxA%
    \dp\myboxB=\dp\myboxA%
    \wd\myboxB=#1\wd\myboxA
    \sbox\myboxB{$\m@th\overline{\copy\myboxB}$}
    \setlength\mylenA{\the\wd\myboxA}
    \addtolength\mylenA{-\the\wd\myboxB}%
    \ifdim\wd\myboxB<\wd\myboxA%
       \rlap{\hskip 0.5\mylenA\usebox\myboxB}{\usebox\myboxA}%
    \else
        \hskip -0.5\mylenA\rlap{\usebox\myboxA}{\hskip 0.5\mylenA\usebox\myboxB}%
    \fi}
\makeatother    

\usepackage{fancyhdr}

\newcommand\shortitle{Global law of conjugate kernel random matrices with heavy-tailed weights}
\newcommand\name{Alice Guionnet and Vanessa Piccolo}

\begin{document}

\title{Global law of conjugate kernel random matrices with heavy-tailed weights}
\author{Alice Guionnet} 
\author{Vanessa Piccolo}
\address{A.G.\ and V.P.\ - Unit\'{e} de Math\'{e}matiques Pures et Appliqu\'{e}es (UMPA), ENS Lyon}
\email{alice.guionnet@ens-lyon.fr} 
\email{vanessa.piccolo@ens-lyon.fr}

\subjclass[2020]{60B20, 15B52, 68T07} 
\keywords{Two-layer conjugate kernel matrices, nonlinear random matrices, heavy-tailed weights, light-tailed data, moment method, traffic probability}
\date{\today}

\maketitle

\begin{abstract}
We study the asymptotic spectral distribution of the conjugate kernel random matrix \(YY^\top\), where \(Y= f(WX)\) arises from a two-layer neural network model. We consider the setting where \(W\) and \(X\) are random rectangular matrices with i.i.d.\ entries, where the entries of \(W\) follow a heavy-tailed distribution, while those of \(X\) have light tails. Our assumptions on \(W\) include a broad class of heavy-tailed distributions, such as symmetric \(\alpha\)-stable laws with \(\alpha \in ]0,2[\) and sparse matrices with \(\mathcal{O}(1)\) nonzero entries per row. The activation function \(f\), applied entrywise, is bounded, smooth, odd, and nonlinear. We compute the limiting eigenvalue distribution of \(YY^\top\) through its moments and show that heavy-tailed weights induce strong correlations between the entries of \(Y\), resulting in richer and fundamentally different spectral behavior compared to the light-tailed case.
\end{abstract}
{
\hypersetup{linkcolor=black}
\tableofcontents
}

\section{Introduction}

We study the asymptotic behavior of the eigenvalues of conjugate kernel random matrices \(Y Y^\top\), where \(Y = f(WX)\) arises from a two-layer feed-forward neural network. Here, \(W\) and \(X\) are random rectangular matrices with i.i.d.\ entries, representing the weight and data matrices, respectively, and \(f\) is a smooth nonlinear activation applied entrywise. The spectral properties of such models were first studied under Gaussian assumptions by Pennington and Worah~\cite{pennington2017}, and later extended to light-tailed distributions by Benigni and P\'{e}ch\'{e}~\cite{benigni2021}. Further generalizations have since been explored in~\cite{liao2018, adlam, schroder2021, pastur2022, dabo2024, speicher2024}.

In this work, we extend these results to the case where \(W\) follows a heavy-tailed distribution. Specifically, we consider settings in which the entries of \(W\) can take very large values and may lack finite second moments, while the entries of \(X\) remain light-tailed. This framework is motivated by empirical observations from overparameterized neural networks, where strongly correlated weights frequently emerge, defying standard Gaussian assumptions~\cite{mahoney2021, mahoney2021bis, wang2023}. Heavy-tailed distributions may thus provide a more realistic framework for capturing the complex structure of these correlations. From a mathematical perspective, random features matrices \(Y = f(WX)\) with heavy-tailed weights exhibit entirely new properties, as their entries happen to be much more correlated than in the light-tailed setting. For instance, consider the case where \(W\) is the adjacency matrix of an Erd\"os-R\'enyi graph, where each edge is drawn independently at random with probability \(q\) over the dimension. In this case, the entry \(Y_{ij} = f(W_i \cdot X_j)\) is strongly correlated with entries \(Y_{i'j}\) whenever \(i\) and \(i'\) share a common neighbor \(k\). This arises because both \(W_i \cdot X_j\) and \(W_{i'} \cdot X_j\) share the term \(X_{kj}\), which does not vanish as the dimension grows. As a result, in each column of \(Y\) there are approximately \(q\) randomly chosen entries that are strongly correlated. These strong dependencies introduce new analytical challenges in understanding the spectral behavior of \(YY^\top\). In this work, we prove convergence of the moments of the empirical eigenvalue distribution for a broad class of conjugate kernel matrices with heavy-tailed weights and light-tailed features.

\subsection{Model and main results}

We begin by introducing the random features model which is defined by a random weight matrix and a random data matrix.
Let \(W \in \R^{p \times n}\) be a random weight matrix with i.i.d.\ entries drawn from a distribution \(\nu_w\), and let \(X \in \R^{n \times m}\) be a random data matrix with i.i.d.\ entries drawn from a distribution \(\nu_x\). We impose the following assumptions on \(\nu_w\) and \(\nu_x\).

\begin{hyp}[Distributions \(\nu_w\) and \(\nu_x\)] \label{hyp1}
\leavevmode
\begin{enumerate}
\item[(a)] The distribution \(\nu_w\) is symmetric, and the random variables \(W_{ij}\) have a characteristic function satisfying, for every \(\lambda \in \R\), 
\[
\Phi_n(\lambda) \coloneqq n \log \E_W \left [\exp (i \lambda W_{ij})\right ]  \to \Phi(\lambda), \quad \text{as} \, n \to \infty,
\]
for some limiting function \(\Phi\). The convergence is assumed to be uniform in \(\lambda\). Since \(\nu_w\) is symmetric, its characteristic function is real-valued and even, and therefore both \(\Phi_n\) and \(\Phi\) are even functions.
\item[(b)] The distribution \(\nu_x\) is symmetric and has finite moments of all orders. We denote its variance by \(\E_X \left [X_{ij}^2 \right ] = \sigma_x^2\).
\end{enumerate}
\end{hyp}

The class of random matrices covered by Assumption~\ref{hyp1}(a) includes both heavy-tailed matrices in the classical probabilistic sense, such as Lévy matrices, and matrices with light-tailed entries but sparse structure, such as sparse Wigner matrices. First, consider a \emph{Lévy matrix}, for which the entries \(W_{ij}\) are independent, \emph{symmetric \(\alpha\)-stable} random variables with  \(\alpha \in ]0,2[\). Specifically, 
\begin{equation}\label{levy}
W_{ij} = \frac{1}{n^{1/\alpha}} A_{ij},
\end{equation}
where \((A_{ij})\) are i.i.d.\ symmetric \(\alpha\)-stable random variables with characteristic function \(\E_A \left [ \exp(i \lambda A_{ij})\right ]  = \exp (-\sigma^\alpha |\lambda|^\alpha)\) for all \(\lambda \in \R\), with scale parameter \(\sigma > 0\). Note that the cases \(\alpha=1\) and \(\alpha=2\) correspond to the Cauchy and Gaussian distributions, respectively. In this case,
\[
\E_W \left [ \exp(i \lambda W_{ij})\right ] = \E_A   \left [ \exp(i \lambda n^{-1/\alpha}A_{ij})\right ]  = \exp (- n^{-1}\sigma^\alpha |\lambda|^\alpha), 
\]
so that \(\Phi_n (\lambda) = \Phi (\lambda) = -\sigma^\alpha |\lambda|^\alpha \). Second, consider a \emph{sparse Wigner matrix}, for which the entries \(W_{ij}\) are defined by
\begin{equation}\label{hada}
W_{ij} = B_{ij} Z_{ij},
\end{equation}
where \((B_{ij})\) are independent Bernoulli\( (q/n) \) random variables with \(q \in ]0,1[\), and \((Z_{ij})\) are i.i.d.\ symmetric random variables, independent of the dimension and independent of the family \((B_{k \ell})\). In this case, a direct computation yields $ \Phi_n (\lambda) = n\log ( 1+\frac{q}{n} \left (  \E_Z \left [e^{i \lambda Z_{ij}} \right ] -1\right))$ and \(\Phi(\lambda) = q \left (  \E_Z \left [e^{i \lambda Z_{ij}} \right] -1\right)\). \\

Let \(f \colon \R \to \R\) be a function satisfying the following.
\begin{hyp}[Activation function]\label{hyp2}
The activation function \(f \colon \R \to \R\) is odd, belongs to \(\mathcal{C}^\infty(\R)\), and satisfies \(f^{(k)} \in L^1(\R)\) for all integers \(k \ge 0\).
\end{hyp}
This assumption implies that \(f\) is bounded and uniformly continuous. Typical examples include smooth, rapidly decaying odd functions such as \(f(x)=xe^{-x^2}\) or \(f(x) = \sin(x)e^{-x^2}\), for which all derivatives are integrable. Although many standard activation functions (e.g., \(\tanh\), \(\arctan\), or \(\sin\)) do not satisfy the \(L^1(\R)\) requirement, one can enforce this condition by multiplying them by a smooth even cutoff or mollifier, thereby obtaining an exponentially or polynomially decaying variant. Such modifications preserve smoothness and oddness and affect only a finite-rank component of the random features matrix defined below in~\eqref{eq: matrix Y}, leaving the limiting spectral distribution invariant.\\

We consider the two-layer conjugate kernel random matrix
\[
M \coloneqq Y_m Y_m^\top \in \R^{p \times p}, 
\]
where \(Y_m = (Y_{ij})_{i \in [p], j \in [m]}\) is defined by
\begin{equation} \label{eq: matrix Y}
Y_{ij} \coloneqq \frac{1}{\sqrt{m}} f( W_i \cdot X_j) =  \frac{1}{\sqrt{m}} f \left ( \sum_{k=1}^n W_{ik} X_{kj} \right).
\end{equation}
By Assumptions~\ref{hyp1} and~\ref{hyp2}, both \(\nu_x\) and \(\nu_w\) are symmetric and \(f\) is odd; hence, the entries \(Y_{ij}\) are symmetric random variables. \\

Our goal is to study the empirical spectral distribution of \(M\) in the high-dimensional proportional-growth regime.

\begin{hyp}[Linear-width regime] \label{hyp3}
We assume that
\[
\frac{n}{m} \to \phi \quad \text{and} \quad \frac{n}{p} \to \psi \quad \text{as} \enspace m,p,n \to \infty,
\]
where \(\phi\) and \(\psi\) are two positive constants.
\end{hyp}

Our first main result concerns the convergence of the moments of the empirical spectral measure \(\hat{\mu}_M \coloneqq \frac{1}{p} \sum_{i=1}^p \delta_{\lambda_i}\), where \(\lambda_1, \ldots, \lambda_p\) denote the eigenvalues of \(M\).

\begin{thm}[Convergence of matrix moments] \label{main1}
Under Assumptions~\ref{hyp1}-\ref{hyp3}, for every integer \(k \in \N\), there exists a real number \(m_k\), depending only on  \(\phi,\psi,f, \Phi\), and \(\nu_x\), such that 
\[
\lim_{m, p, n \rightarrow \infty }\frac{1}{p} \tr M^k = m_k,
\]
where the convergence holds both in expectation and in probability. The limiting moment \(m_k\) is given explicitly in Proposition~\ref{prop: main cycle}.
\end{thm}

Theorem~\ref{main1} shows that the limiting moments \(m_k\) also depend on the distribution of the input variables \(\nu_x\). Indeed, according to Proposition~\ref{prop: main cycle}, the coefficients appearing in the expression of \(m_k\) involve expectations of the form \(\E_X \left [ \Phi \left ( \sum_i a_i X_i \right) \right ]\), for i.i.d.\ random variables \(X_i \sim \nu_x\) (see Definition~\ref{def: C_W (f)}). Whenever the function \(\Phi\) is non-quadratic (that is, beyond the Gaussian case), this expectation depends on higher-order moments of \(\nu_x\), not only on its variance \(\sigma_x^2\). This leads to a \emph{non-universal} global spectral behavior of \(M\), even for light-tailed input distributions. This contrasts with the universality result of~\cite{benigni2021}, where the limiting moments \(m_k\) depend on the light-tailed distributions \(\nu_w\) and \(\nu_x\) only through their second moments \(\sigma_w^2\) and \(\sigma_x^2\). In Lemma~\ref{lem: moments alpha=2}, we show that the limiting moments \(m_k\) given in Proposition~\ref{prop: main cycle} coincide exactly with those obtained in~\cite[Theorem 3.5]{benigni2021} in the special case where \(\nu_w\) is Gaussian and \(\nu_x\) is light-tailed, as specified in Assumption~\ref{hyp1}(b). This non-universality of the global spectral behavior is however also true for Wigner matrices with heavy tails~\cite{zakarevich2006,BGM,benarous2008}.

From Theorem~\ref{main1}, we can deduce the weak convergence of the empirical spectral measure \(\hat{\mu}_M\), provided the moments \(m_k\) grow sufficiently slowly to uniquely determine a probability measure. We establish this convergence result under the following additional assumption.
\begin{hyp} \label{hyp4}
We assume that \(\Phi\) satisfies one of the following:
\begin{itemize}
\item[(a)] \(\Phi\) is bounded on \(\R\);
\item[(b)]\(\Phi (\lambda) = - \sigma^\alpha |\lambda|^\alpha\) with \(\alpha \in ]0,2[\) and \(\sigma >0\), and \(\nu_x\) is the centered normal distribution with variance \(\sigma_x^2\) .
\end{itemize}
\end{hyp}
The first case in Assumption~\ref{hyp4} includes, for instance, sparse Wigner matrices~\eqref{hada}, where \(|\Phi (\lambda) | \le 2q\) for all \(\lambda \in \R\), while the second case corresponds to L\'{e}vy matrices~\eqref{levy}.

\begin{thm}[Global law] \label{main1bis}
Under Assumptions~\ref{hyp1}-\ref{hyp3} and assuming either (a) or (b) in Assumption~\ref{hyp4}, there exists a unique probability measure \(\mu\), depending on \(\phi,\psi,f, \Phi\), and \(\nu_x\), supported on the non-negative real line, such that for every integer \(k \in \N\), 
\[
m_k =\int x^k \textnormal{d} \mu (x). 
\]
Furthermore, the empirical spectral measure \(\hat{\mu}_M\) converges weakly almost surely to \(\mu\).
\end{thm}

When the limiting measure \(\mu\) exists (in particular, under Assumption~\ref{hyp4}), it exhibits a light-tailed behavior since \(\int_\R x \text{d} \mu(x)\le \|f\|_\infty^2\), and all moments \(m_k\) are finite. This stands in sharp contrast to classical heavy-tailed models and is mainly a consequence of the boundedness of the activation function \(f\). If \(f\) is unbounded, the entries of \(Y_m\) may inherit heavy tails from the weight distribution, and the matrix moments may diverge. In such cases, the method of moments can not be used and our approach fails. Moreover, the asymptotic behavior then depends on the growth rate of \(f\) relative to the tail index of the weight distribution. For instance, if \(f\) grows at least linearly and the weights follow a symmetric \(\alpha\)-stable law with \(\alpha<2\), the resulting spectrum is expected to exhibit heavy-tailed features and may require a different normalization or truncation. Conversely, if \(f\) grows sublinearly, the corresponding moments remain finite, and a light-tailed global law may still be expected. A rigorous analysis of these regimes would require analytical tools distinct from the Fourier and cumulant expansions employed here (see Subsection~\ref{subsection: outline proof}). Figure~\ref{fig:1} illustrates the empirical spectral distribution of the conjugate kernel matrix \(M\) under condition (a) of Assumption~\ref{hyp4}, using the activation function \(f(x) = \arctan(x)\), for different values of the parameter \(\alpha\). As \(\alpha\) increases, the eigenvalue distribution transitions from a heavy-tailed regime with widely spread eigenvalues to a more concentrated spectrum, highlighting the effect of the weight distribution on spectral behavior.

\begin{figure}[h] 
\centering
\includegraphics[scale=0.6]{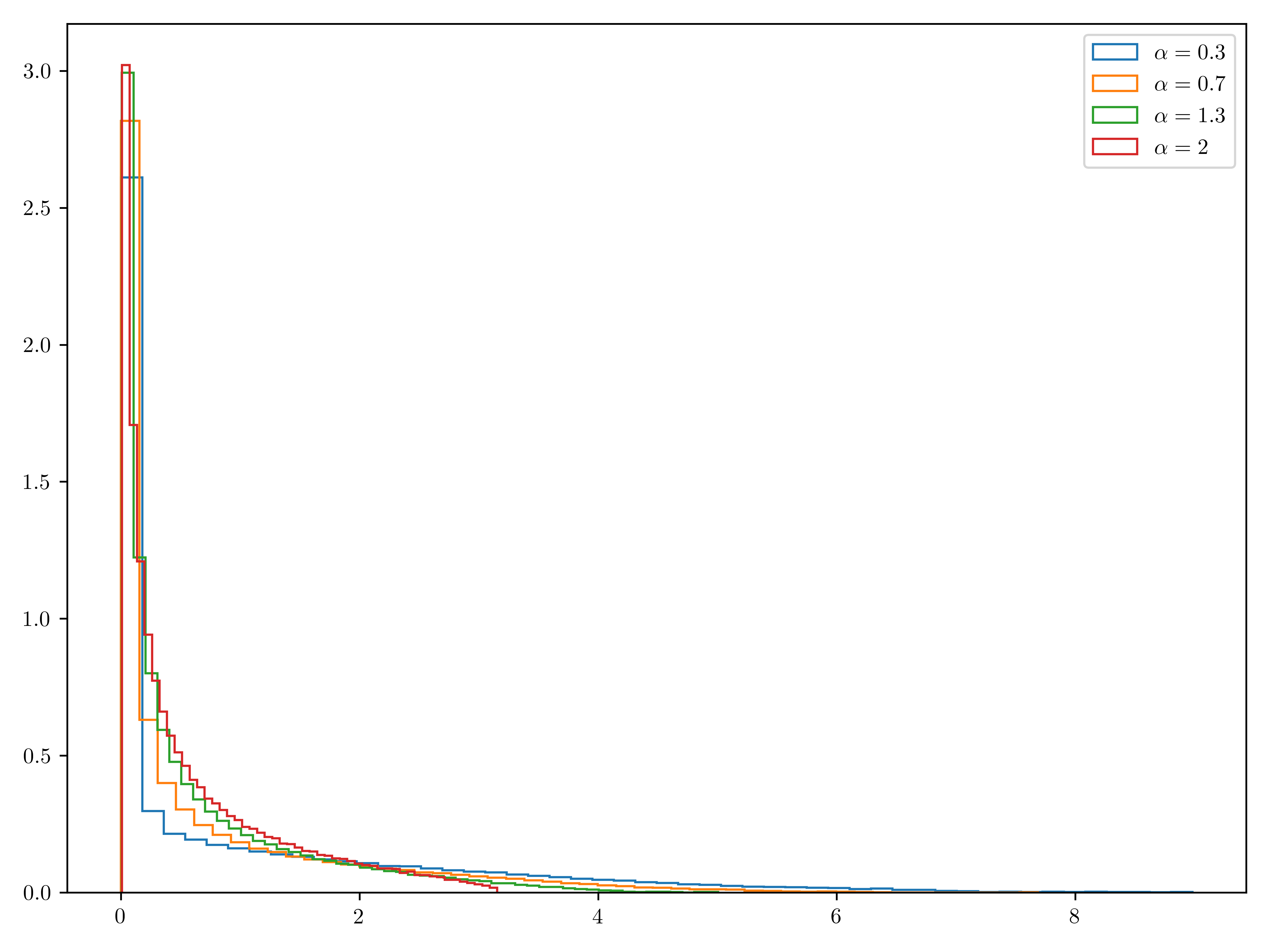}
\caption{Eigenvalue histogram of \(M=Y_mY_m^\top\) for the activation function \(f(x) = \arctan(x)\). The weight distribution \(\nu_w\) follows a symmetric \(\alpha\)-stable distribution with \(\sigma = 1\) and different values of \(\alpha \in ]0,2]\), while \(\nu_x\) is the standard normal distribution. Numerical experiments were conducted with \(m=n=10000\) and \(p=6500\).}
\label{fig:1}
\end{figure}

\subsection{Related work}

The study of random matrices with nonlinear dependencies was initiated by El Karoui~\cite{elkaroui} and Cheng and Singer~\cite{cheng} in the context of \emph{random inner-product kernel matrices}, where the nonlinearity is applied to the sample covariance matrix---formally, \( f (X^\top X)\), with \(X\) being a rectangular matrix with i.i.d.\ entries. In the case of Gaussian \(X\) and in the linear-width regime, the bulk eigenvalues asymptotically follow the free convolution of the semicircle and Marchenko-Pastur distributions~\cite{cheng,fan2019}. More recently,~\cite{lu2023} extended these results to the polynomial growth regime, and~\cite{dubova2023} further generalized them to settings where \(X\) has i.i.d.\ entries with finite moments, identifying the bulk as a (weighted) free convolution of semicircle and Marchenko-Pastur laws~\cite{dubova2023} and demonstrating the universality of this phenomenon. In the quadratic regime,~\cite{pandit2024} analyzed inner-product kernel matrices for data with general covariance structure (under a Gaussian moment-matching assumption), showing that their spectra and kernel ridge regression performance admit precise asymptotic descriptions.

Nonlinear random matrices also arise from two-layer feed-forward neural networks, where they play a central role in understanding training dynamics and generalization. In this context, rather than applying a nonlinearity to the sample covariance matrix, one instead considers the sample covariance matrix of a nonlinearity. This is the focus of the present work, which considers random matrices of the form \(YY^\top\), where \(Y = f(WX)\) denotes the so-called \emph{random features matrix}. The spectral properties of such matrices in high-dimensional settings are closely tied to the expected training loss and generalization error of the associated neural networks. From a mathematical perspective, characterizing the asymptotic spectrum of the random matrix \(YY^\top\) is challenging due to the nonlinear dependencies introduced by the activation function. These dependencies make the analysis significantly more complex than that of classical linear random matrix ensembles. The global law for such \emph{conjugate kernel} matrices was first studied by Pennington and Worah~\cite{pennington2017} in the setting where \(W\) and \(X\) have i.i.d.\ centered Gaussian entries. Their result was later extended by Benigni and P\'{e}ch\'{e}~\cite{benigni2021} to matrices with sub-Gaussian tails and real analytic activation functions. P\'{e}ch\'{e}~\cite{peche2019} further showed that the nonlinear random matrix \(YY^\top\) is asymptotically equivalent to a \emph{Gaussian linear model}, where the asymptotic effect of the nonlinearity is captured by a linear combination of the involved matrices and an additional independent Gaussian matrix. Building on this line of work, the second author in collaboration with Schr\"{o}der~\cite{schroder2021} analyzed the practically important case of the biased random features model \(Y = f(WX+B)\), where \(B\) is an independent rank-one Gaussian matrix. They used the resolvent method and cumulant expansion---rather than the moment method employed in earlier works~\cite{pennington2017,benigni2021}---to study the spectral behavior of this model. Speicher and Wendel~\cite{speicher2024} recently extended this perspective by computing cumulants for a broader class of nonlinear random matrices, where the nonlinearity is applied to symmetric orthogonally invariant random matrices, and showed that a Gaussian equivalence principle holds. In parallel, Dabo and Male~\cite{dabo2024} further generalized the model by considering random matrices with variance profiles, namely matrices where the variance of the entries varies from one variable to another. They showed that such models are asymptotically \emph{traffic-equivalent} to an information-plus-noise type sample covariance matrix, consistent with the Gaussian equivalence principle of~\cite{peche2019}. Independently of this line of work, Louart, Liao, and Couillet~\cite{couillet2018} initiated a complementary approach to analyzing conjugate kernel random matrices, focusing on the case where \(X\) is deterministic, \(W\) is a random (with entries given by functions of standard Gaussian random variables), and \(f\) is a Lipschitz activation function. Using concentration inequalities, they derived a deterministic equivalent for the expectation of the resolvent and showed that the eigenvalue distribution aligns with that of a standard sample covariance matrix. A key work in this direction is by Fan and Wang~\cite{fanwang2020}, who established a global law for the conjugate and neural tangent kernel matrices of multi-layer networks with Gaussian weights and analytic activations, showing that the limiting spectral distribution can be expressed via free probability convolutions. Wang and Zhou~\cite{wang2024} further analyzed the regime where the network width is much larger than the sample size, confirming the universality of these spectral phenomena. More recently, Chouard~\cite{chouard2023} strengthened these results by deriving quantitative global and local deterministic equivalents, allowing non-differentiable activations and bias terms.

In this article, we study conjugate kernel random matrices with light-tailed inputs and heavy-tailed weights. Linear models of symmetric matrices with independent heavy-tailed entries have been extensively analyzed in~\cite{zakarevich2006,benarous2008,bordenave2011,bordenave2011bis}. These matrices fall outside the Wigner universality class. Specifically, while the empirical measure of their eigenvalues converges, the limiting distribution is not the semicircular law. Instead, it is a probability measure with unbounded support. Depending on the model, this limit can exhibit atoms \cite{salez2015}, as in the case of adjacency matrices of Erd\"os-R\'enyi graphs, or have a smooth density, such as when the entries follow an \(\alpha\)-stable distribution \cite{Belinschi2009}. The eigenvalue fluctuations resemble those of independent random variables~\cite{BGM}. However, the local spectral fluctuations remain largely unknown, except in the case of \(\alpha\)-stable entries, where certain regimes exhibit fluctuations similar to the Gaussian Orthogonal Ensemble~\cite{bordenave2013,bordenave2017,aggarwal2021}. In contrast, the behavior of conjugate kernel matrices with heavy-tailed weights is even less understood. In this paper, the empirical spectral measure of these models has light tails, in fact all finite moments, although we conjecture that the limiting distribution is not compactly supported. For the eigenvalue fluctuations, we conjecture that they follow the usual scaling of the central limit theorem which agrees with our rough bounds on the covariance derived in Section~\ref{sec:cov}.

\subsection{Outline of proofs} \label{subsection: outline proof}

Our approach to proving the weak convergence of the empirical spectral measure of \(M=Y_m Y_m^\top\) is based on the classical method of moments. Specifically, for every integer \(k \geq 1\), we aim to compute the normalized trace of the \(k\)th power of \(M\):
\begin{equation} \label{eq: moments expansion}
\E \left [\frac{1}{p} \tr M^k \right ]  = \frac{1}{p} \sum_{1 \leq i_1,\ldots, i_k \leq p} \sum_{1 \leq j_1,\ldots, j_k \leq m} \E \left [\prod_{\ell=1}^k Y_{i_\ell j_\ell} Y_{i_{\ell+1} j_\ell}\right ],
\end{equation}
with the convention that \(i_{k+1}=i_1\). The right-hand side of~\eqref{eq: moments expansion} admits a natural graphical interpretation as a sum over closed walks of length \(2k\) on a bipartite graph: one set of vertices indexed by \(\{i_1, i_2, \ldots, i_k\}\) and the other by \(\{j_1, j_2, \ldots, j_k\}\), with edges connecting \(i_\ell\) to \(j_\ell\) and \(j_\ell\) to \(i_{\ell +1}\) for every \(1 \leq \ell \leq k\). Unlike the light-tailed regime, where most terms in this expansion vanish asymptotically, the present framework gives rise to a substantially larger number of non-negligible contributions. To account for these systematically, we employ a refined version of the moment method inspired by \emph{traffic probability theory}~\cite{male2020}. This approach extends classical moment computations by considering more general functionals of the matrix entries, referred to as \emph{injective moments}. The traffic probability approach was also recently used by Dabo and Male~\cite{dabo2024} to compute the traffic trace of \(M\) for profiled matrices \(W\) and \(X\) with light-tailed distributions. To formalize our approach, we begin by recalling the definition of \emph{traffic traces} introduced by Male~\cite{male2020}.

\begin{defn} [Traffic trace~\cite{male2020}] \label{def: traffic}
\leavevmode
\begin{enumerate}
\item[(a)] A \emph{test graph of matrices} consists of a triple \(T = (V, E, \mathcal{A})\), where \(G = (V,E)\) is a finite, oriented graph (possibly with multiple edges and loops) and \(\mathcal{A} = (A_e)_{e \in E}\) is a collection of \(N \times N\) matrices labeling each edge \(e \in E\) in \(G\). 
\item[(b)] For every test graph \(T = (V, E, \mathcal{A})\), the \emph{traffic trace} is defined by
\[
\tau_N [T] \coloneqq  \E \left [\frac{1}{N^c} \sum_{\phi \colon V \to [N]} \prod_{e = (u,v) \in E} A_e(\phi(u), \phi(v)) \right ],
\]
where \(c\) denotes the number of connected components of the graph \(G = (V,E)\).
\item[(c)] For every test graph \(T = (V, E, \mathcal{A})\), the \emph{mean injective trace} is defined by
\[
\tau_N^0 [T] \coloneqq  \E \left [\frac{1}{N^c} \sum_{\phi \colon V \to [N] \atop \textnormal{s.t.} \: \phi \: \textnormal{is injective}} \prod_{e = (u,v) \in E} A_e(\phi(u), \phi(v)) \right ].
\]
The \emph{traffic trace} is then recovered via:
\[
\tau_N [T]  = \sum_{\pi \in \mathcal{P}(V)} \tau^0_N [T^\pi],
\]
where the sum runs over all partitions of \(V\) and \(T^\pi\) denotes the test graph obtained from \(T\) by identifying the vertices within each block of \(\pi\). 
\item[(d)] We say that a collection \(\mathcal{A}\) of \(N \times N\) matrices \emph{converges in traffic distribution} if, for any test graph \(T = (V, E, \mathcal{A})\), the traffic trace \(\tau_N [T]\) converges as \(N \to \infty\). 
\end{enumerate}
\end{defn} 

In our setting, it suffices to consider non-oriented bipartite multigraphs, as the matrix \(M=Y_m Y_m^\top\) is self-adjoint. According to Definition~\ref{def: traffic}, for a non-oriented, finite, connected bipartite test graph \(T =(W \cup V, E,Y_m)\), with edges running from \(W\) to \(V\) and labeled by the random matrix \(Y_m\), the mean injective trace is given by
\begin{equation} \label{eq: tau0}
\tau^0_{p,m,n} \left [T  \right ] = \E \left [ \frac{1}{p} \sum_{\phi_W \colon W \to [p] \atop \phi_W  \textnormal{injective}}  \sum_{\phi_V \colon V \to [m] \atop \phi_V \textnormal{injective}} \prod_{e=(w,v)\in E} \left ( Y_m(\phi_{W}(w),\phi_{V}(v))\right )^{m(e)} \right],
\end{equation}
where \(m(e)\) denotes the multiplicity of edge \(e \in E\). By item (c) of Definition~\ref{def: traffic}, the normalized tracial moments of \(Y_m Y_m^\top\) can then be written as
\begin{equation} \label{traffic cycle}
\E \left [\frac{1}{p} \tr M^k \right] = \tau_{p, m, n} \left [T_{\text{cycle}} \right ] = \sum_{\pi\in \mathcal{P} (W\cup V)} \tau^0_{p,m,n} \left [T_{\text{cycle}}^\pi \right ],
\end{equation}
where \(T_{\text{cycle}} = (G, Y_m)\) is the test graph with \(G = (W \cup V, E)\) being the simple bipartite cycle of length \(2k\). Our strategy for proving the convergence of matrix moments---and thereby proving item (a) of Theorem~\ref{main1}---is to prove the convergence in traffic distribution of \(Y_m\). Due to the invariance property described in (c) of Definition~\ref{def: traffic}, this is equivalent to proving the convergence of the mean injective trace.

\begin{thm}[Convergence in traffic distribution]  \label{main2} 
Under Assumptions~\ref{hyp1}-\ref{hyp3}, the random matrix \(Y_m\) converges in traffic distribution. Specifically, for every finite, connected bipartite test graph \(T= (W \cup V, E , Y_m)\), there exists a real number \(\tau^0_G\), depending only on $G=(V\cup W, E)$ and \(\phi, \psi, f, \Phi, \nu_x\), such that 
\[
\lim_{m,p,n \rightarrow \infty} \tau^0_{p,m,n} \left [ T \right] = \tau^0_G.
\]
Moreover, the limiting injective trace \(\tau^0_G\) is defined explicitly in Proposition~\ref{main3}.
\end{thm}

To compute the injective trace \(\tau^0_{p,m,n} \left [ T \right] \), we begin by expanding~\eqref{eq: tau0} using the Fourier inversion theorem, which states that \(f(x)=\frac{1}{2\pi}\int_{\R} \hat{f}(t) e^{itx} \textnormal{d} t\). Specifically, we obtain  
\[
\tau^0_{p,m,n} \left [T  \right ]  =  \frac{1}{p m^{|E|/2}} \sum_{\phi_W \colon W \to [p] \atop \phi_W  \textnormal{injective}}  \sum_{\phi_V \colon V \to [m] \atop \phi_V \textnormal{injective}} \frac{1}{(2\pi)^{|E|}} \int_{\R^{|E|}} \prod_{e \in E} \prod_{i=1}^{m(e)} \text{d} \gamma_e^i \hat f (\gamma_e^i) \Lambda_G^n (\boldsymbol{\gamma}) ,
\]
where \(|E| = \sum_{e \in E} m(e)\), \(\boldsymbol{\gamma} = ( \gamma_e^1,\ldots, \gamma_e^{m(e)})_{e \in E}\), and 
\[
\Lambda_G^n (\boldsymbol{\gamma}) \coloneqq \E \left [ \exp \left ( i \sum_{e=(w,v)\in E}  (\gamma_e^1 + \ldots+ \gamma_e^{m(e)}) W_{\phi_W(w)} \cdot X_{\phi_V(v)}  \right) \right].
\]
The use of the Fourier inversion theorem and the interchange of integration and expectation via Fubini’s theorem is justified by Assumption~\ref{hyp2}, which ensures that \(f, \hat{f} \in L^1(\R)\) and that \(\hat{f}\) decays faster than any polynomial. To identify the leading-order contributions, we expand \(\Lambda_G^n\) by taking the expectation with respect to \(W\) and using the function \(\Phi\) from item (a) of Assumption~\ref{hyp1}. This is carried out in detail in Subsection~\ref{subsection: proof main prop}. Alternatively, integrating first with respect to \(X\) provides a complementary perspective on the computation. If, for instance, \(\nu_x\) is the centered normal distribution, then
\[
\Lambda_G^n (\boldsymbol{\gamma}) = \E \left [ e^{i \tr (X E_G(\boldsymbol{\gamma})^\top W)}\right ]  = \E_W \left [ e^{ - \frac{1}{2} \sigma_x^2 \tr \left ( E_G(\boldsymbol{\gamma}) W W^\top E_G(\boldsymbol{\gamma})^\top \right)} \right ],
\]
where \(E_G(\boldsymbol{\gamma}) \in \R^{p \times m}\) is the matrix with entries given by \((E_G(\boldsymbol{\gamma}))_{\phi_W(w), \phi_V(v)} = (\gamma_e^1 + \cdots + \gamma_e^{m(e)}) \mathbf{1}_{e = (w,v) \in E}\). This approach shows that the random features matrix $Y_{m}$ can be viewed as a matrix with random covariance structure of the form $E_{G}(\boldsymbol{\gamma})^\top E_{G}(\boldsymbol{\gamma})$. While this point of view highlights the correlation structure among matrix entries, we were not able to exploit it to obtain a closed-form description of the limiting spectral distribution.

\subsection{Overview}
In Section~\ref{section: convergence traffic}, we present the proof of Theorem~\ref{main2}. In particular, we identify the connected bipartite graphs that contribute to the limiting injective trace \(\tau_G^0\), relying on key combinatorial estimates from Section~\ref{section: combinatorics}. This result is then applied in Section~\ref{section: cycle} to compute the asymptotics of the normalized tracial moments of \(M=Y_mY_m^\top\), by studying the traffic trace \(\tau_{p,m,n} [T_{\text{cycle}}]\) associated with the simple bipartite cycle test graph \(T_{\text{cycle}}\) (see in~\eqref{traffic cycle}). More precisely, Subsection~\ref{subsection: matrix moments} provides the proof of the convergence in expectation stated in Theorem~\ref{main1}. The convergence in probability follows by Chebyshev's inequality, using a variance estimate obtained in Subsection~\ref{sec:cov}. Finally, Section~\ref{sec:concen} proves the almost sure weak convergence of the empirical spectral measure. \\

\textbf{Acknowledgements.}\ We thank G\'{e}rard Ben Arous and Camille Male for invaluable discussions throughout the project. This work was supported by the ERC Advanced Grant LDRaM (No.\ 884584).

\section{Convergence in traffic distribution} \label{section: convergence traffic}

The goal of this section is to compute the limiting injective trace \(\tau_G^0\) from Theorem~\ref{main2}, thereby proving convergence in traﬃc distribution. To this end, we identify the connected bipartite test graphs \(T = (G,Y_m)\) for which the limiting injective trace \(\tau_G^0\) does not vanish. 

\subsection{Preliminaries}

For certain test graphs, the mean injective trace vanishes due to the symmetry of the distributions of the entries $W_{ij}$ and $X_{ij}$. In general, however, the limit of the injective trace is more intricate, especially due to the correlations among the entries \(Y_{ij}\). To handle these dependencies, we now introduce several definitions that describe the limiting behavior. 

Let \(G=(W \cup V,E)\) be a finite, connected bipartite multigraph, where \(E\) is a multiset of edges and \(m \colon E \to \N\) assigns a multiplicity to each edge \(e\in E\). The total number of edges is then given by \(|E| = \sum_{e \in E} m(e)\), and the degree of a vertex \(x \in W \cup V\) is defined by \(\deg(x)=\sum_{y \colon y \sim x} m((x,y))\), where \(w \sim v\) stands for \((w,v) \in E\). 

\begin{defn}[Induced subgraphs]  \label{def: subgraphs}
Let \(G = (W \cup V, E)\) be a connected bipartite multigraph. For some positive integer \(K\), consider subsets \(W_1, \ldots, W_K\) of \(W\) such that \(\cup_{i=1}^K W_i \subseteq W\), which may have a nontrivial intersection. For each subset \(W_i\), we define the bipartite subgraph \(G_i = G(W_i)\) induced by \(W_i\), where 
\[
V_i= \{ v \in V \colon \exists \, w \in W_i \enspace \textnormal{such that} \enspace w \sim v\},
\]
and
\[
E_i =\{e = (w,v) \in E \colon w \in W_i, v \in V_i\}.
\] 
The degree of a vertex \(x \in W_i \cup V_i\) within the subgraph \(G_i\) is denoted by \(\deg_{G_i}(x)\) and its degree in the entire graph \(G\) by \(\deg(x)\). By construction, for every \(w \in W_i\), the set \(E_i\) includes all edges in \(G\) that are incident to \(w\), so that \(\deg(w) = \deg_{G_i}(w)\). For every \(2 \leq k \leq K\) and every \(1 \leq \ell_1 < \cdots < \ell_k \leq K\), we define the sets of common vertices in \(W\) and \(V\), respectively, as follows:
\[
W_{G_{\ell_1}, \ldots, G_{\ell_k}} \coloneqq  \left \{ w \in W \colon \exists \, 1 \leq i < j \leq k \enspace \text{such that} \enspace w \in W_{\ell_i} \cap W_{\ell_j} \right \},
\]
and
\[
V_{G_{\ell_1}, \ldots, G_{\ell_k}} \coloneqq \left \{ v \in V \colon \exists \, 1 \leq i < j \leq k \enspace \text{such that} \enspace v \in V_{\ell_i} \cap V_{\ell_j} \right \},
\]
respectively.
\end{defn}

\begin{ex} \label{ex1}
Figure~\ref{fig1} provides an example of a connected bipartite graph \(G = (W \cup V, E)\), along with three different collections of subsets \(W_1, \ldots, W_K\) of \(W\) (with \(K \in \{2,3\}\)) and the corresponding induced subgraphs \(G(W_1), \ldots, G(W_K)\). 
\end{ex}

\begin{figure} 
\centering

\begin{subfigure}[]{0.8\textwidth} 
\centering

\begin{tikzpicture}

\draw (0,0) circle [radius=1];
\def\nodesA{8}
\foreach \i in {1,...,\nodesA} {
    \coordinate (P\i) at ({360/\nodesA * (\i - 1)}:1); 
    \ifodd\i 
        \fill (P\i) circle [radius=2pt]; 
    \else        
        \draw[fill=white] (P\i) circle [radius=2pt]; 
    \fi
}

\node[right] at (P1) {\tiny $v_1$};
\node[above] at (P2) {\tiny $w_1$};
\node[above] at (P3) {\tiny $v_4$};
\node[left] at (P4) {\tiny $w_4$};
\node[left] at (P5) {\tiny $v_3$};
\node[below] at (P6) {\tiny $w_3$};
\node[below] at (P7) {\tiny $v_2$};
\node[below] at (P8) {\tiny $w_2$};

\draw (2.2,0) circle [radius=1.2];
\def\nodesB{12}

\foreach \i in {1, ..., \nodesB} {
        
\pgfmathsetmacro\angle{360/\nodesB * \i}
        
\pgfmathsetmacro\xpos{2.2 + 1.2 * cos(\angle)}
\pgfmathsetmacro\ypos{0 + 1.2 * sin(\angle)}
        
\coordinate (Q\i) at (\xpos, \ypos);
       
\ifodd\i
\draw[fill=white] (\xpos, \ypos) circle [radius=2pt];  
        \else
            \fill[black] (\xpos, \ypos) circle [radius=2pt];  
        \fi
    }
    
\node[right] at (Q1) {\tiny $w_5$};
\node[above] at (Q2) {\tiny $v_{10}$};
\node[above] at (Q3) {\tiny $w_{10}$};
\node[above] at (Q4) {\tiny $v_9$};
\node[right] at (Q5) {\tiny $w_9$};
\node[right] at (Q7) {\tiny $w_8$};
\node[below] at (Q8) {\tiny $v_7$};
\node[below] at (Q9) {\tiny $w_7$};
\node[below] at (Q10) {\tiny $v_6$};
\node[right] at (Q11) {\tiny $w_6$};
\node[right] at (Q12) {\tiny $v_5$};
    
\pgfmathsetmacro\xcoord{cos(30)} 
\pgfmathsetmacro\ycoord{sin(30)}

\draw (2.2 + 2 * \xcoord, 2 * \ycoord)  circle [radius=0.8];    
\def\nodesC{6}

\foreach \i in {1, ..., \nodesC} {
        
        \pgfmathsetmacro\angle{30 + 360/\nodesC * \i}
        
        \pgfmathsetmacro\xpos{2.2 + 2 * \xcoord + 0.8 * cos(\angle)}
        \pgfmathsetmacro\ypos{0 + 2 * \ycoord + 0.8 * sin(\angle)}
        
         \coordinate (R\i) at (\xpos, \ypos);
         
        \ifodd\i
             \draw[fill=white] (\xpos, \ypos) circle [radius=2pt];  
        \else
	 \fill[black] (\xpos, \ypos) circle [radius=2pt];  
        \fi
    }

\node[above] at (R1) {\tiny $w_{11}$};
\node[right] at (R2) {\tiny $v_{13}$};
\node[above] at (R4) {\tiny $v_{12}$};
\node[right] at (R5) {\tiny $w_{12}$};    
\node[right] at (R6) {\tiny $v_{11}$};
 
\end{tikzpicture}
\caption{A connected bipartite graph \(G = (W \cup V, E)\), where \(W = \{w_1, \ldots, w_{12}\}\) and \(V = \{v_1, \ldots, v_{12}\}\).}
\label{subfig1}
\end{subfigure}

\hfill

\begin{subfigure}[]{0.8\textwidth}
\centering

\begin{tikzpicture}

\draw (-1.5,0) circle [radius=1];
\def\nodesA{8}

\foreach \i in {1, ..., \nodesA} {
        
\pgfmathsetmacro\angle{360/\nodesA * \i}
\pgfmathsetmacro\xpos{-1.5 + cos(\angle)}
\pgfmathsetmacro\ypos{0 + sin(\angle)}
  
 \coordinate (P\i) at (\xpos, \ypos);      
  \ifodd\i
           \draw[fill=white] (\xpos, \ypos) circle [radius=2pt];  
        \else
            \fill[black] (\xpos, \ypos) circle [radius=2pt];  
        \fi
    }

\node[right] at (P1) {\tiny $w_1$};
\node[above] at (P2) {\tiny $v_4$};
\node[left] at (P3) {\tiny $w_4$};
\node[left] at (P4) {\tiny $v_3$};
\node[left] at (P5) {\tiny $w_3$};
\node[below] at (P6) {\tiny $v_2$};
\node[right] at (P7) {\tiny $w_2$};
\fill[red] (P8)   circle [radius=2pt];  
\node[right] at (P8) {\tiny $v_1$};

\draw (2.2,0) circle [radius=1.2];
\def\nodesB{12}

\draw (2.2,0) circle [radius=1.2];
\def\nodesB{12}

\foreach \i in {1, ..., \nodesB} {
        
\pgfmathsetmacro\angle{360/\nodesB * \i}
        
 \pgfmathsetmacro\xpos{2.2 + 1.2 * cos(\angle)}
 \pgfmathsetmacro\ypos{0 + 1.2 * sin(\angle)}
        
         \coordinate (Q\i) at (\xpos, \ypos);  
         
        \ifodd\i
           \draw[fill=white] (\xpos, \ypos) circle [radius=2pt];  
        \else
            \fill[black] (\xpos, \ypos) circle [radius=2pt];  
        \fi
    }
    
\node[right] at (Q1) {\tiny $w_5$};
\node[above] at (Q2) {\tiny $v_{10}$};
\node[above] at (Q3) {\tiny $w_{10}$};
\node[above] at (Q4) {\tiny $v_9$};
\node[left] at (Q5) {\tiny $w_9$};
\node[left] at (Q6) {\tiny $v_1$};
\fill[red] (Q6) circle [radius=2pt]; 
\node[left] at (Q7) {\tiny $w_8$};
\node[below] at (Q8) {\tiny $v_7$};
\node[below] at (Q9) {\tiny $w_7$};
\node[below] at (Q10) {\tiny $v_6$};
\node[right] at (Q11) {\tiny $w_6$};
\node[right] at (Q12) {\tiny $v_5$};

\pgfmathsetmacro\xcoord{cos(30)} 
\pgfmathsetmacro\ycoord{sin(30)} 

\draw (2.2 + 2 * \xcoord, 2 * \ycoord)  circle [radius=0.8];    
\def\nodesC{6}

\foreach \i in {1, ..., \nodesC} {
        
        \pgfmathsetmacro\angle{30 + 360/\nodesC * \i}
        
        \pgfmathsetmacro\xpos{2.2 + 2 * \xcoord + 0.8 * cos(\angle)}
        \pgfmathsetmacro\ypos{0 + 2 * \ycoord + 0.8 * sin(\angle)}
        
         \coordinate (R\i) at (\xpos, \ypos); 
         
        \ifodd\i
             \draw[fill=white] (\xpos, \ypos) circle [radius=2pt];  
        \else
	 \fill[black] (\xpos, \ypos) circle [radius=2pt];  
        \fi
    }
    
\node[above] at (R1) {\tiny $w_{11}$};
\node[right] at (R2) {\tiny $v_{13}$};
\node[above] at (R4) {\tiny $v_{12}$};
\node[right] at (R5) {\tiny $w_{12}$};    
\node[right] at (R6) {\tiny $v_{11}$};    
 
\end{tikzpicture}
\caption{The two connected subgraphs \(G_1=G(W_1)\) and \(G_2=G(W_2)\) obtained from the disjoint subsets \(W_1 = \{w_1, \ldots, w_4\}\) and \(W_2 = \{w_5, \ldots, w_{12}\}\), respectively, according to Definition~\ref{def: subgraphs}. Moreover, \(W_{G_1,G_2} = \emptyset\) and \(V_{G_1, G_2} = \{v_1\}\).}
\label{subfig2}
\end{subfigure}

\hfill

\begin{subfigure}[]{0.8\textwidth}

\centering

\begin{tikzpicture}

\draw (-1.5,0) circle [radius=1];
\def\nodesA{8}

\foreach \i in {1, ..., \nodesA} {
        
        \pgfmathsetmacro\angle{360/\nodesA * \i}
        
        \pgfmathsetmacro\xpos{-1.5 + cos(\angle)}
        \pgfmathsetmacro\ypos{0 + sin(\angle)}
        
        \coordinate (P\i) at (\xpos, \ypos); 
        
        \ifodd\i
           \draw[fill=white] (\xpos, \ypos) circle [radius=2pt];  
        \else
            \fill[black] (\xpos, \ypos) circle [radius=2pt];  
        \fi
    }

\node[right] at (P1) {\tiny $w_1$};
\node[above] at (P2) {\tiny $v_4$};
\node[left] at (P3) {\tiny $w_4$};
\node[left] at (P4) {\tiny $v_3$};
\node[left] at (P5) {\tiny $w_3$};
\node[below] at (P6) {\tiny $v_2$};
\node[right] at (P7) {\tiny $w_2$};
\node[right] at (P8) {\tiny $v_1$};  
\fill[red] (P8) circle [radius=2pt];  
    
\draw (2.2,0) circle [radius=1.2];
\def\nodesB{12}

\draw (2.2,0) circle [radius=1.2];
\def\nodesB{12}

\foreach \i in {1, ..., \nodesB} {
        
        \pgfmathsetmacro\angle{360/\nodesB * \i}
        
        \pgfmathsetmacro\xpos{2.2 + 1.2 * cos(\angle)}
        \pgfmathsetmacro\ypos{0 + 1.2 * sin(\angle)}
        
        \coordinate (Q\i) at (\xpos, \ypos); 
        
        \ifodd\i
           \draw[fill=white] (\xpos, \ypos) circle [radius=2pt];  
        \else
            \fill[black] (\xpos, \ypos) circle [radius=2pt];  
        \fi
    }
 
\node[right] at (Q1) {\tiny $w_5$};
\node[above] at (Q2) {\tiny $v_{10}$};
\node[above] at (Q3) {\tiny $w_{10}$};
\node[above] at (Q4) {\tiny $v_9$};
\node[left] at (Q5) {\tiny $w_9$};
\node[left] at (Q6) {\tiny $v_1$};
\fill[red] (Q6) circle [radius=2pt]; 
\node[left] at (Q7) {\tiny $w_8$};
\node[below] at (Q8) {\tiny $v_7$};
\node[below] at (Q9) {\tiny $w_7$};
\node[below] at (Q10) {\tiny $v_6$};
\node[right] at (Q11) {\tiny $w_6$};
\node[right] at (Q12) {\tiny $v_5$};

\pgfmathsetmacro\xcoord{cos(30)} 
\pgfmathsetmacro\ycoord{sin(30)} 

\pgfmathsetmacro\xangle{cos(150)} 
\pgfmathsetmacro\yangle{sin(150)} 

\pgfmathsetmacro\xanglee{cos(270)} 
\pgfmathsetmacro\yanglee{sin(270)} 

\coordinate (A) at (2.2 + 1.2 * \xcoord, 1.2 * \ycoord);
\coordinate (B) at (2.2 + 2 * \xcoord + 0.8 * \xangle,  2 * \ycoord + 0.8 * \yangle);
\coordinate (C) at (2.2 + 2 * \xcoord + 0.8 * \xanglee,  2 * \ycoord + 0.8 * \yanglee);


 \draw (A) to[bend left = 20] (B);
 \draw (A) to[bend right = 20] (C);

\draw[fill=white] (A) circle [radius=2pt]; 
\fill[magenta] (B) circle [radius=2pt];   
\fill[teal] (C) circle [radius=2pt];
\node[right] at (B) {\tiny $v_{13}$};
\node[right] at (C) {\tiny $v_{12}$};

\pgfmathsetmacro\xcord{cos(-30)} 
\pgfmathsetmacro\ycord{sin(-30)} 
\pgfmathsetmacro\xpi{cos(180)} 
\pgfmathsetmacro\ypi{sin(180)} 

\coordinate (D) at (6+ 0.8 * \xpi, 0.8 * \ypi);
\coordinate (E) at (6 + 2 * \xpi + 1.2 * \xcoord, 1.2 * \ycoord);
\coordinate (F) at (6 + 2 * \xpi + 1.2 * \xcord, 1.2 * \ycord);

\node[left] at (D) {\tiny $w_5$}  ;
\node[left] at (E) {\tiny $v_{10}$};
\node[left] at (F) {\tiny $v_5$};

 \draw (E) to[bend left = 20] (D);
 \draw (F) to[bend right = 20] (D);

\fill[blue] (E) circle [radius=2pt];   
\fill[green] (F) circle [radius=2pt];
\fill[blue] (Q2) circle [radius=2pt];   
\fill[green] (Q12) circle [radius=2pt];

\draw (6,0)  circle [radius=0.8];    
\def\nodesC{6}

\foreach \i in {1, ..., \nodesC} {
        
        \pgfmathsetmacro\angle{360/\nodesC * \i}
        
        \pgfmathsetmacro\xpos{6 + 0.8 * cos(\angle)}
        \pgfmathsetmacro\ypos{0 + 0.8 *sin(\angle)}
        
         \coordinate (R\i) at (\xpos, \ypos); 
         
        \ifodd\i
             \draw[fill=white] (\xpos, \ypos) circle [radius=2pt];  
        \else
	 \fill[black] (\xpos, \ypos) circle [radius=2pt];  
        \fi
    }

\fill[magenta] (R2) circle [radius=2pt];   
\fill[teal] (R4) circle [radius=2pt];
\node[above] at (R2) {\tiny $v_{13}$};
\node[below] at (R4) {\tiny $v_{12}$};

\node[above] at (R1) {\tiny $w_{11}$};
\node[right] at (R5) {\tiny $w_{12}$};    
\node[right] at (R6) {\tiny $v_{11}$}; 
   
\end{tikzpicture}
\caption{The three connected subgraphs \(G_1=G(W_1), G_3=G(W_3)\) and \(G_4=G(W_4)\) obtained from the subsets \(W_1 = \{w_1, w_2,w_3, w_4\}\), \(W_3 = \{w_5,w_6, w_7, w_8,w_9, w_{10}\}\), and \(W_4 = \{w_5, w_{11}, w_{12}\}\), respectively, according to Definition~\ref{def: subgraphs}. Moreover, \(W_{G_1, G_3} = \emptyset, W_{G_1, G_4} = \emptyset, W_{G_3, G_4} = \{w_5\}, W_{G_1, G_3, G_4} = \{w_5\}\) and \(V_{G_1, G_3} = \{v_1\}, V_{G_1, G_4} = \emptyset, V_{G_3, G_4} = \{v_2, v_3, v_4, v_5\}, V_{G_1, G_3, G_4} = \{v_1, v_2, v_3, v_4, v_5\}\).}
\label{subfig3}
\end{subfigure}

\hfill

\begin{subfigure}[]{0.8\textwidth}
\centering

\begin{tikzpicture}

\pgfmathsetmacro\xa{cos(45)} 
\pgfmathsetmacro\ya{sin(45)} 
\pgfmathsetmacro\xb{cos(90)} 
\pgfmathsetmacro\yb{sin(90)} 
\pgfmathsetmacro\xc{cos(135)} 
\pgfmathsetmacro\yc{sin(135)} 
\pgfmathsetmacro\xd{cos(180)} 
\pgfmathsetmacro\yd{sin(180)} 
\pgfmathsetmacro\xe{cos(225)} 
\pgfmathsetmacro\ye{sin(225)} 
\pgfmathsetmacro\xf{cos(270)} 
\pgfmathsetmacro\yf{sin(270)} 
\pgfmathsetmacro\xg{cos(315)} 
\pgfmathsetmacro\yg{sin(315)} 
\pgfmathsetmacro\xh{cos(360)} 
\pgfmathsetmacro\yh{sin(360)} 

\coordinate (1) at (-1.5 + \xa, \ya);
\coordinate (2) at (-1.5 + \xb, \yb);
\coordinate (3) at (-1.5 + \xc, \yc);
\coordinate (4) at (-1.5 + \xd, \yd);
\coordinate (6) at (-1.5 + \xf, \yf);
\coordinate (7) at (-1.5 + \xg, \yg);
\coordinate (8) at (-1.5 + \xh, \yh);

\draw (1) to[bend right=20] (2);
\draw (2) to[bend right=20]  (3);
\draw (3) to[bend right=20] (4);
\draw (6) to[bend right=20] (7);
\draw (7) to[bend right=20]  (8);
\draw (8) to[bend right=20] (1);

\draw[fill=white] (1) circle [radius=2pt]; 
 \fill[violet] (2) circle [radius=2pt];    
\draw[fill=white] (3) circle [radius=2pt];    
 \fill[cyan] (4) circle [radius=2pt];    
 \fill[orange] (6) circle [radius=2pt];    
\draw[fill=white] (7) circle [radius=2pt];  
 \fill[red] (8) circle [radius=2pt];      

\node[right] at (1) {\tiny  $w_1$}  ;
\node[above] at (2) {\tiny  $v_4$}  ;
\node[left] at (3) {\tiny  $w_4$}  ;
\node[left] at (4) {\tiny  $v_3$}  ;
\node[below] at (6) {\tiny  $v_2$}  ;
\node[right] at (7) {\tiny  $w_2$}  ;
\node[right] at (8) {\tiny  $v_1$}  ;

\coordinate (9) at (-3.5 + \xc, \yc);
\coordinate (10) at (-3.5 + \xd, \yd);
\coordinate (11) at (-3.5 + \xe, \ye);
\coordinate (12) at (-3.5 + \xf, \yf);
\coordinate (13) at (-3.5 + \xb, \yb);

\draw (13) to[bend right=20] (9);
\draw (9) to[bend right=20] (10);
\draw (10) to[bend right=20] (11);
\draw (11) to[bend right=20] (12);

 \fill[violet] (13) circle [radius=2pt]; 
\draw[fill=white] (9) circle [radius=2pt]; 
 \fill[cyan] (10) circle [radius=2pt];    
\draw[fill=white] (11) circle [radius=2pt];    
 \fill[orange] (12) circle [radius=2pt];  

\node[above] at (13) {\tiny  $v_4$}  ;
\node[below] at (12) {\tiny  $v_2$}  ;
\node[below] at (11) {\tiny  $w_3$}  ;
\node[left] at (10) {\tiny  $v_3$}  ;
\node[left] at (9) {\tiny  $w_4$}  ;

\draw (2.2,0) circle [radius=1.2];
\def\nodesB{12}

\draw (2.2,0) circle [radius=1.2];
\def\nodesB{12}

\foreach \i in {1, ..., \nodesB} {
        
\pgfmathsetmacro\angle{360/\nodesB * \i}
        
 \pgfmathsetmacro\xpos{2.2 + 1.2 * cos(\angle)}
 \pgfmathsetmacro\ypos{0 + 1.2 * sin(\angle)}
        
         \coordinate (Q\i) at (\xpos, \ypos);  
         
        \ifodd\i
           \draw[fill=white] (\xpos, \ypos) circle [radius=2pt];  
        \else
            \fill[black] (\xpos, \ypos) circle [radius=2pt];  
        \fi
    }
    
\node[right] at (Q1) {\tiny $w_5$};
\node[above] at (Q2) {\tiny $v_{10}$};
\node[above] at (Q3) {\tiny $w_{10}$};
\node[above] at (Q4) {\tiny $v_9$};
\node[left] at (Q5) {\tiny $w_9$};
\node[left] at (Q6) {\tiny $v_1$};
\fill[red] (Q6) circle [radius=2pt]; 
\node[left] at (Q7) {\tiny $w_8$};
\node[below] at (Q8) {\tiny $v_7$};
\node[below] at (Q9) {\tiny $w_7$};
\node[below] at (Q10) {\tiny $v_6$};
\node[right] at (Q11) {\tiny $w_6$};
\node[right] at (Q12) {\tiny $v_5$};

\pgfmathsetmacro\xcoord{cos(30)} 
\pgfmathsetmacro\ycoord{sin(30)} 

\draw (2.2 + 2 * \xcoord, 2 * \ycoord)  circle [radius=0.8];    
\def\nodesC{6}

\foreach \i in {1, ..., \nodesC} {
        
        \pgfmathsetmacro\angle{30 + 360/\nodesC * \i}
        
        \pgfmathsetmacro\xpos{2.2 + 2 * \xcoord + 0.8 * cos(\angle)}
        \pgfmathsetmacro\ypos{0 + 2 * \ycoord + 0.8 * sin(\angle)}
        
         \coordinate (R\i) at (\xpos, \ypos); 
         
        \ifodd\i
             \draw[fill=white] (\xpos, \ypos) circle [radius=2pt];  
        \else
	 \fill[black] (\xpos, \ypos) circle [radius=2pt];  
        \fi
    }
    
\node[above] at (R1) {\tiny $w_{11}$};
\node[right] at (R2) {\tiny $v_{13}$};
\node[right] at (R4) {\tiny $v_{12}$};
\node[right] at (R5) {\tiny $w_{12}$};    
\node[right] at (R6) {\tiny $v_{11}$};    

\end{tikzpicture}

\caption{The three connected subgraphs  \(G_5=G(W_5), G_6=G(W_6)\) and \(G_2=G(W_2)\) obtained from the subsets \(W_5 = \{w_1, w_2, w_4\}\), \(W_6 = \{w_3, w_4\}\), and \(W_2 = \{w_5, \ldots, w_{12}\}\), respectively, according to Definition~\ref{def: subgraphs}. Moreover, \(W_{G_5, G_6} = \{w_4\}, W_{G_5, G_2} = W_{G_6, G_2}=\emptyset\) and \(V_{G_5, G_6} = \{v_2, v_3, v_4\}, V_{G_5, G_2} = \emptyset, V_{G_6, G_2} = \{v_1\}, V_{G_5, G_6, G_2} = \{v_1, v_2, v_3, v_4\}\). }
\label{subfig4}
\end{subfigure}

\caption{An example of a connected bipartite graph \(G = (W \cup V, E)\), together with three distinct collections of subgraphs obtained by choosing subsets \(W_1, \ldots, W_K\) of \(W\) for \(K \in \{2,3\}\), as described in Definition~\ref{def: subgraphs}.}
\label{fig1}
\end{figure}

\begin{defn}[Block structure of connected bipartite multigraphs] \label{def: block structure}
Let \(G = (W \cup V, E)\) be a connected bipartite multigraph.
\begin{itemize}
\item[(a)] A vertex \(v \in V\) is called a \emph{separating vertex} if \(G\) can be decomposed into connected subgraphs \(G(W_1), \ldots, G(W_K)\), where \(W_1, \ldots, W_K \subseteq W\) and \(E_1, \ldots, E_K \subseteq E\) are disjoint, and \(v\) is the only vertex common to \(V_1, \ldots, V_K\).
\item[(b)] A \emph{block} \(B\) of \(G\) is a maximal connected subgraph \(B = (W_B \cup V_B, E_B)\) containing no separating vertices in \(V_B\). We write \(B = G(W_B)\) for the subgraph of \(G\) induced by the set \(W_B\) of vertices in \(W\) belonging to \(B\), as in Definition~\ref{def: subgraphs}.
\item[(c)] A graph \(G\) with \(N\) blocks \(B_1, \ldots, B_N\) is called a \emph{block tree} if, for every \(2 \le k \le N\) and every collection of distinct indices \(\ell_1, \ldots, \ell_k \in [N]\), one has \(|V_{B_{\ell_1}, \ldots, B_{\ell_k}}| \le k-1\).
\end{itemize}
\end{defn}

\begin{rmk}
The term ``block tree'' reflects the fact that the blocks of \(G\) form a tree-like structure: any random walk in \(G\) starting from a block \(B_1\), traversing \(k\) blocks, and returning to \(B_1\) must revisit at least one separating vertex. Otherwise, such a random walk would start from \(B_1\), pass through the blocks \(B_2, \ldots, B_k\), and return to \(B_1\) without revisiting any separating vertex---effectively forming a cycle. In such a case, \(|V_{B_1, \ldots, B_k}| =k\), contradicting the defining property of a block tree.
\end{rmk}

Examples of block trees include \emph{cactus graphs} and \emph{double trees}, defined as follows.

\begin{defn}[Cactus graph and double tree]
Let \(G = (W \cup V, E)\) be a connected bipartite graph. We say that \(G\) is a bipartite \emph{cactus graph} if it is a block tree (with blocks defined via separating vertices in \(V\)) such that each block is either a simple cycle or a union of simple cycles glued together at common vertices in \(W\). We call \(G\) a bipartite \emph{double tree} if it is a cactus graph in which every simple cycle in every block has length two. 
\end{defn}

\begin{ex}[Example~\ref{ex1} continued]
Consider the bipartite graph \(G\) shown in Figure~\ref{subfig1}. The vertex \(v_1 \in V\) is a separating vertex, and the subgraphs \(G_1\) and \(G_2\) in Figure~\ref{subfig2} are the two blocks of \(G\). Consequently, \(G\) is a block tree. More precisely, \(G\) is a cactus graph: \(G_1\) is a simple cycle, while \(G_2\) consists of two simple cycles connected at the vertex \(w_5 \in W\).
\end{ex}

In our setting, we restrict attention to a subclass of block trees, which we now define precisely.

\begin{defn}[Admissible graph] \label{def: admissible graph}
A connected bipartite multigraph \(G = (W \cup V, E)\) is said to be \emph{admissible} if it is a block tree with separating vertices in \(V\), such that within each block, vertices in \(W\) have even degree while vertices in \(V\) have degree \(2\).
\end{defn}

By item (b) of Definition~\ref{def: block structure}, we write \(B_i = G(W_{B_i})\) for each block of an admissible graph \(G\). Let $R$ denote the number of blocks $B_1,\ldots,B_R$ such that
$|W_{B_i}|\ge 2$ for each $1\le i\le R$, with the convention $R=0$ if every block contains exactly one vertex of $W$. Since separating vertices occur only in \(V\), the sets \(W_{B_1}, \ldots, W_{B_R}\) are disjoint. We then define
\[
S \coloneqq |W| - |\cup_{i=1}^R W_{B_i}| = |W| - \sum_{i=1}^R |W_{B_i}|,
\]
so that \(S\) is the number of blocks containing exactly one vertex from \(W\). In each of these \(S\) blocks, all edges have multiplicity \(2\), and each such block is therefore a double tree. Consequently, an admissible graph with \(R=0\) is itself a double tree. More generally, we note that every edge of an admissible graph has multiplicity either \(1\) or \(2\).

\begin{rmk}
Let \(B = G(W_B)\) be a block of an admissible graph with \(|W_B| \ge 2\). If \(\deg_B(w)=2\) for every \(w \in W_B\), then \(B\) is a simple cycle. More generally, \(B\) is a cactus graph if it consists of simple cycles that are connected only through vertices in \(W_B\). Therefore, an admissible graph \(G\) is a cactus graphs whenever all its \(R\) blocks containing more than one vertex from \(W\) are either simple cycles or cactus graphs.
\end{rmk}

For an admissible graph \(G = (W \cup V, E)\) and a block \(B = G(W_{B})\) with \(|W_B| \ge 2\), we introduce the notion of an admissible decomposition of \(W_B\).

\begin{defn}[Admissible decomposition of a block] \label{def: admissible decomposition}
Let \(B = G(W_B)\) be a block of an admissible graph \(G = (W \cup V, E)\) with \(|W_B| \ge 2\). A collection of subsets \(W_1, \ldots, W_K \subseteq W_B\) with \(K \ge 1\) is called an \emph{admissible decomposition} of \(W_B\) if 
\begin{enumerate}
\item[(a)] \(W_B = \cup_{i=1}^K W_i\);
\item[(b)] each \(W_i\) intersects at least one other \(W_j\), i.e., for all \(i \in [K]\) there exists \(j \neq i\) such that \(W_i \cap W_j \neq \emptyset\);
\item[(c)] \(|W_i| \ge 2\) for every \(i \in [K]\);
\item[(d)] the induced subgraph \(G_i = G(W_i)\) is connected for every \(i\);
\item[(e)] for every \(v \in V_B\), there exists at least one \(V_i\) such that \(\deg_{G_i}(v) \ge 2\),
\item[(f)] for every \(2 \le k \le K\) and distinct \(\ell_1, \ldots, \ell_k \in [K]\), \(|W_{G_{\ell_1}, \ldots, G_{\ell_k}}| \le k-1\).
\end{enumerate}
The set of admissible decompositions of \(W_B\) into \(K\) subsets is denoted by \(\mathcal{A}_K(W_B)\). Note that \(\{W_B\} \in \mathcal{A}_1(W_B)\).
\end{defn}

\begin{ex}[Example~\ref{ex1} continued]
The connected bipartite graph shown in Figure~\ref{subfig1} is an admissible graph, with blocks given by the subgraphs \(B_1 = G(W_1)\) and \(B_2 = G(W_2)\), displayed in Figure~\ref{subfig2}. There is only one admissible decomposition of \(W_1\), namely \(\{W_1\} \in \mathcal{A}_1(W_1)\). For instance, the decomposition of \(W_1\) into the subsets \(W_5\) and \(W_6\), illustrated in Figure~\ref{subfig4}, is not admissible. Indeed, condition (e) is violated: in the induced subgraphs \(G(W_5)\) and \(G(W_6)\), the vertex \(v_2\) has degree \(1\) in each subgraph, whereas condition (e) requires that \(\deg_{G_i}(v_2) \ge 2\) for at least one subgraph \(G_i = G(W_i)\). On the other hand, there are two admissible decompositions of \(W_2\): \(\{W_2\} \in \mathcal{A}_1(W_2)\) and \(\{W_3, W_4\} \in \mathcal{A}_2(W_2)\), where the subsets \(W_3\) and \(W_4\) are illustrated in Figure~\ref{subfig3}.
\end{ex}

\begin{ex} \label{ex2}
The connected bipartite graph shown in Figure~\ref{fig2} is an admissible graph with three blocks: \(B_1 = G(W_{B_1})\), \(B_2 = G(W_{B_2})\), and \(B_3 = G(W_{B_3})\). These blocks are defined by the subsets \(W_{B_1} = \{w_1\}, W_{B_2} = \{w_2, w_3, w_4, w_5, w_6\}\), and \(W_{B_3} = \{w_7\}\). There are two admissible decompositions of \(W_{B_2}\): \(\{W_{B_2}\} \in \mathcal{A}_1(W_{B_2})\) and \(\{W_1, W_2\}  \in \mathcal{A}_2(W_{B_2})\), where \(W_1 = \{w_2, w_3, w_4, w_5\}\) and \(W_2 = \{w_3, w_6\}\).
\end{ex}

\begin{figure}
\begin{tikzpicture}[scale=1.2]

\draw (0,0) circle [radius=1];
\def\nodesA{8}

\foreach \i in {1, ..., \nodesA} {
        
\pgfmathsetmacro\angle{360/\nodesA * \i}
\pgfmathsetmacro\xpos{0 + cos(\angle)}
\pgfmathsetmacro\ypos{0 + sin(\angle)}
     \coordinate (P\i) at (\xpos, \ypos); 
    
  \ifodd\i
           \fill[black] (\xpos, \ypos) circle [radius=2pt];  
        \else
           \draw[fill=white]  (\xpos, \ypos) circle [radius=2pt];  
        \fi
    }

\node[above] at (P2) {\tiny $w_2$}  ;
\node[left] at (P4) {\tiny $w_5$}  ;
\node[below] at (P6) {\tiny $w_4$}  ;
\node[left] at (P8) {\tiny $w_3$}  ;
\node[above] at (P1) {\tiny $v_2$}  ;
\node[right] at (P3) {\tiny $v_1$}  ;
\node[below] at (P5) {\tiny $v_4$}  ;
\node[right] at (P7) {\tiny $v_3$}  ;

\pgfmathsetmacro\xangle{cos(135)} 
\pgfmathsetmacro\yangle{sin(135)}
 
\draw (-1.1,1.1)  to[bend right = 40] (\xangle, \yangle);    
\draw (-1.1,1.1)  to[bend left = 40] (\xangle, \yangle);    
\draw[fill=white]  (-1.1,1.1) circle [radius=2pt]; 
\node[above] at  (-1.1,1.1)  {\tiny $w_1$}  ;

\draw (1.5,0) circle [radius=0.5];
\def\nodesB{4}

\foreach \i in {1, ..., \nodesB} {
        
\pgfmathsetmacro\angle{360/\nodesB * \i}
\pgfmathsetmacro\xpos{1.5 + 0.5 * cos(\angle)}
\pgfmathsetmacro\ypos{0.5* sin(\angle)}
         \coordinate (Q\i) at (\xpos, \ypos); 
  \ifodd\i
           \fill[black] (\xpos, \ypos) circle [radius=2pt];  
        \else
           \draw[fill=white]  (\xpos, \ypos) circle [radius=2pt];  
        \fi
    }

\node[right] at (Q4) {\tiny $w_6$}  ;
\node[below] at (Q1) {\tiny $v_5$}  ;
\node[below] at (Q3) {\tiny $v_6$}  ;
       
\draw (Q1) to[bend right = 40] (1.5,1);    
\draw (Q1) to[bend left = 40] (1.5,1);    
\draw (1.85,1.35)  to[bend right = 40] (1.5,1);    
\draw (1.85,1.35) to[bend left = 40] (1.5,1); 
\draw (1.15,1.35)  to[bend right = 40] (1.5,1);    
\draw (1.15,1.35) to[bend left = 40] (1.5,1);     

\draw[fill=white]  (1.5,1) circle [radius=2pt]; 
\fill[black]  (1.85,1.35) circle [radius=2pt]; 
\fill[black]  (1.15,1.35) circle [radius=2pt]; 
\node[right] at (1.5,0.9)  {\tiny $w_7$}  ;
\node[left] at  (1.15,1.35)  {\tiny $v_7$}  ;
\node[right] at  (1.85,1.35)  {\tiny $v_8$}  ;

\end{tikzpicture}
\caption{An example of an admissible graph \(G = (W \cup V, E)\) with three blocks. Block \(B_1 = G(\{w_1\})\) is a simple cycle of length \(2\); block \(B_2 = G(\{w_2, w_3, w_4, w_5, w_6\})\) is a cactus graph consisting of two simple cycles connected at the vertex \(w_3\); and block \(B_3 = G(\{w_7\})\) is a double tree. The separating vertices are \(v_1\) and \(v_5\). In particular, \(G\) is a cactus graph. Block \(B_2\) admits two admissible decompositions: \(\{\{w_2, w_3, w_4, w_5, w_6\}\} \in \mathcal{A}_1(\{w_2, w_3, w_4, w_5, w_6\})\) and \(\{\{w_2, w_3, w_4, w_5\} , \{w_3, w_6\} \} \in \mathcal{A}_2(\{w_2, w_3, w_4, w_5, w_6\})\).}
\label{fig2}
\end{figure}

\subsection{The limiting injective trace}

We are now ready to present the asymptotics of the mean injective trace \(\tau^0_{p,m,n}[T]\) for any bipartite test graph \(T = (W \cup V, E, Y_m)\). To begin, we introduce two key parameters. 

\begin{defn} \label{def: C_deg (f)} 
For every even integer \(d\), we define the parameter \(C_d(f)\) by
\begin{equation} \label{eq: C_deg (f)} 
C_d(f) = \frac{1}{(2\pi)^d} \int_{\R^d} \prod_{i=1}^{d/2} \textnormal{d} \gamma_{i}^1 \textnormal{d} \gamma_{i}^2 \hat f(\gamma_{i}^1) \hat f(\gamma_{i}^2)  e^{\E_X \left [ \Phi \left ( \sum_{i=1 }^{d/2} (\gamma_i^1 + \gamma_i^2) X_i \right )  \right ]},
\end{equation}
where \(X_1, \ldots, X_{d/2}\) are i.i.d.\ random variables distributed according to \(\nu_x\). 
\end{defn}
\begin{defn}\label{def: C_W (f)} 
For every subsets \(W_1, \ldots, W_K \subseteq W\) with \(|W_k| \ge 2\) and such that the induced subgraphs \(G(W_k)\) are connected, we define the parameter \(C_{(W_k)_{k=1}^K}(f)\) by
\begin{equation} \label{eq: C_{W_i}} 
C_{(W_k)_{k=1}^K}(f) = \frac{1}{(2\pi)^{|\cup_k^K E_k|}} \int_{\R^{|\cup_{k=1}^K E_k|}} \prod_{e \in \cup_{k=1}^K E_k} \prod_{i=1}^{m(e)} \textnormal{d} \gamma_e^i \hat{f}(\gamma_e^i) e^{\sum_{w \in \cup_{k=1}^K W_k} \E_X \left [Z_w(\boldsymbol{\gamma})\right ]} \prod_{k=1}^K \E_X \left [ \prod_{w \in W_k} Z_w(\boldsymbol{\gamma}) \right ],
\end{equation}
where, for every \(w \in W\), the random variable \(Z_w(\boldsymbol{\gamma})\) is defined by
\begin{equation} \label{eq: Z_w}
Z_w(\boldsymbol{\gamma}) = \Phi \left (  \sum_{v \in V \colon v \sim w } (\gamma_{(w,v)}^1 + \cdots + \gamma_{(w,v)}^{m((w,v))}) X_v \right ), 
\end{equation}
with \((X_v)_{v \in V}\) being i.i.d.\ random variables distributed according to \(\nu_x\). Note that \(\E_X \left [ \prod_{w \in W_k} Z_w(\boldsymbol{\gamma}) \right ]\) does not factorize since \(G(W_k)\) is connected.
\end{defn}

We now present a more detailed version of Theorem~\ref{main2}, formulated as Proposition~\ref{main3}.

\begin{prop}[Theorem~\ref{main2} continued] \label{main3}
Under Assumptions~\ref{hyp1}-\ref{hyp3}, for every connected bipartite test graph \(T = (W \cup V, E , Y_m)\), 
\[
\lim_{p,m,n \to \infty} \tau^0_{p,m,n} \left [T \right ] = \tau_G^0,
\]
where \( \tau_G^0 \) depends on \(G = (W \cup V , E)\), \(\phi, \psi, f, \Phi\), and \(\nu_x\). In particular, the limiting injective trace \(\tau_G^0\) vanishes unless \(G\) is an admissible graph, in which case it is given by
\[
\tau_G^0 =
\frac{\phi^{\frac{|E|}{2} - |V|}}{\psi^{|W|-1}} \prod_{w \in W \backslash \cup_{i=1}^R W_{B_i}} C_{\deg(w)} (f)  \prod_{i=1}^R \left ( \sum_{K \geq 1} \sum_{\{W_1,\ldots, W_K\} \in \mathcal{A}_K(W_{B_i})} C_{(W_k)_{k=1}^K}(f)  \right )  ,
\]
where \(\{B_i = G(W_{B_i})\}_{i=1}^R\) are the blocks of \(G\) with \(|W_{B_i}| \ge 2\). In particular, when \(G\) is a double tree (i.e., \(R = 0\)), this reduces to
\[
\tau_G^0 = \frac{\phi^{\frac{|E|}{2} - |V|}}{\psi^{|W|-1}} \prod_{w \in W} C_{\deg(w)}(f).
\]
\end{prop} 

\begin{rmk}
A \emph{fat tree} is a graph that becomes a tree when the multiplicity of the edges is forgotten. In particular, a double tree is a fat tree in which every edge has multiplicity two. From Proposition~\ref{main3}, if \(G\) is a fat tree containing edges of multiplicity greater than two, the limiting injective trace \( \tau_G^0 \) vanishes. This contrasts with the behavior of general heavy-tailed random matrices, for which fat trees with edges of even multiplicity yield a nonvanishing contribution (see~\cite{male2017heavy}).
\end{rmk}

From Proposition~\ref{main3}, we observe that the limiting injective trace \(\tau_G^0\) generally depends on the distributions \(\nu_w\) and \(\nu_x\) through the function \(\Phi\) and higher-order moments of \(\nu_x\), and is therefore not universal, even when $\nu_x$ is light-tailed. We now focus more closely on the two coefficients appearing in Proposition~\ref{main3} in the symmetric $\alpha$-stable case. Consider first the coefficient \(C_d(f)\) for even integers \(d\). Using elementary properties of the Fourier transform,~\eqref{eq: C_deg (f)} can be rewritten as
\begin{equation} \label{eq: C_deg} 
C_d (f) = \frac{1}{(2\pi)^{d/2}} \int_{\R^{d/2}} \prod_{i=1}^{d/2} \textnormal{d} x_i f^2(x_i) \hat{\varphi}(\boldsymbol{x}),
\end{equation}
where \(\varphi(\boldsymbol{t}) = e^{\E_X \left [ \Phi  \left ( \sum_{i=1 }^{d/2} t_i X_i \right )\right ]}\). If the weights $W_{ij}$ are symmetric \(\alpha\)-stable random variables (i.e., \(\Phi(\lambda) = - \sigma^\alpha |\lambda|^\alpha\)) and \(d=2\), then
\[
\varphi(t) = e^{- \sigma^\alpha |t|^\alpha \E_X \left [ |X|^\alpha \right ]}
\]
which corresponds to the characteristic function of \(S \sim S_\alpha(\sigma \E_X \left [ |X|^\alpha \right ]^{1/\alpha})\) with \(X \sim \nu_x\). Hence, if $f_S$ denotes the density of $S$, we obtain
\[
C_2(f) = \int_{\R} f^2(x) f_S(x) \textnormal{d} x = \E_S \left [ f^2(S)\right ].
\]
For \(d > 2\) and Gaussian inputs \(X_{ij} \sim \mathcal{N}(0,\sigma_x^2)\), we have
\[
\varphi(\boldsymbol{t}) = e^{- \sigma^\alpha \E_X \left [ \left |\sum_{i=1 }^{d/2} t_i X_i \right |^\alpha \right ]} =  e^{- \sigma^\alpha \left ( \sum_{i=1}^{d/2} t_i^2\right)^{\alpha/2} \E_X \left [ |X|^\alpha \right ] } = e^{- \sigma^\alpha  \E_X \left [ |X|^\alpha \right ]  |\boldsymbol{t}|^\alpha},
\]
where \(X \sim \mathcal{N}(0,\sigma_x^2)\). This corresponds to the joint characteristic function of the isotropic multivariate stable distribution, see e.g.~\cite{Nolan}. If \(\boldsymbol{S} = (S_1, \ldots, S_{d/2}) \in \R^{d/2}\) is a random vector having the isotropic multivariate stable distribution and \( f_{\boldsymbol{S}}(\boldsymbol{x})\) denotes its joint probability density function, then
\[
C_d (f) = \int_{\R^{d/2}} \prod_{i=1}^{d/2} \textnormal{d} x_i f^2(x_i) f_{\boldsymbol{S}}(\boldsymbol{x}) .
\]

The Gaussian case $\alpha=2$ leads to a particularly tractable situation.

\begin{rmk} \label{rmk: C_deg(w) alpha=2}
Assume \(W\) has i.i.d.\ centered Gaussian entries with variance \(\sigma_w^2\), i.e., \(\alpha = 2\) and \(\Phi (\lambda) = - \sigma_w^2 \lambda^2 /2\). Suppose \(X\) has i.i.d.\ centered entries with variance \(\sigma_x^2\) as in Assumption~\ref{hyp1}(b). From the previous computation~\eqref{eq: C_deg}, the Fourier transform of the Gaussian function \(\varphi(\boldsymbol{t}) = e^{- \sigma_w^2 \sigma_x^2 |\boldsymbol{t}|^2 / 2}\) is given by 
\[
\hat{\varphi}(\boldsymbol{x}) = \left ( \frac{\sqrt{2 \pi}}{\sigma_w \sigma_x} \right )^{d/2} e^{- |\boldsymbol{x}|^2 / (2 \sigma_w^2 \sigma_x^2)}, 
\]
leading to
\[
C_d (f)  = \left (  \frac{1}{\sqrt{2 \pi \sigma_w^2 \sigma_x^2} }  \int_\R f^2(x) e^{- x^2 / (2 \sigma_w^2 \sigma_x^2)} \textnormal{d}x \right)^{d/2} = \left ( \E_{Z \sim \mathcal{N}(0, \sigma_w^2 \sigma_x^2)} \left [ f^2(Z) \right ] \right )^{d/2}.
\]
Consequently, if \(G\) is a double tree, Proposition~\ref{main3} implies that
\[
\tau_G^0    = \frac{\phi^{\frac{|E|}{2}-|V|}}{\psi^{|W|-1}} \theta_1(f)^{\sum_{w \in W} \deg(w)/2} =\frac{\phi^{\frac{|E|}{2}-|V|}}{\psi^{|W|-1}}\theta_1(f)^{|W|},
\]
where \(\theta_1 (f) \coloneqq \E_{Z \sim \mathcal{N}(0, \sigma_w^2 \sigma_x^2)} \left [ f^2(Z) \right ] \), so \(\tau_G^0\) is universal with respect to \(\nu_x\) and \(\nu_w\) (it depends on them only through \(\sigma_w^2\) and \(\sigma_x^2\)).
\end{rmk} 

We now turn to the coefficient \(C_{(W_k)_{k=1}^K}(f)\) appearing in Proposition~\ref{main3}. When \(\Phi(\lambda) = - \sigma^\alpha |\lambda|^\alpha\) and \(X_v \stackrel{\mathrm{i.i.d.}}{\sim} \mathcal{N}(0,\sigma_x^2)\), for any \(w \in \cup_{k=1}^KW_k\) we have
\[
\E_X \left [ Z_w(\boldsymbol{\gamma}) \right] = - \sigma^\alpha \E_X \left [ \left |  \sum_{v \colon v\sim w} \left ( \sum_{i=1}^{m((w,v))} \gamma_{(w,v)}^i\right) X_v \right|^\alpha\right] = -\sigma^\alpha \E_X \left [ |X|^\alpha \right] \left ( \sum_{v \sim w} \left ( \sum_{i=1}^{m((w,v))} \gamma_{(w,v)}^i \right)^2 \right)^{\alpha/2},
\]
where \(X \sim \mathcal{N}(0,\sigma_x^2)\). Therefore, the exponential factor in~\eqref{eq: C_{W_i}} depends on \(\nu_x\) through \( \E_X \left [ |X|^\alpha \right ] \). In contrast, the product term
\[
\E_X \left [ \prod_{w \in W_k} Z_w (\boldsymbol{\gamma}) \right ] = (-\sigma^\alpha)^{|W_k|}  \E_X \left [\prod_{w \in W_k} \left | \sum_{v \colon v \sim w} \left (\sum_{i=1}^{m((w,v))} \gamma_{(w,v)}^i \right) X_v \right|^\alpha \right ]
\]
depends on \(\nu_x\) through absolute moments of order \(\alpha |W_k|\). For \(\alpha=2\) and Gaussian inputs, Wick’s formula implies that this expectation depends only on the second moment, and hence the parameter \(C_{(W_k)_{k=1}^K}(f)\) depends on \(\nu_w\) and \(\nu_x\) only through \(\sigma_w^2\) and \(\sigma_x^2\). For \(\alpha=2\) and light-tailed inputs, this expectation can still exhibit universality for specific graph structures such as the simple bipartite cycle. For general graphs \(G\), however, high-order moments of \(\nu_x\) typically contribute.

\subsection{Proof of convergence of the injective trace} \label{subsection: proof main prop}

In this section, we prove Proposition~\ref{main3}. We consider a connected bipartite test graph \(T = (W \cup V, E, Y_m)\), where \(E\) is a multiset of edges, and set \(G = (W \cup V, E)\) for the underlying bipartite multigraph. We begin by showing that if \(T\) contains a vertex in either \(W\) or \(V\) with an odd degree, then the mean injective trace \(\tau^0_{p,m,n} \left [T \right]\) vanishes.

\begin{lem} \label{lem: even degree} 
Let \(T = (W \cup V, E, Y_m)\) be a finite, connected bipartite test graph with at least one vertex in either \(W\) or \(V\) having odd degree. Then, \(\tau^0_{p,m,n} \left [T \right]=0\) for every \(p,m,n \in \N\).
\end{lem}
\begin{proof} 
We first assume that there exists \(v_0 \in V\) having an odd degree. For every injective maps \(\phi_W \colon W \to [p]\) and \(\phi_V \colon V \to [m]\), according to~\eqref{eq: tau0} we have 
\[
\begin{split}
& \E \left[ \prod_{(w,v) \in E} Y_m(\phi_W(w),\phi_V(v)) \right] \\
& = \E \left[ \E \left [\prod_{(w,v_0) \in E} Y_m (\phi_W(w), \phi_V(v_0)) \, | \, \mathcal F_{\phi_V(v_0)} \right ] \prod_{(w,v)\in E, v \in V \backslash \{v_0\}} Y_m(\phi_W(w),\phi_V(v))\right]
\end{split}
\]
where \(\mathcal F_{\phi_V(v_0)}\) denotes the \(\sigma\)-algebra generated by \(\{(W_i)_{i \in [p]}, (X_j)_{j \neq \phi_V(v_0)}\}\). Now, we recall from~\eqref{eq: matrix Y} that \(Y_m(\phi_W(w),\phi_V(v))\) is given by
\[
Y_m (\phi_W(w),\phi_V(v)) = \frac{1}{\sqrt{m}} f \left ( W_{\phi_W(w)} \cdot X_{\phi_V(v)} \right) =  - \frac{1}{\sqrt{m}} f \left (  W_{\phi_W(w)} \cdot (-X_{\phi_V(v)})  \right),
\]
where we used the fact that \(f\) is an odd function. Since the law of the random vector \(X_{\phi_V(v)}\) is symmetric by Assumption~\ref{hyp1}, we obtain 
\[
\begin{split}
& \E \left [\prod_{(w,v_0) \in E} Y_m (\phi_W(w), \phi_V(v_0)) \, | \, \mathcal F_{\phi_V(v_0)} \right ] \\
&= (-1)^{\deg(v_0)} \E \left [\prod_{(w,v_0) \in E} \frac{1}{\sqrt{m}} f \left (  W_{\phi_W(w)} \cdot (-X_{\phi_V(v)})  \right)  \, | \, \mathcal F_{\phi_V(v_0)} \right ] \\
&=- \E \left [\prod_{(w,v_0) \in E} Y_m (\phi_W(w), \phi_V(v_0)) \, | \, \mathcal F_{\phi_V(v_0)} \right ],
\end{split}
\]
where we used that \(\deg(v_0)\) is an odd integer. This shows that \(\E \left [\prod_{(w,v_0) \in E} Y_m (\phi_W(w), \phi_V(v_0)) \, | \, \mathcal F_{\phi_V(v_0)}\right ]\) vanishes and so does \(\tau^0_{p,m,n} \left [T \right]\). The same argument applies to any vertex \(w_0 \in W\) with an odd degree, as the law of the entries of \(W\) is also symmetric by Assumption~\ref{hyp1}.
\end{proof}

By Lemma~\ref{lem: even degree}, we may assume that all vertices in \(W \cup V\) have even degree. To compute the mean injective trace \(\tau_{p,m,n}^0 \left[T \right]\), we start from~\eqref{eq: tau0} and expand it as
\begin{equation} \label{eq: tau0 2}
\begin{split}
\tau^0_{p,m,n} \left[T \right] & =   \frac{1}{p} \sum_{\phi_W \colon W \to [p] \atop \phi_W \textnormal{injective}}  \sum_{\phi_V \colon  V \to [m] \atop \phi_V \textnormal{injective}} 
\E \left [\prod_{e = (w,v)\in E} \left (Y_m (\phi_W(w),\phi_V(v)) \right )^{m(e)} \right] \\
& =  \frac{m^{|V| }p^{ |W|}}{p m^{|E|/2}} (1 + \mathcal{O}(m^{-1})) (1 + \mathcal{O}(p^{-1})) \, \E \left [\prod_{e=(w,v) \in E} \left ( f \left ( W_{\phi_W(w)} \cdot  X_{\phi_V(v)}\right )\right )^{m(e)} \right ].
\end{split}
\end{equation}
Since \(f \in C^\infty (\R) \cap L^1(\R)\) and \(f^{(k)} \in L^1(\R)\) for all \(k \ge 1\) by Assumption~\ref{hyp2}, repeated integration by parts implies that the Fourier transform
\[
\hat{f}(t) = \int_\R f(x) e^{-itx} \mathrm{d}x
\]
is continuous and for every integer number \(\ell\), there exists a finite constant \(C_\ell >0\) such that for every \(t \in \R\),
\begin{equation}\label{thehypf}
|\hat f(t) |\le \frac{C_\ell}{(1+|t|)^\ell}.
\end{equation}
In particular, \(\hat{f} \in L^1(\R)\) and decays faster than any polynomial. Therefore, the Fourier inversion theorem applies, yielding
\begin{equation} \label{eq: invFourier}
f(x)=\frac{1}{2\pi}\int_{\R} \hat{f}(t) e^{itx} \textnormal{d} t.
\end{equation}
Furthermore, for every real number $a \ge 0$ and every integer number $\ell$,
\begin{equation}\label{thehypf2}
\int_{[-a,a]^{c}}|\hat f(t) | \mathrm{d} t\le C_\ell\int_{[-a,a]^{c}}\frac{1}{(1+|t|)^\ell} \mathrm{d} t \le C_\ell \frac{1}{1+ a^{\ell-2}}\int  \frac{1}{(1+|t|)^2} \mathrm{d} t \eqqcolon \frac{C'_{\ell}}{ 1+ a^{\ell-2}},
\end{equation}
where \(C_\ell'>0\) is a finite constant independent of \(a\). Finally, we note that, since \(f\) is odd by Assumption~\ref{hyp2}, its Fourier transform \(\hat{f}\) is also odd. 

Combining~\eqref{eq: tau0 2} and~\eqref{eq: invFourier} and using Assumption~\ref{hyp3}, we obtain the following expression for the mean injective trace: 
\begin{equation} \label{eq: traffic development}
\tau^0_{p,m,n} \left [T \right ] = n^{\rho(G)}  (1 + \mathcal{O}(n^{-1}))  \frac{\phi^{\frac{|E|}{2} - |V|}}{\psi^{|W|-1}} \frac{1}{(2\pi)^{|E|}}\int_{\R^{|E|}} \prod_{e \in E} \prod_{i=1}^{m(e)} \textnormal{d} \gamma_e^i \hat f(\gamma_e^i) \Lambda_G^n (\boldsymbol{\gamma}),
\end{equation}
where 
\begin{equation} \label{eq: rho(G)}
\rho(G) \coloneqq |W| + |V| - \frac{|E|}{2} -1, 
\end{equation}
and for \(\boldsymbol{\gamma} = (\gamma_e^1,\ldots, \gamma_e^{m(e)})_{e \in E}\), 
\[
\begin{split}
\Lambda_G^n(\boldsymbol{\gamma}) & = \E_{W,X} \left[\exp \left ( i \sum_{e = (w,v) \in E} (\gamma_e^1 + \cdots + \gamma_e^{m(e)}) \,W_w \cdot X_v \right ) \right]  \\
& = \E_{W,X} \left[\exp \left ( i \sum_{w \in W} \sum_{k=1}^n W_{wk} \left ( \sum_{v \in V \colon v \sim w } ( \gamma_{(w,v)}^1 + \cdots + \gamma_{(w,v)}^{m((w,v))}) X_{kv} \right ) \right ) \right] .
\end{split}
\]
The main challenge in the proof of Proposition~\ref{main3} lies in estimating \(\Lambda_G^n(\boldsymbol{\gamma})\). To address this, we take the expectation with respect to \(W\), yielding
\begin{equation} \label{eq: lambda}
\Lambda_G^n(\boldsymbol{\gamma}) = \left (\E_X \left [ e^{ n^{-1}S_G^n} \right ] \right)^n,
\end{equation}
where 
\[
\begin{split}
S_G^n (\boldsymbol{\gamma}) & \coloneqq \sum_{w \in W} Z_w^n(\boldsymbol{\gamma}),\\
Z_w^n(\boldsymbol{\gamma}) &\coloneqq n \log  \E_W \left [ e^{i W_w \left( \sum_{v \in V \colon v \sim w} (\gamma_{(w,v)}^1 + \cdots + \gamma_{(w,v)}^{m((w,v))}) X_v\right)  }\right ] .
\end{split}
\]
Here, \((X_v)_{v \in V}\) are i.i.d.\ with distribution \(\nu_x\), and \((W_w)_{w \in W}\) are i.i.d.\ with distribution \(\nu_w\). From Assumption~\ref{hyp1}, it follows that \(S_G^n (\boldsymbol{\gamma})\) converges to \(S_G  (\boldsymbol{\gamma}) = \sum_{w \in W} Z_w (\boldsymbol{\gamma})\) as \(n \to \infty\), where 
\[
Z_w  (\boldsymbol{\gamma}) = \Phi \left (  \sum_{v \in V \colon v \sim w} (\gamma_{(w,v)}^1 + \cdots + \gamma_{(w,v)}^{m((w,v))}) X_v \right ).
\]

To analyze \(\Lambda_G^n(\boldsymbol{\gamma})\) further, we now proceed to expand the right-hand side of~\eqref{eq: lambda} by means of a cumulant expansion. Without loss of generality, we may assume that the \(\gamma_e^i\)'s are bounded in absolute value by \(n^{\epsilon}\) for some \(\epsilon>0\). Indeed, for the integral in~\eqref{eq: traffic development}, define
\begin{equation} \label{largegamma}
R_n^{\epsilon}= \int (\mathbf{1}_{\R^{|E|}} - \mathbf{1}_{[-n^{\epsilon},n^{\epsilon}]^{|E|}}) \prod_{e \in E} \textnormal{d} \gamma_e^1 \cdots \textnormal{d} \gamma_e^{m(e)}\hat f(\gamma_e^1) \cdots \hat f(\gamma_e^{m(e)})  \Lambda_G^n(\boldsymbol{\gamma}),
\end{equation} 
which can be bounded easily since \(S_G^n(\boldsymbol{\gamma})\) has a nonpositive real part (i.e., \(\Re (S_G^n(\boldsymbol{\gamma})) \le 0\)), implying \(|\Lambda_G^n(\boldsymbol{\gamma})| \le 1\). Moreover, by~\eqref{thehypf} and~\eqref{thehypf2}, for every integer \(\ell \ge 2\), there exist constants \(C_2', C_\ell'>0\) such that
\[
\begin{split}
|R_n^\epsilon | & \leq \sum_{k=1}^{|E|} {|E| \choose k} \left(\int_{[-n^{\epsilon},n^{\epsilon}]^\textnormal{c}} |\hat{f}(t)| \textnormal{d} t\right)^k \left(\int_\R  |\hat{f}(t)| \textnormal{d} t \right)^{|E|-k} \\
&\le \sum_{k=1}^{|E|} {|E| \choose k} (C_{\ell}')^k (1 + n^{\epsilon (\ell-2)})^{-k} (C_2')^{|E|-k}  \\
&= (C_2')^{|E|}\left( \left(1+\frac{C_{\ell}'}{C_2' (1 + n^{\epsilon (\ell-2)})} \right)^{|E|} -1\right).
\end{split}
\]
This quantity can be made arbitrarily small---indeed negligible compared with \(n^{-\rho(G)}\)---by choosing \(\ell\) and \(n\) sufficiently large. Hence, the asymptotic behavior of the integral in~\eqref{eq: traffic development} is entirely determined by the region \(|\gamma_e^i| \le n^\epsilon\) where the argument of the exponential in \(\Lambda_G^n(\boldsymbol{\gamma})\) is small. In this region the cumulant expansion of \(\Lambda_G^n(\boldsymbol{\gamma})\) is valid, since \(\Lambda_G^n(\boldsymbol{\gamma})^{1/n} = \E_X \left [\exp \left ( n^{-1} S_G^n(\boldsymbol{\gamma}) \right ) \right ]\) is the Fourier transform of random variables with finite exponential moments, evaluated in a neighborhood of zero.

The \(\ell\)th cumulant \(\kappa_\ell^G\) of \(S_G^n(\boldsymbol{\gamma})\) is defined by
\begin{equation} \label{eq: cumulant}
\kappa_\ell^G (\boldsymbol{\gamma}) = \sum_{\pi \in \mathcal{P} ([\ell])} \mu(\pi) \prod_{i=1}^{|\pi|} \E_X \left [ \left (S_G^n(\boldsymbol{\gamma})\right)^{|B_i|} \right ],
\end{equation}
where the sum runs over the set of partitions \(\mathcal{P}([\ell])\) of the set \([\ell] =\{1, \ldots, \ell\}\) (see Definition~\ref{def: partition}), and the product runs over the blocks \(B_1, \ldots, B_{|\pi|}\) of the partition \(\pi\). Here, \(\mu(\pi) = (-1)^{|\pi|-1} (|\pi| -1)!\), \(|B_i|\) denotes the cardinality of block \(B_i\), and \(|\pi|\) the number of blocks in \(\pi\). Since the cumulants \(\kappa_\ell^G(\boldsymbol{\gamma})\) are the coefficients in the power series expansion of the cumulant generating function, we have for all \(t \in \R\),
\[
\E_X \left [ e^{t S_G^n(\boldsymbol{\gamma})} \right ] = \exp \left ( \sum_{i = 1}^\infty \kappa_i^G (\boldsymbol{\gamma}) \frac{t^i}{i!} \right ),
\]
so that from~\eqref{eq: lambda} it follows that
\begin{equation} \label{eq: Lambda}
\begin{split}
\Lambda_G^n(\boldsymbol{\gamma}) &  = e^{\kappa_1^G (\boldsymbol{\gamma})} \exp \left \{ \sum_{i \geq 2} \frac{1}{i ! n^{i -1 }} \kappa_i^G(\boldsymbol{\gamma}) \right \}= e^{\E_X \left [ S_G^n(\boldsymbol{\gamma}) \right]} \sum_{m_2, m_3, \ldots \geq 0} \prod_{\ell \ge 2} \frac{1}{m_\ell!} \left ( \frac{1}{\ell ! n^{\ell -1 }} \kappa_\ell^G(\boldsymbol{\gamma}) \right ) ^{m_\ell}.
\end{split}
\end{equation}
For \(|\gamma_e^i| \le n^\epsilon\), the sum \(S_G^n (\boldsymbol{\gamma})\) has finite exponential moments uniformly in \(\boldsymbol{\gamma}\), and \(n^{-1} S_G^n (\boldsymbol{\gamma})\) is of order \(n^{\epsilon-1}\). Hence the Taylor expansion of the logarithm in~\eqref{eq: Lambda} is valid, and the remainder after including cumulants up to order \(i\) is \(\mathcal{O} ((n^{\epsilon-1})^i)\). Such terms are negligible whenever \(i(1-\epsilon)>\rho(G)\). The second expansion in~\eqref{eq: Lambda}, obtained by expanding the exponential function, may likewise be truncated, meaning that we keep only those terms with total order \(\sum_{\ell \geq 2} m_{\ell}(\ell-1)\le \rho(G)\). Terms beyond this bound contribute at order \(\mathcal{O}(n^{-\rho(G)})\) or smaller and can thus be ignored. In the sequel, we keep the series formally infinite, with the understanding that it may be truncated at this finite order without affecting the asymptotic expansion up to \(n^{-\rho(G)}\).

We now present our expansion for \(\Lambda_G^n (\boldsymbol{\gamma})\). In the following lemma, for any subset \(E_0 \subseteq E\), let \(T_{E_0}\) denote the map on \(\R^{|E|}\) that changes \(\gamma_e^i\) into \(-\gamma_e^i\) for every \(e\in E_0\) and every \(i\in \{1,\ldots,m(e)\}\), while leaving the other entries unchanged. 

\begin{lem} \label{lem: main estimate}
Let \(G = (W \cup V, E)\) be a connected bipartite graph in which all vertices have even degree. Then, for every \(\epsilon>0\), 
\[
\Lambda_G^n(\boldsymbol{\gamma}) = e^{\E_X \left [ S_G^n(\boldsymbol{\gamma})\right ]}  \left ( 1 + h_G^n(\boldsymbol{\gamma}) + \frac{1}{n^{\rho(G)}} g_G^n(\boldsymbol{\gamma})  \mathbf{1}_{\{G \, \textnormal{is admissible with} \, R \ge 1\}} + o \left (\frac{1}{n^{\rho(G)}} \right)  \right),
\]
where \(o \left (\frac{1}{n^{\rho(G)}} \right)\) tends to zero uniformly on \(\{\|\gamma\|_{\infty}\le n^{\epsilon}\}\), \(h_G^n(\boldsymbol{\gamma})\) is a (finite) sum of functions which are invariant under \(T_{E_0}\) for some subset \(E_0 \subseteq E\) with odd cardinality \(|E_0|\), and \(g_G^n(\boldsymbol{\gamma}) \) is given by
\begin{equation} \label{eq: g_G^n}
g_G^n(\boldsymbol{\gamma}) = \prod_{i=1}^R \left ( \sum_{K \geq 1} \sum_{\{W_1,\ldots, W_K\} \in \mathcal{A}_K (W_{B_i})} \prod_{k =1}^K \E_X \left [ \prod_{w \in W_k} Z_w^n(\boldsymbol{\gamma})\right ] \right).
\end{equation}
\end{lem}

Having Lemma~\ref{lem: main estimate} at hand, we now prove Proposition~\ref{main3}. 

\begin{proof}[\textbf{Proof of Proposition~\ref{main3}}]
From~\eqref{eq: traffic development} and~\eqref{largegamma}, it follows that
\[
\tau^0_{p,m,n} \left [ T \right ]= \frac{ \phi^{\frac{|E|}{2} - |V|}}{\psi^{|W|-1}} \frac{1}{(2\pi)^{|E|}}\left(n^{\rho(G)} \int_{[-n^{\epsilon},n^{\epsilon}]^{|E|}} \prod_{e \in E} \prod_{i=1}^{m(e)} \textnormal{d} \gamma_e^i \hat f(\gamma_e^i) \Lambda_G^n(\boldsymbol{\gamma}) + o(1)\right).
\]
Plugging the expansion of \(\Lambda_G^n\) from Lemma~\ref{lem: main estimate} into this expression yields
\[
\tau^0_{p,m,n} \left [ T \right ] =  I^n_1 (G) +  I^n_2 (G) + I^n_3(G) + o(1),
\]
where
\[
\begin{split}
I^n_1 (G) & = \frac{ \phi^{\frac{|E|}{2} - |V|}}{\psi^{|W|-1}} \frac{1}{(2\pi)^{|E|}} n^{\rho(G)} \int_{[-n^{\epsilon},n^{\epsilon}]^{|E|}} \prod_{e \in E} \prod_{i=1}^{m(e)} \textnormal{d} \gamma_e^i \hat f(\gamma_e^i)  e^{\E_X \left [ S_G^n(\boldsymbol{\gamma})\right ]},  \\
I^n_2 (G) &= \frac{ \phi^{\frac{|E|}{2} - |V|}}{\psi^{|W|-1}} \frac{1}{(2\pi)^{|E|}} n^{\rho(G)} \int_{[-n^{\epsilon},n^{\epsilon}]^{|E|}} \prod_{e \in E} \prod_{i=1}^{m(e)} \textnormal{d} \gamma_e^i \hat f(\gamma_e^i)  e^{\E_X \left [ S_G^n(\boldsymbol{\gamma})\right ]} h_G^n (\boldsymbol{\gamma}),\\
I^n_3 (G) &= \frac{ \phi^{\frac{|E|}{2} - |V|}}{\psi^{|W|-1}} \frac{1}{(2\pi)^{|E|}} \int_{[-n^{\epsilon},n^{\epsilon}]^{|E|}} \prod_{e \in E} \prod_{i=1}^{m(e)} \textnormal{d} \gamma_e^i \hat f(\gamma_e^i)  e^{\E_X \left [ S_G^n(\boldsymbol{\gamma})\right ]} g_G^n (\boldsymbol{\gamma})   \mathbf{1}_{\{G \, \textnormal{is admissible with} \, R \ge 1\}} .
\end{split}
\]
The integrand is uniformly bounded since
\[
\left|\int_{[-n^{\epsilon},n^{\epsilon}]^{|E|}} \prod_{e \in E} \prod_{i=1}^{m(e)} \textnormal{d} \gamma_e^i \hat f(\gamma_e^i) e^{\E_X \left [ S_G^n (\boldsymbol{\gamma}) \right ]} \right| \le \left(\int_{[-n^{\epsilon},n^{\epsilon}]}|\hat f(\gamma_e)| \textnormal{d} \gamma_{e}\right)^{|E|} < \infty,
\]
where we used the fact that \(|e^{\E_X \left [ S_G^n (\boldsymbol{\gamma})\right ]}| \le 1\) and \(\hat{f} \in L^1(\R)\). Hence, the error terms in \(\Lambda_G^n(\boldsymbol{\gamma})\) become error terms in \(\tau^0_{p,m,n} \left [T \right ]\). 

We first claim that \(I^n_2 (G)\) vanishes for all integers \(n\). Since the law of \(X_v\) is symmetric, \(\E_X \left [S_G^n(\boldsymbol{\gamma}) \right ]\) is invariant under any sign change \(T_e\) of the form \(\gamma^i_e \mapsto -\gamma^i_e\) for every \(e \in E\) and \(1 \le i \le m(e)\), i.e.,
\[
\E_X \left [ S_G^n(\boldsymbol{\gamma}) \right ]=\E_X \left [S_G^n(T_e \boldsymbol{\gamma}) \right ].
\]
Each summand \(h(\boldsymbol{\gamma})\) of \(h_G^n (\boldsymbol{\gamma})\) is invariant under a transformation \(T_{E_0}\) for some subset \(E_0\subseteq E\) of odd cardinality. Applying this change of variables and using the fact that \(\hat{f}\) is odd by Assumption~\ref{hyp2} yields
\[
\begin{split}
& \int_{[-n^{\epsilon},n^{\epsilon}]^{|E|}} \prod_{e \in E} \prod_{i=1}^{m(e)} \textnormal{d} \gamma_e^i \hat f(\gamma_e^i) e^{\E_X \left [ S_G^n (\boldsymbol{\gamma}) \right ]}  h(\boldsymbol{\gamma})\\
& \quad = (-1)^{|E_{0}|}\int_{[-n^{\epsilon},n^{\epsilon}]^{|E|}} \prod_{e \in E}  \prod_{i=1}^{m(e)} \textnormal{d} \gamma_e^i \hat f(\gamma_e^i) e^{ \E_X \left [ S_G^n (\boldsymbol{\gamma})\right ]}  h(\boldsymbol{\gamma}).
\end{split}
\]
Because \(|E_0|\) is odd, the integral equals its negative and therefore vanishes. The same reasoning applies to every summand of \(h_G^n\), hence \(I^n_2 (G) = 0\).

We now analyze \(I^n_1 (G)\). We first note that if there exists an edge \(e_0 \in E\) with odd multiplicity \(m(e_0)\), then the integral 
\[
\int_{[-n^\epsilon, n^\epsilon]^{|E|}} \prod_{e \in E} \prod_{i=1}^{m(e)} \textnormal{d} \gamma_e^i \hat f(\gamma_e^i) e^{\E_X \left [ S_G^n (\boldsymbol{\gamma}) \right ]}
\]
vanishes for all \(n\), by the same sign-change argument as above using the transformation \(T_{e_0}\). Therefore, \(I^n_1(G)\) is nonvanishing only if every edge \(e \in E\) has even multiplicity. In this case, let \(\widetilde{G} = (W \cup V, \widetilde{E})\) denote the simple graph obtained from \(G\) by forgetting edge multiplicities. Since \(G\) is connected, so is \(\widetilde{G}\), and it satisfies 
\[
|W| + |V| \leq |\widetilde{E}| + 1, 
\]
with equality if and only if \(\widetilde{G}\) is a tree. Because every edge in \(G\) has even multiplicity, we have \(|E| \geq 2 |\widetilde{E}|\), which implies 
\[
\rho(G) = |W| + |V| - |E|/2 - 1 \leq 0.
\]
Equality in this bound holds precisely when every edge in \(G\) has multiplicity \(2\) and \(\widetilde{G}\) is a tree, that is, when \(G\) is a double tree. If \(\rho(G)<0\), the integral \(I_1^n(G)\) tends to zero as \(n \to \infty\), since the integral is uniformly bounded. If \(\rho(G)=0\) (i.e., \(G\) is a double tree), the tail outside \([-n^{\epsilon},n^{\epsilon}]^{|E|}\) contributed only \(o(1)\) by~\eqref{largegamma}. Moreover, since \(|e^{\E_X \left [ S_G^n (\boldsymbol{\gamma}) \right ]}| \le 1\), \(\prod_{e,i} |\hat{f}(\gamma_e^i)| \in L^1 (\R^{|E|})\), and \(S_G^n (\boldsymbol{\gamma}) \to S_G(\boldsymbol{\gamma})\) by Assumption~\ref{hyp1}, the dominated convergence theorem yields
\[
I^n_1(G) =
\frac{ \phi^{\frac{|E|}{2} - |V|}}{\psi^{|W|-1}} \frac{1}{(2\pi)^{|E|}} \int_{\R^{|E|}} \prod_{e \in E} \prod_{i=1}^2 \textnormal{d} \gamma_e^i \hat f(\gamma_e^i)  e^{ \sum_{w \in W} \E_X \left [ Z_w(\boldsymbol{\gamma}) \right ]} + o(1), 
\]
where \(Z_w (\boldsymbol{\gamma}) = \Phi \left ( \sum_{v \in V \colon v \sim w} \left ( \gamma_{(w,v)}^1 +  \gamma_{(w,v)}^2 \right) X_v \right )\).

It remains to estimate \(I^n_3(G)\), which vanishes unless \(G\) is an admissible graph with \(R \ge 1\) blocks, each containing more than one vertex of \(W\). In this case, 
\[
I^n_3(G) = \frac{ \phi^{\frac{|E|}{2} - |V|}}{\psi^{|W|-1}} \frac{1}{(2\pi)^{|E|}} \int_{[-n^\epsilon, n^\epsilon]^{|E|}} \prod_{e \in E}  \prod_{i=1}^{m(e)} \textnormal{d} \gamma_e^i \hat f(\gamma_e^i)  e^{\E_X \left [ S_G^n(\boldsymbol{\gamma})\right ]} g_G^n(\boldsymbol{\gamma}).
\]
By Assumption~\ref{hyp1}, \(S_G^n (\boldsymbol{\gamma}) = \sum_{w \in W} \E_X \left [ Z_w^n (\boldsymbol{\gamma})\right ]\) converges to \(S_G(\boldsymbol{\gamma}) = \sum_{w \in W} \E_X \left [ Z_w (\boldsymbol{\gamma})\right ]\), where \(Z_w(\boldsymbol{\gamma})\) is defined in~\eqref{eq: Z_w}, and \(g_G^n (\boldsymbol{\gamma})\) converges to
\begin{equation} \label{eq: function g_G}
g_G (\boldsymbol{\gamma}) \coloneqq \prod_{i=1}^R \left ( \sum_{K \geq 1} \sum_{\{W_1,\ldots, W_K\} \in \mathcal{A}_K (W_{B_i})} \prod_{k =1}^K \E_X \left [ \prod_{w \in W_k} Z_w(\boldsymbol{\gamma})\right ] \right).
\end{equation}
Since \(|e^{\E_X \left [ S_G^n (\boldsymbol{\gamma})\right ]}| \le 1\) and \(g_G^n (\boldsymbol{\gamma})\) grows at most polynomially in \(\boldsymbol{\gamma}\), while \(\hat{f}\) decays faster than any polynomial by~\eqref{thehypf}, the integrand in \(I_3^n(G)\) is dominated on \(\R^{|E|}\) by the integrable function \(\prod_{e,i} |\hat{f}(\gamma_e^i)|\). Moreover, as in~\eqref{largegamma}, the contribution of the tail outside \([-n^\epsilon, n^\epsilon]^{|E|}\) is \(o(1)\). Hence, by the dominated convergence theorem,
\[
I^n_3(G)  = \frac{ \phi^{\frac{|E|}{2} - |V|}}{\psi^{|W|-1}} \frac{1}{(2\pi)^{|E|}} \int_{\R^{|E|}} \prod_{e \in E}  \prod_{i=1}^{m(e)} \textnormal{d} \gamma_e^i \hat f(\gamma_e^i)  e^{\E_X \left [ S_G(\boldsymbol{\gamma})\right ]} g_G(\boldsymbol{\gamma})  + o(1) .
\]
Let \(E_{B_i}\) denote the edge set of block \(B_i = G(W_{B_i})\). Then, the above integrand splits across the disjoint edge sets \(E \backslash \cup_{i=1}^R E_{B_i}\) and \(E_{B_1}, \ldots, E_{B_R}\), yielding
\[
\begin{split}
I^n_3(G) & =  \frac{ \phi^{\frac{|E|}{2} - |V|}}{\psi^{|W|-1}} \frac{1}{(2 \pi)^{|E \backslash \cup_{i=1}^R E_i|}} \int_{\R^{|E \backslash \cup_{i=1}^R E_i|}} \prod_{e \in E \backslash \cup_{i=1}^R E_i} \prod_{i=1}^2 \textnormal{d} \gamma_e^i  \hat f(\gamma_e^i) e^{\sum_{w \in W \backslash \cup_{i=1}^R W_{B_i}} \E_X \left [ Z_w (\boldsymbol{\gamma}) \right ]} \\
& \quad \times  \prod_{i=1}^R  \frac{1}{(2\pi)^{|E_{B_i}|}} \int_{\R^{|E_{B_i}|}} \prod_{e \in E_{B_i}}  \prod_{i=1}^{m(e)} \textnormal{d} \gamma_e^i \hat f(\gamma_e^i)  e^{\sum_{w \in W_{B_i}} \E_X \left [ Z_w (\boldsymbol{\gamma})\right ]}   \\
& \qquad \qquad \qquad \qquad  \times  \left ( \sum_{K \ge 1} \sum_{\{W_1, \ldots, W_K\} \in \mathcal{A}_K(W_{B_i})} \prod_{k=1}^K \E_X \left [ \prod_{w \in W_i} Z_w (\boldsymbol{\gamma})  \right] \right) + o(1) \\
& = \frac{ \phi^{\frac{|E|}{2} - |V|}}{\psi^{|W|-1}}  \prod_{w \in W \backslash \cup_{i=1}^R W_{B_i}} C_{\deg(w)} (f) \prod_{i=1}^R \left ( \sum_{K \ge 1} \sum_{\{W_1, \ldots, W_K\} \in \mathcal{A}_K(W_{B_i})}  C_{(W_k)_{k=1}^K}(f)  \right ) + o(1),
\end{split}
\]
as desired.
\end{proof}

It remains to prove the crucial expansion of \(\Lambda_G^n\) stated in Lemma~\ref{lem: main estimate}. To prove this result, we will use the combinatorial estimates from Section~\ref{section: combinatorics}. 

\begin{proof}[\textbf{Proof of Lemma~\ref{lem: main estimate}}]
We again write series formally, as we have already discussed how to truncate them and control the remainder terms. According to~\eqref{eq: Lambda}, \(\Lambda_G^n\) is given by
\begin{equation}   \label{eq: Lambda 1}
\Lambda_G^n(\boldsymbol{\gamma})=  e^{\E_X \left [ S_G^n(\boldsymbol{\gamma}) \right ]} \left ( 1 + \sum_{m_2 + m_3 + \cdots \geq 1} \prod_{\ell \geq 2} \frac{1}{m_\ell!}\left ( \frac{\kappa_\ell^G(\boldsymbol{\gamma})}{\ell! n^{\ell-1}} \right )^{m_\ell} \right ).
\end{equation}
The sum of products of cumulants in the above display can be rewritten as 
\begin{equation} \label{eq: prod cumulants}
\sum_{m_2 + m_3 + \cdots \geq 1} \prod_{\ell \geq 2} \frac{1}{m_\ell!}\left ( \frac{\kappa_\ell^G(\boldsymbol{\gamma})}{\ell! n^{\ell-1}} \right )^{m_\ell} = \sum_{K \ge 1} \frac{1}{K!} \sum_{\ell_1, \ldots, \ell_K \ge 2} \prod_{k=1}^K  \left ( \frac{\kappa_{\ell_k}^G(\boldsymbol{\gamma}) }{\ell_k! n^{\ell_k-1}}\right ),
\end{equation}
where we use the correspondence \(K = \sum_{\ell \geq 2} m_\ell \geq 1\), \(\sum_{k=1}^K \ell_k = \sum_{\ell \geq 2} \ell m_\ell\), and define the sequence \((\ell_k)_{k=1}^K\) by \(\ell_k = 2\) for \(k \in \{1, \ldots, m_2\}\), \(\ell_k = 3\) for \(k \in \{m_2+1, \ldots, m_2+m_3\}\), and so on, i.e., 
\[
\ell_k = \ell \quad \textnormal{for} \enspace k \in \left \{ \sum_{i=2}^{\ell-1} m_i +1, \ldots, \sum_{i=2}^\ell m_i \right \}.
\]
We therefore need to estimate the product \(\prod_{k=1}^K \kappa_{\ell_k}^G(\boldsymbol{\gamma})\) appearing in~\eqref{eq: prod cumulants}. According to~\eqref{eq: cumulant} and the decomposition \(S_G^n (\boldsymbol{\gamma}) = \sum_{w \in W} Z_w^n(\boldsymbol{\gamma})\), we can expand \(\kappa_\ell^G (\boldsymbol{\gamma})\) as
\begin{equation} \label{eq: expansion cumulant}
\kappa_\ell^G (\boldsymbol{\gamma}) = \sum_{\pi \in \mathcal{P} ([\ell])} \mu(\pi) \prod_{i =1}^{|\pi|} \sum_{w_1^i, \ldots, w_{|B_i|}^i \in W} \E_X \left [ \prod_{j=1}^{|B_i|} Z_{w_j^i}^{n}(\boldsymbol{\gamma})\right ],
\end{equation}
where we recall that \(\mathcal{P}([\ell])\) denotes the set of partitions of \([\ell]\) and \(B_1, \ldots, B_{|\pi|}\) are the blocks of \(\pi\). For each block \(B_i\), we regroup the sum by collecting all tuples \((w_1^i,\ldots,w_{|B_i|}^i)\) that contain the same multiset of indices. Equivalently,
\begin{equation} \label{eq: expansion cumulant 2}
\sum_{w^i_1, \ldots, w^i_{|B_i|} \in W}  \E_X \left [\prod_{j=1}^{|B_i|}Z_{w^i_j}^{n}(\boldsymbol{\gamma}) \right ] = \sum_{W_i} \sum_{\boldsymbol{\eta}_{W_i}} \frac{|B_i| ! }{\prod_{w \in W_i} \eta_{W_i}(w)! } \E_X \left [\prod_{w \in W_i} \left(Z_w^n(\boldsymbol{\gamma})\right)^{\eta_{W_i}(w)} \right ],
\end{equation}
where the first sum runs over all subsets \(W_i \subseteq W\) with \(|W_i| \le |B_i|\), and the second runs over integer vectors \(\boldsymbol{\eta}_{W_i} = (\eta_{W_i}(w))_{w \in W_i} \in \mathbb{N}^{|W_i|}\) satisfying
\[
\eta_{W_i}(w) \ge 1, \qquad \sum_{w \in W_i}\eta_{W_i}(w) = |B_i|.
\]
In particular, \(|W_i| = |B_i|\) if and only if \(\eta_{W_i}(w) = 1\) for all \(w \in W_i\). Combining~\eqref{eq: expansion cumulant} and~\eqref{eq: expansion cumulant 2} yields
\[
\kappa_\ell^G (\boldsymbol{\gamma}) =  \sum_{\pi \in \mathcal{P} ([\ell])} \mu(\pi) \sum_{W_1, \ldots, W_{|\pi|}} \sum_{\boldsymbol{\eta}_{W_1},\ldots, \boldsymbol{\eta}_{W_{|\pi|}}}  \prod_{i=1}^{|\pi|} \frac{|B_i| ! }{\prod_{w \in W_i} \eta_{W_i}(w)! } \E_X \left [\prod_{w \in W_i} \left (Z_w^n(\boldsymbol{\gamma})\right)^{\eta_{W_i}(w)} \right ],
\]
where different subsets \(W_i\) (associated with distinct blocks \(B_i\)) may have a nontrivial intersection. As a result, we obtain
\begin{equation} \label{eq: prod cumulant}
\prod_{k=1}^K \kappa_{\ell_k}^G(\boldsymbol{\gamma}) = \sum_{\pi_1, \ldots, \pi_K} \mu(\pi_1) \cdots \mu(\pi_K) \sum_{\boldsymbol{W}^1, \ldots, \boldsymbol{W}^K} \sum_{\boldsymbol{\eta}^1, \ldots, \boldsymbol{\eta}^K} P_{(\pi_k, \boldsymbol{W}^k, \boldsymbol{\eta}^k)_{k=1}^K}(\boldsymbol{\gamma}),
\end{equation}
where
\begin{itemize}
\item[(i)] each \(\pi_k \in \mathcal{P} ( [\ell_k])\) is a partition of \([\ell_k]\) with blocks \(B_1^k,\ldots, B_{|\pi_k|}^k\);
\item[(ii)] \(\boldsymbol{W}^k= \{W^k_i, 1 \leq i \leq |\pi_k|\}\) denotes a collection of subsets \(W_i^k \subseteq W\) satisfying \(|W_i^k| \leq  |B_i^k|\), with equality if and only if 
\(\eta_{W_i^k}(w) = 1\) for all \(w \in W_i^k\);
\item[(iii)] \(\boldsymbol{\eta}^k= \{ \boldsymbol{\eta}_{W^k_i}, 1 \leq i \leq |\pi_k|\}\) denotes a collection of integer vectors \(\boldsymbol{\eta}_{W^k_i} = (\eta_{W_i^k}(w))_{w \in W_i^k}  \in \N^{|W_i^k|}\) satisfying \(\eta_{W_i^k}(w) \ge 1\) for all \(w \in W_i^k\) and \(\sum_{w \in W^k_i} \eta_{W_i^k}(w)  = |B_i^k|\);
\item[(iv)] the term \(P_{(\pi_k, \boldsymbol{W}^k, \boldsymbol{\eta}^k)_{k=1}^K}(\boldsymbol{\gamma})\) is given by
\begin{equation} \label{eq: product}
P_{(\pi_k, \boldsymbol{W}^k, \boldsymbol{\eta}^k)_{k=1}^K}(\boldsymbol{\gamma}) = \prod_{k=1}^K \prod_{i=1}^{|\pi_k|} \frac{|B_i^k| ! }{\prod_{w \in W_i^k} \eta_{W_i^k}(w)! }  \E_X \left [ \prod_{w \in W^k_i} \left (Z_w^n(\boldsymbol{\gamma}) \right)^{\eta_{W_i^k}(w)} \right ].
\end{equation}
\end{itemize}
Since the subgraph \(G(W_i^k)\) associated with \(W_i^k\) is not necessarily connected, the expectation in~\eqref{eq: product} may factorize. We therefore decompose each subset \(W_i^k\) into the vertex sets of the connected components of the induced subgraph \(G(W_i^k)\), i.e., 
\[
W_i^k = \sqcup_{j=1}^{c_i^k} \widetilde{W}_{i,j}^k,
\]
so that \(G(\widetilde{W}_{i,j}^k)\) is connected. We define \(r(\pi_k) \coloneqq \sum_{i=1}^{|\pi_k|} c_i^k\), so that \(\boldsymbol{W}^k\) decomposes into \(r(\pi_k) \ge |\pi_k|\) subsets \(\left \{\widetilde{W}_{i,j}^k, 1 \le i \le |\pi_k|, 1 \le j \le c_i^k \right\}\), each corresponding to a connected subgraph \(G(\widetilde{W}_{i,j}^k)\). For each component \( \widetilde{W}_{i,j}^k\), we define the restricted multiplicity 
\[
\widetilde{\eta}_{i,j}^k (w) \coloneqq \begin{cases} \eta_{W_i^k}(w) & \enspace \text{if} \: w \in \widetilde{W}_{i,j}^k,\\
0 & \enspace \text{otherwise}. 
\end{cases}
\]
By independence of disjoint connected subgraphs, we have
\[
\E_X \left [ \prod_{w \in W^k_i} \left (Z_w^n(\boldsymbol{\gamma}) \right)^{\eta_{W_i^k}(w)} \right ] = \prod_{j=1}^{c_i^k} \E_X \left [ \prod_{w \in \tilde{W}^k_{i,j}} \left (Z_w^n(\boldsymbol{\gamma}) \right)^{\tilde{\eta}^k_{i,j}(w)} \right ] .
\]
In the sequel, set \(r \coloneqq \sum_{k=1}^K r(\pi_k)\) and enumerate all connected components as \(\{\widetilde{W}_\ell\}_{\ell=1}^r = \{\widetilde{W}_{i,j}^k, 1 \le k \le K, 1 \le i \le |\pi_k|, 1 \le j \le c_i^k\}\). For each \(\widetilde{W}_\ell\), define \(\tilde{\boldsymbol{\eta}}_\ell = \big(\tilde{\eta}_\ell(w)\big)_{w \in \widetilde{W}_\ell}\) by
\[
\tilde{\eta}_\ell(w) \coloneqq 
\begin{cases}
\eta_{W_i^k}(w) & \text{if} \enspace w \in \widetilde{W}_\ell \text{ for the unique }(k,i,j)\text{ with }\widetilde{W}_\ell=\widetilde{W}_{i,j}^k,\\
0 & \text{otherwise}.
\end{cases}
\]
Then, substituting the previous identity into~\eqref{eq: product} yields
\begin{equation}\label{eq: product connected}
P_{(\pi_k, \boldsymbol{W}^k, \boldsymbol{\eta}^k)_{k=1}^K}(\boldsymbol{\gamma})  = \left ( \prod_{k=1}^K \prod_{i=1}^{|\pi_k|} \frac{|B_i^k|!}{\prod_{w \in W_i^k} \eta_{W_i^k}(w)!} \right) P_{(\widetilde{W}_\ell, \tilde{\boldsymbol{\eta}}_\ell)_{\ell=1}^r} (\boldsymbol{\gamma}),
\end{equation}
where
\[
P_{(\widetilde{W}_\ell, \tilde{\boldsymbol{\eta}}_\ell)_{\ell=1}^r}(\boldsymbol{\gamma}) = \prod_{\ell=1}^r \E_X \left [ \prod_{w \in \widetilde{W}_\ell} \left (Z_w^n(\boldsymbol{\gamma}) \right)^{\tilde{\eta}_\ell (w)}  \right ].
\]
Finally, combining~\eqref{eq: prod cumulant} and~\eqref{eq: product connected}, we obtain the following expansion for the product of cumulants:
\begin{equation}\label{eq: final product cumulant}
\prod_{k=1}^K \kappa_{\ell_k}^G(\boldsymbol{\gamma}) = \sum_{\pi_1,\ldots,\pi_K} \mu(\pi_1)\cdots \mu(\pi_K) \sum_{\boldsymbol{W}^1,\ldots,\boldsymbol{W}^K} \sum_{\boldsymbol{\eta}^1,\ldots,\boldsymbol{\eta}^K} \left(\prod_{k=1}^K \prod_{i=1}^{|\pi_k|} \frac{|B_i^k|!}{\prod_{w\in W_i^k}\eta_{W_i^k}(w)!} \right) P_{(\widetilde{W}_\ell, \tilde{\boldsymbol{\eta}}_\ell)_{\ell=1}^r} (\boldsymbol{\gamma}).
\end{equation}

We next study the right-hand side of~\eqref{eq: final product cumulant} using the results from Section~\ref{section: combinatorics}. Because Section~\ref{section: combinatorics} deals with connected subgraphs of \(G\), its results apply to the subsets \(\widetilde{W}_1,\ldots, \widetilde{W}_r\), since by definition the associated subgraphs \(G(\widetilde{W}_1),\ldots, G(\widetilde{W}_r)\) are connected. According to Proposition~\ref{prop: combinatorics}, if \(\{\widetilde{W}_1, \ldots, \widetilde{W}_r\} \notin \mathcal{W}_r\), then either \(P_{(\widetilde{W}_\ell, \tilde{\boldsymbol{\eta}}_\ell)_{\ell=1}^r}(\boldsymbol{\gamma})\) is invariant under \(T_{E_0}\) for some subset \(E_0 \subset E\) with odd cardinality \(|E_0|\)---and can thus be absorbed into \(h_G^n (\boldsymbol{\gamma})\)---or 
\[
\rho(G) = |W| + |V| - \frac{|E|}{2} - 1 < \sum_{\ell =1}^r (|\widetilde{W}_\ell|-1).
\]
Since by definition \(\sum_{\ell=1}^r |\widetilde{W}_\ell| =  \sum_{k=1}^K \sum_{i=1}^{|\pi_k|} \sum_{j=1}^{c_i^k} |\widetilde{W}_{i,j}^k| =  \sum_{k=1}^K \sum_{i=1}^{|\pi_k|} |W_i^k|\) and \(r = \sum_{k=1}^K r(\pi_k)\), we obtain
\[
\sum_{\ell=1}^r (|\widetilde{W}_\ell|-1) = \sum_{k=1}^K \left( \sum_{i=1}^{|\pi_k|} |W_i^k| - r(\pi_k) \right ) \leq \sum_{k=1}^K \left( \sum_{i=1}^{|\pi_k|} |B_i^k| - |\pi_k| \right ) \leq \sum_{k=1}^{K} \left( \ell_k-1 \right),
\]
where the first inequality follows from \(|W_i^k| \leq |B_i^k|\) (item (ii)) and \(r(\pi_k) \ge |\pi_k|\), and the second from \(\ell_k = \sum_{i=1}^{|\pi_k|} |B_i^k| \) (item (i)) together with \(|\pi_k| \geq 1\). As a result, if \(\rho(G) < \sum_{\ell =1}^r (|\widetilde{W}_\ell|-1)\), then \(\rho(G) < \sum_{k=1}^{K} \left( \ell_k-1 \right)\), and thus, by~\eqref{eq: prod cumulants}, the corresponding contribution is absorbed into the error term \(g_G^n\). If instead \(\{\widetilde{W}_1, \ldots, \widetilde{W}_r\} \in \mathcal{W}_r\), then Proposition~\ref{prop: combinatorics} gives \(\rho(G) = \sum_{\ell=1}^r (|\widetilde{W}_\ell|-1)\), and the argument above implies that
\[
\rho(G) = \sum_{\ell=1}^r (|\widetilde{W}_\ell|-1) \leq \sum_{k=1}^K (\ell_k-1),
\] 
with equality when \(|W_i^k| = |B_i^k|\) (that is, \(\eta_{W_i^k}(w)=1\) for every \(w \in W_i^k\) by (ii)) and \(r(\pi_k) = |\pi_k| =1\). In this case, the sum in~\eqref{eq: prod cumulant} simplifies drastically: only the trivial partition \(\pi_k = \{\{1, \ldots, \ell_k\}\}\) for all \(k \in [K]\) contribute at leading order. Therefore, we only sum over subsets \(W_1^1, \ldots, W_1^K\) with \(|W_1^k| = |B_1^k| = \ell_k \ge 2\) for every \(k \in [K]\). For convenience, we denote \(W_1^k = W_k\). The main contribution to \(\Lambda_G^n\) therefore comes from the family of subsets \(\{W_1,\ldots, W_K\} \in \mathcal{W}_K\) with \(|W_k| = \ell_k \ge 2\). From~\eqref{eq: final product cumulant} we have
\[
\begin{split}
& \sum_{K \geq 1} \frac{1}{K!}  \sum_{\ell_1, \ldots, \ell_K \ge 2} \prod_{k=1}^K  \left ( \frac{\kappa_{\ell_k}^G(\boldsymbol{\gamma})}{\ell_k! n^{\ell_k-1}}  \right ) \\
& \quad = h_G^n(\boldsymbol{\gamma}) + \frac{1}{n^{\rho(G)}} \sum_{K\ge 1} \sum_{\{W_1,\ldots, W_K\} \in \mathcal{W}_K \atop |W_k| \ge 2} \prod_{k=1}^K \E_X \left [ \prod_{w \in W_k} Z_w^n(\boldsymbol{\gamma}) \right ] + o\left ( \frac{1}{n^{\rho(G)}} \right ), 
\end{split}
\]
where the error \(o\left ( n^{-\rho(G)} \right )\) is uniform for \(\|\gamma\|_{\infty}\le n^\epsilon\) and \(h_G^n (\boldsymbol{\gamma})\) is a sum of functions which are invariant under \(T_{E_0}\) for some subset \(E_0 \subset E\) such that \(|E_0|\) is odd. The normalizing factor \(1/K!\) disappears because the class \(\mathcal{W}_K\) consists of unordered families of distinct subsets of \(W\) (subsets may intersect), so each configuration is counted only once. Combining this estimate with~\eqref{eq: Lambda 1} and~\eqref{eq: prod cumulants}, we obtain 
\[
\begin{split}
\Lambda_G^n(\boldsymbol{\gamma}) & = e^{\E_X \left [ S_G^n(\boldsymbol{\gamma}) \right ]} \left ( 1 + h_G^n (\boldsymbol{\gamma}) +  \frac{1}{n^{\rho(G)}} \sum_{K\geq 1}  \sum_{\{W_1,\ldots, W_K\} \in \mathcal{W}_K \atop |W_k| \ge 2} P_{(W_k)_{k=1}^K}(\boldsymbol{\gamma}) + o\left ( \frac{1}{n^{\rho(G) }} \right ) \right ),
\end{split}
\]
where \(P_{(W_k)_{k=1}^K}(\boldsymbol{\gamma})\) is given by \(P_{(W_k)_{k=1}^K}(\boldsymbol{\gamma})= \prod_{k=1}^K \E_X \left [ \prod_{w \in W_k} Z_w^n(\boldsymbol{\gamma}) \right ]\). 

By Lemma~\ref{lem: equiv def} (and the bijection constructed in its proof), for any connected bipartite graph \(G=(W \cup V,E)\) in which all vertices have even degree, there is a one-to-one correspondence between families \(\{W_1, \ldots, W_K\} \in \mathcal{W}_K\) (for some integer \(K \ge 1\)) satisfying \(|W_k| \ge 2\) for all \(k \in [K]\), and admissible graphs with \(R \ge 1\) blocks \(B_i = G(W_{B_i})\), \(|W_{B_i}| \ge 2\), together with admissible decompositions \(\{ W_\ell \colon \ell \in I_i \} \in \mathcal{A}_{|I_i|}(W_{B_i})\), where the index sets \(I_1, \ldots, I_R \subset [K]\) are disjoint and satisfy \(K = \sum_{i=1}^R |I_i|\). Setting \(K_i \coloneqq |I_i|\) and writing \(\{W_1^i, \ldots, W_{K_i}^i\} \coloneqq \{W_\ell \colon \ell \in I_i\}\), the sum over \(\mathcal{W}_K\) appearing in the expansion of \(\Lambda_G^n(\boldsymbol{\gamma})\) can therefore be rewritten in a blockwise form as
\[
\begin{split}
& \sum_{K\geq 1}  \sum_{\{W_1,\ldots, W_K\} \in \mathcal{W}_K \atop |W_k| \ge 2} P_{(W_k)_{k=1}^K}(\boldsymbol{\gamma})  \\
& = \mathbf{1}_{\{G \, \textnormal{is admissible with} \, R \ge 1\}}  \sum_{K_1, \ldots, K_R \geq 1}  \prod_{i=1}^R \left(\sum_{\{W_1^i,\ldots, W_{K_i}^i\} \in \mathcal{A}_{K_i} (W_{B_i})}   P_{(W_k^i)_{k=1}^{K_i}}(\boldsymbol{\gamma})\right) \\
& = \mathbf{1}_{\{G \, \textnormal{is admissible with} \, R \ge 1\}} \prod_{i=1}^R \left ( \sum_{K \ge 1} \sum_{\{W_1, \ldots, W_K\} \in \mathcal{A}_K (W_{B_i})} P_{(W_k)_{k=1}^K} (\boldsymbol{\gamma}) \right ).
\end{split}
\]
This completes the proof of Lemma~\ref{lem: main estimate}.
\end{proof}

\section{Combinatorial estimates} \label{section: combinatorics}

Let \(G=(W \cup V,E)\) be a finite, connected bipartite multigraph, where \(E\) is a multiset of edges and \(m \colon E \to \N\) assigns to each edge \(e\in E\) its multiplicity. For a fixed positive integer \(r\), consider subsets \(W_1,\ldots, W_r\) of \(W\), which may have a nontrivial intersection, and sequences \(\boldsymbol{\eta}_1, \ldots, \boldsymbol{\eta}_r\) with \(\boldsymbol{\eta}_i = \{\eta_i(w), w \in W_i\} \in \N^{|W_i|}\). For every integer \(n \ge 1\), consider the integral
\begin{equation} \label{eq: integral}
\int \prod_{e \in E} \prod_{i=1}^{m(e)} \textnormal{d} \gamma_e^i g(\gamma_e^i) e^{\sum_{w \in W} \E_X \left [ Z_w^n(\boldsymbol{\gamma})\right ]}  P_{(W_i, \boldsymbol{\eta}_i)_{i=1}^r}  (\boldsymbol{\gamma}),
\end{equation}
where \(g \colon \R \to \R\) is an odd, \(L^1\)-integrable function and
\[
P_{(W_i, \boldsymbol{\eta}_i)_{i=1}^r}  (\boldsymbol{\gamma}) = \prod_{i=1}^r \E_X \left [\prod_{w \in W_i} \left ( Z_w^n(\boldsymbol{\gamma})\right)^{\eta_i(w)} \right ].
\]
For every \(w \in W\) and integer \(n\), the random variable \(Z_w^n(\boldsymbol{\gamma})\) is given by 
\[
Z_w^n(\boldsymbol{\gamma}) = n \log \E_W \left [ e^{i W_w \sum_{v \in V \colon v \sim w} \left( \gamma_{(w,v)}^1 + \cdots + \gamma_{(w,v)}^{m((w,v))} \right )X_v} \right ],
\]
where \((W_w)_{w \in W}\) are i.i.d.\ random variables distributed according to \(\nu_w\) and \((X_v)_{v \in V}\) are i.i.d.\ random variables distributed according to \(\nu_x\) (see Assumption~\ref{hyp1}). The integral~\eqref{eq: integral} arises in the proof of Proposition~\ref{main3}. In this section, we establish auxiliary estimates that will be used in that proof. Specifically, we aim (1) to determine the conditions on the subsets under which the integral in~\eqref{eq: integral} does not vanish, (2) to prove that \(\rho (G) \le \sum_{i=1}^r (|W_i|-1)\), and (3) to identify the precise conditions under which equality holds.

Fix \(w \in W\). We first observe that \(\E_X \left [ Z_w^n(\boldsymbol{\gamma})\right ]\) remains unchanged if we simultaneously replace each \(\gamma_{(w,v)}^i\) by \(-\gamma_{(w,v)}^i\) for every \(v \sim w\) and \(1 \le i \le m((w,v))\). This invariance follows from the symmetry of \(\nu_x\) and from the fact that each random variable \(X_v\) appears only once inside \(\E_X \left [Z_w^n(\boldsymbol{\gamma})\right ]\). For any subset \(E_0 \subseteq E\), define the sign-flip operator
\[
T_{E_0} \colon \R^{\sum_{e\in E} m(e)} \to \R^{\sum_{e\in E} m(e)}, 
\]
which acts on \(\boldsymbol{\gamma} = (\gamma_e^i)_{e \in E, 1 \le i \le m(e)}\) by 
\[
(T_{E_0} \boldsymbol{\gamma})_e^i = \begin{cases}
- \gamma_e^i &\enspace e \in E_0, \\
 \gamma_e^i &\enspace e \notin E_0.
\end{cases}
\]
That is, \(T_{E_0}\) flips the signs of all \(\gamma_e^i\) corresponding to edges \(e \in E_0\) and leaves the others unchanged. Under this notation, \(\E_X \left [ Z_w^n(\boldsymbol{\gamma})\right ]\) is invariant under \(T_{\{e = (w,v) \colon v \sim w\}}\). More generally, if \(E(w) \coloneqq \{ (w,v) \in E \colon w \sim v\}\), then \(\E_X \left [ Z_w^n(\boldsymbol{\gamma})\right ]\) is invariant under \(T_{E_0(w)}\) for any subset \(E_0(w) \subseteq E(w)\). As a consequence, \(\E_X \left [ \sum_{w \in W} Z_w^n(\boldsymbol{\gamma})\right]\) is invariant under \(T_{E_0}\) for any subset \(E_0 \subseteq E\). Since the function \(g\) is odd, if \(P_{(W_i,\boldsymbol{\eta}_i)_{i=1}^r}(\boldsymbol{\gamma})\) is invariant under some transformation \(T_{E_0}\) with odd cardinality \(|E_0| = \sum_{e \in E_0} m(e)\), then the integral~\eqref{eq: integral} vanishes. In the first part of this section, our goal is to determine conditions on \(W_1,\ldots, W_r\) such that \(P_{(W_i,\boldsymbol{\eta}_i)_{i=1}^r}(\boldsymbol{\gamma})\) fails to be invariant under any transformation \(T_{E_0}\) with \(|E_0|\) odd. \\

Given subsets \(W_1,\ldots, W_r\) of \(W\), let \(G_i = (W_i \cup V_i, E_i)\) denote the subgraph of \(G\) as in Definition~\ref{def: subgraphs}, where
\[
V_i = \{ v \in V \colon \exists \, w \in W_i , \, v \sim w\},
\]
and 
\[
E_i = \{ e = (w,v) \in E \colon w \in W_i, \, v \in V_i\}.
\] 
Write \(\deg_{G_i}(x)\) for the degree of a vertex \(x \in W_i \cup V_i\) within the subgraph \(G_i\), and \(\deg(x)\) its degree within the entire graph \(G\). We assume that each subgraph \(G_i\) is connected, ensuring that the expectation \(\E_X \left [\prod_{w \in W_i} \left ( Z_w^n(\boldsymbol{\gamma})\right)^{\eta_i(w)} \right ]\) in \(P_{(W_i, \boldsymbol{\eta}_i)_{i=1}^r}  (\boldsymbol{\gamma}) \) does not factorize. We now provide the necessary conditions for the integral~\eqref{eq: integral} to be nonvanishing.

\begin{lem} \label{lem: necessary cond}
Let \(\tilde{w}_1,\ldots, \tilde{w}_q\) (\(q \geq 0\)) denote the vertices that belong to at least two of the subsets \(W_1,\ldots, W_r\). If \(W_1,\ldots, W_r\) are disjoint, then \(q=0\). Consider the following conditions:
\begin{enumerate} 
\item[(A)] for every \(i \in [r]\) and every \(w \in W_i\), \(\deg(w) = \deg_{G_i}(w) \geq 2\);
\item[(B)] for every \(v\in \cup_{i=1}^r V_i\), there exists at least one \(V_i\) such that \(\deg_{G_i}(v) \geq 2\);
\item[(C)] for every \(v \in V_i\) such that \(v \not \sim \tilde{w}_1,\ldots, \tilde{w}_q\), \(\deg_{G_i}(v)\) is even (if \(W_1, \ldots,W_r\) are disjoint, then \(\deg_{G_i}(v)\) is even for all \(v \in V_i\));
\item[(D)] for every \(e \in E \backslash \cup_{i=1}^r E_i\), the multiplicity \(m(e)\) is even. 
\end{enumerate}
If any of the above conditions is not satisfied, then there exists at least one subset \(E_0 \subset E\) with odd cardinality \(|E_0|\) such that \(P_{(W_i,\boldsymbol{\eta}_i)_{i=1}^r}(\boldsymbol{\gamma})\) is invariant under \(T_{E_0}\). Consequently, the integral~\eqref{eq: integral} vanishes. 
\end{lem}

\begin{proof}
We first focus on condition (A). We observe that for every \(i \in [r]\) with \(w \in W_i\), the degree of \(w\) in \(G_i\) is independent of \(i\) because the neighborhood of vertices in \(W\) does not depend on \(i\). Consequently, \(\deg(w) = \deg_{G_i}(w)\) for all \(w \in \cup_{i=1}^r W_i\). If condition (A) is not satisfied, there exists \(w_0 \in \cup_{i=1}^r W_i\) with \(\deg(w_0) = 1\). Let \(v_0 \in V\) denote the unique neighbor of \(w_0\). Then, for every \(i \in [r]\) such that \(w_0\in W_i\), it follows that
\[
\begin{split}
\E_X \left [\prod_{w \in W_i} \left(Z_w^n(\boldsymbol{\gamma})\right)^{\eta_i(w)} \right] & = \E_X \left[ \left(Z_{w_0}^n (\boldsymbol{\gamma})\right)^{\eta_i(w_0)}  \prod_{w \in W_i \backslash \{w_0\}} \left( Z_w^n(\boldsymbol{\gamma}) \right)^{\eta_i(w)} \right ] \\
& = \E_X \left [ \left (n \log \E_W \left [ e^{i \gamma_{(w_0,v_0)} X_{v_0} W_{w_0}} \right ] \right)^{\eta_i(w_0)} \prod_{w \in W_i \backslash \{w_0\}} \left(Z_w^n(\boldsymbol{\gamma})\right)^{\eta_i(w)} \right ] 
\end{split}
\]
is even in \(\gamma_{(w_0,v_0)}\), since the law of \(X_{v_0}\) is symmetric and \(X_{v_0}\) appears only once in the expectation. This implies that \(P_{(W_i, \boldsymbol{\eta}_i)_{i=1}^r} (\boldsymbol{\gamma})\) is invariant under \(T_{\{(w_0,v_0)\}}\). If condition (B) is not satisfied, then there exists at least one \(v_0 \in \cup_{i=1}^r V_i\) with \(\deg(v_0) =1\) in \(\cup_{i=1}^r G_i\). Let \(w_0 \in W\) be the unique neighbor of \(w_0 \sim v_0\). For every \(i \in [r]\) such that \(w_0\in W_i\), we have 
\[
\E_X \left [\prod_{w \in W_i} \left ( Z_w^n(\boldsymbol{\gamma}) \right)^{\eta_i(w)}  \right ] = \E_X \left [ \left(Z_{w_0}^n(\boldsymbol{\gamma})\right)^{\eta_i(w_0)} \prod_{w \in W_i \backslash \{w_0\}} \left( Z_w^n(\boldsymbol{\gamma}) \right)^{\eta_i(w)}  \right ] ,
\]
where 
\[
Z_{w_0}^n (\boldsymbol{\gamma}) = n \log \E_W \left [ e^{i \left ( \gamma_{(w_0,v_0)} X_{v_0} + \sum_{v \in V \backslash \{v_0\} \colon v \sim w_0}  (\gamma_{(w_0,v)}^1 + \cdots +\gamma_{(w_0,v)}^{m((w_0,v))})  X_v \right) W_{w_0}} \right ] .
\]
Since the law of \(X_{v_0}\) is symmetric and \(X_{v_0}\) appears only once in the expectation, \(\E_X \left [\prod_{w \in W_i} \left ( Z_w^n(\boldsymbol{\gamma}) \right)^{\eta_i(w)}  \right ] \) is even in \(\gamma_{(w_0,v_0)}\). This implies that \(P_{(W_i, \boldsymbol{\eta}_i)_{i=1}^r} (\boldsymbol{\gamma})\) is even in \(\gamma_{(v_0, w_0)}\). We now consider condition (C), first assuming that \(W_1,\ldots,W_r\) are disjoint. If (C) is not satisfied, then there exists at least one subgraph \(G_i\) with at least one vertex \(v_0 \in V_i\) such that \(\deg_{G_i}(v_0)\) is odd. Let \(w_1, \ldots, w_k \in W_i\) denote the neighbors of \(v_0\) in \(G_i\). Then,
\[
\E_X \left [\prod_{w \in W_i} \left ( Z_w^n(\boldsymbol{\gamma}) \right)^{\eta_i(w)}  \right ] = \E_X \left [\left ( Z_{w_1}^n (\boldsymbol{\gamma}) \right)^{\eta_i(w_1)} \cdots \left ( Z_{w_k}^n (\boldsymbol{\gamma}) \right)^{\eta_i(w_k)} \prod_{w \in W_i\backslash \{w_1,\ldots, w_k\}} \left(Z_w^n (\boldsymbol{\gamma})\right)^{\eta_i(w)} \right ] 
\]
is invariant under \(T_{E_0}\) for \(E_0 = \{(v_0,w_i) \colon i \in [k]\}\) by the same argument as above. Since \(W_1, \ldots, W_r\) are disjoint, it follows that \(P_{(W_i, \boldsymbol{\eta}_i)_{i=1}^r}(\boldsymbol{\gamma})\) is also invariant under \(T_{E_0}\), and \(|E_0| = \sum_{i=1}^k m((v_0,w_i)) = \deg_{G_i}(v_0)\) is odd. Next, we consider the case where \(W_1,\ldots,W_r\) are not necessarily disjoint. Let \(\mathcal{I}\) denote
\[
\mathcal{I} = \left \{i \in [r] \colon \exists \,j \in [q] \enspace \text{such that} \enspace \widetilde{w}_j \in W_i\right \} 
\]
where \(\widetilde{w}_1,\ldots, \widetilde{w}_q \in \cup_{i \in \mathcal{I}} W_i\). Note that the subsets \((W_i)_{i \in \mathcal{I}^\textnormal{c}}\) are disjoint and have a empty intersection with \(\cup_{i \in \mathcal{I}} W_i\). Assume first that there exists \(v_0 \in V_i\) for some \(i \in \mathcal{I}^\textnormal{c}\) such that \(\deg_{G_i}(v_0)\) is odd. Then, by the above computation we have that \(\E_X \left [ \prod_{w \in W_i} \left( Z_w^n(\boldsymbol{\gamma}) \right)^{\eta_i(w)}\right ]\) is invariant under \(T_{E_0}\) for \(E_0 = \{(v_0,w) \colon w \in W_i, w \sim v_0\}\), and so is \(P_{(W_i, \boldsymbol{\eta}_i)_{i =1}^r} (\boldsymbol{\gamma})\) since \(W_i\) is disjoint from the other subsets. Assume now that there exists \(v_0 \in V_i\) for some \(i \in \mathcal{I}\) such that \(\deg_{G_i}(v_0)\) is odd and such that the neighbors \(w_1,\ldots, w_k\) of \(v_0\) in \(G_i\) are distinct from \(\widetilde{w}_1, \ldots, \widetilde{w}_q\). Since the expectation \(\E \left [\prod_{w \in W_i} \left( Z_w^n(\boldsymbol{\gamma}) \right)^{\eta_i(w)}  \right ]\) is invariant under \(T_{E_0}\) for \(E_0 = \{(v_0,w_i) \colon i \in [k]\}\) and since \(w_1,\ldots, w_k\) belong only to \(W_i\), it follows that \(P_{(W_i, \boldsymbol{\eta}_i)_{i=1}^r} (\boldsymbol{\gamma})\) is also invariant under \(T_{E_0}\). The integral~\eqref{eq: integral} vanishes since \(|E_0| = \sum_{i=1}^k m((v_0,w_i))= \deg_{G_i}(v_0)\) is odd. Finally, assume that condition (D) is not satisfied. Then, there exists an edge \(e_0 \in E \backslash \cup_{i=1}^r E_i\) such that \(m(e_0)\) is odd. The product \(P_{(W_i, \boldsymbol{\eta}_i)_{i=1}^r} (\boldsymbol{\gamma})\) is invariant under \(T_{\{e_0\}}\), since it does not depend on \(\gamma_{e_0}^1, \ldots, \gamma_{e_0}^{m(e_0)}\). Therefore, the integral~\eqref{eq: integral} vanishes, as \(m(e_0)\) is odd. 
\end{proof}

With conditions (A)-(D) from Lemma~\ref{lem: necessary cond} in place, our next goal is to estimate \(\sum_{i=1}^r |W_i|\) in terms of \(|W|, |V|\), and \(|E|\). To this end, let \(s \ge 0\) be an integer, and let \(W_{r+1}, \ldots, W_{r+s}\) form a partition of the remaining vertices of \(W\), namely
\[
W \backslash \cup_{i=1}^r W_i = \cup_{i=1}^s W_{r+i},
\]
with the convention that \(s=0\) if \( W = \cup_{i=1}^r W_i\). We define the corresponding bipartite subgraphs \(G(W_{r+1}), \ldots, G(W_{r+s})\) as in Definition~\ref{def: subgraphs} and assume that \(G(W_{r+1}), \ldots, G(W_{r+s})\) are connected. Note that the subsets \(W_{r+1},\ldots, W_{r+s}\) may have a nontrivial intersection.

We first consider the simpler situation where all subsets \(W_1, \ldots, W_r\) and \(W_{r+1}, \ldots, W_{r+s}\) are disjoint. Since \(G\) is a connected graph and \(W_1, \ldots, W_{r+s}\) are disjoint subsets, the corresponding subgraphs \(G(W_1), \ldots, G(W_{r+s})\) must be connected through vertices in \(V\). Recall that we denote the set of vertices in \(V\) that are common to several subgraphs \(G_1, \ldots, G_{r+s}\) as follows. For every \(k \in [2,r+s]\) and every choice of indices \(1 \leq \ell_1 < \cdots < \ell_k \leq r+s\), define
\[
V_{G_{\ell_1}, \ldots, G_{\ell_k}} \coloneqq \left \{ v \in V \colon \exists \, 1 \le i < j \le k \enspace \textnormal{such that} \enspace v \in V_{\ell_i} \cap V_{\ell_j} \right \}.
\]

\begin{lem} \label{lem: card disjoint case}
Let \(G = (W \cup V, E)\), and let \(\{G_i = G(W_i), 1 \leq i \leq r+s\}\) be connected bipartite subgraphs such that \(G = \cup_{i=1}^r G_i \sqcup \cup_{i=1}^{s} G_{r+i}\). Assume that 
\begin{itemize}
\item[(i)] \(W_1, \ldots, W_r\) are disjoint and satisfy conditions (A)-(C);
\item[(ii)] \(W_{r+1}, \ldots,W_{r+s}\) are disjoint, and condition (D) holds, i.e., for every \(i \in [s]\) and every \(e \in E_{r+i}\), \(m(e)\) is even.
\end{itemize}
Then, 
\begin{equation} \label{eq: estimate V disjoint case}
|V| \leq \frac{|E|}{2} - (r+s) + 1,
\end{equation}
from which it follows that 
\begin{equation} \label{eq: estimate W+V disjoint case}
|W| + |V| - \frac{|E|}{2} -1 \leq \sum_{i=1}^{r+s} (|W_i| -1) .
\end{equation}
Moreover, equality in~\eqref{eq: estimate V disjoint case} (and hence also in~\eqref{eq: estimate W+V disjoint case}) holds if and only if the following conditions are satisfied:
\begin{itemize}
\item[(a)] for every \(i \in [r]\) and every \(v\in V_i\), \(\deg_{G_i}(v)=2\);
\item[(b)] for every \(i \in [s]\) and every \(v \in V_{r+i}\), \(\deg_{G_{r+i}}(v)=2\);
\item[(c)] for every \(k \in [2,r+s]\) and every \(1 \leq \ell_1 < \cdots < \ell_k \leq r+s\), \(|V_{G_{\ell_1}, \ldots, G_{\ell_k}}| \leq k-1\).
\end{itemize}
\end{lem}

\begin{rmk} \label{rmk: cond outside 1}
Item (b) of Lemma~\ref{lem: card disjoint case} admits a more explicit characterization. Under assumptions (i) and (ii), the condition \(\deg_{G_{r+i}}(v)=2\) for every \(v \in V_{r+i}\) is equivalent to
\[
|W_{r+i}| = 1 \quad \textnormal{and} \quad m((w,v)) = 2 \enspace \textnormal{for every} \, w \sim v.
\]
Indeed, suppose first that \(\deg_{G_{r+i}}(v) = \sum_{w \in W_{r+i} \colon w \sim v} m((w,v)) = 2\) for every \(v \in V_{r+i}\). Since \(m(e) \ge 2\) for all \(e \in E_{r+i}\) by assumption (ii), each such vertex \(v\) must have exactly one neighbor \(w_i \in W_{r+i}\) with \(m((w_i,v))=2\). In particular, all edges in \(E_{r+i}\) have multiplicity two, and \(|W_{r+i}|=1\). Conversely, assume that \(W_{r+i} = \{w_i\}\) and that \(m((w_i, v)) = 2\) for every \(v \sim w_i\). Then each \(v \in V_{r+i}\) has a unique neighbor \(w_i \in W_{r+i}\), and \(\deg_{G_{r+i}}(v) = m((w_i,v))=2\).
\end{rmk}

\begin{proof}
It suffices to prove~\eqref{eq: estimate V disjoint case}, since~\eqref{eq: estimate W+V disjoint case} follows directly from~\eqref{eq: estimate V disjoint case} and the identity \(|W| = \sum_{i=1}^{r+s} |W_i|\), which holds by assumptions (i) and (ii). Since \(G\) is connected and the subsets \(W_1, \ldots, W_{r+s}\) are disjoint, the subgraphs \(G_1, \ldots, G_{r+s}\) are connected through vertices in \(V\), and their edge sets \(E_1, \ldots, E_{r+s}\) are disjoint. By assumption (i), we have \(\deg_{G_i}(v) \geq 2\) for every \(v \in V_i\) and \(i \in [r]\). Similarly, assumption (ii) implies that \(\deg_{G_{r+i}}(v) = \sum_{w \in W_{r+i} \colon w \sim v} m((w,v)) \geq 2\) for every \(v \in V_{r+i}\) and \(i \in [s]\). Consequently, the connected components \(G_1, \ldots, G_{r+s}\) are connected via vertices in \(v \in V\) satisfying \(\deg_{G}(v)=\sum_{i}\deg_{G_i}(v) \ge 4\). Assume that the components \(G_1, \ldots, G_{r+s}\) are connected through distinct vertices \(\tilde v_1, \ldots, \tilde v_p \in V\) for some \(p \geq 1\). For every \(i \in [p]\), let \(n(\tilde v_i)\) denote the number of components going through \(\tilde v_i\). Then
\begin{equation} \label{eq: estimate V proof}
|V| = \sum_{i=1}^{r+s} |V_i| - \sum_{i=1}^p (n(\tilde v_i)-1)  \leq \sum_{i=1}^{r+s} \frac{|E_i|}{2} - \sum_{i=1}^p (n( \tilde v_i)-1) = \frac{|E|}{2}- \sum_{i=1}^p (n(\tilde v_i)-1),
\end{equation}
where the last equality follows from the fact that \(E_1, \ldots, E_{r+s}\) are disjoint and the inequality follows from 
\begin{equation} \label{eq: estimate E_i disjoint case}
|E_i| = \sum_{v \in V_i} \sum_{w \in W_i} m((w,v))\textbf{1}_{\{(w,v) \in E_i\}} = \sum_{v \in V_i} \deg_{G_i}(v) \geq 2 |V_i|.
\end{equation}
We therefore need to estimate the term \(\sum_{i=1}^p (n(\tilde v_i)-1)\) in~\eqref{eq: estimate V proof}. We first consider the \(n(\tilde v_1)\) components going through \(\tilde v_1\). One of them contains at least one vertex among \(\tilde v_2, \ldots, \tilde v_p\), say \(\tilde v_2\). There are at most \(n(\tilde v_2) - 1\) new components going through \(\tilde v_2\), since at least one component is also connected to \(\tilde v_1\). Proceeding in this way, there are at most \(n(\tilde v_3) - 1\) new components connected to \(\tilde v_3\), since at least one component is already connected to \(\tilde v_1\) or \(\tilde v_2\). Recursively, we obtain that \(n(\tilde v_1) + \sum_{i=2}^p (n(\tilde v_i)-1) \geq r+s\), yielding 
\begin{equation} \label{eq: estimate mult v_i}
\sum_{i=1}^p (n(\tilde v_i)-1) \geq r + s -1.
\end{equation}
Combining~\eqref{eq: estimate V proof} and~\eqref{eq: estimate mult v_i}, we obtain
\[
|V| \leq  \frac{|E|}{2} - \sum_{i=1}^p (n(\tilde v_i)-1) \leq  \frac{|E|}{2} - (r+s) +1,
\] 
which proves~\eqref{eq: estimate V disjoint case}. \\

We now study the case of equality in~\eqref{eq: estimate V disjoint case}. From~\eqref{eq: estimate E_i disjoint case}, we first observe that for every \(1 \le i \le r+s\), \(|E_i| = |V_i|/2\) if and only if \(\deg_{G_i}(v) =2\) for every \(v \in V_i\). In particular, equality holds in~\eqref{eq: estimate V proof} if and only if conditions (a) and (b) are satisfied. It therefore remains to analyze the case of equality in~\eqref{eq: estimate mult v_i}.

We first assume that item (c) holds. We consider the \(n(\tilde v_1)\) components \((G_i)_{i \in \mathcal{I}_1}\) going through \(\tilde v_1\), where \(\mathcal{I}_1 = \{1 \leq i \leq r+s \colon \tilde{v}_1 \in G_i\}\) and \(n(\tilde v_1) = |\mathcal{I}_1|\). From item (c) it follows that \(|V_{G_i, G_j}| = | \{\tilde v_1\} | = 1\) for every \(i,j \in \mathcal{I}_1\), thus each of the vertices \(\tilde v_2, \ldots, \tilde v_p\) belongs to at most one subgraph \((G_i)_{i \in \mathcal{I}_1}\). We may assume without loss of generality that \(\tilde v_2\) belongs to one of the components \((G_i)_{i \in \mathcal{I}_1}\). Then, there are exactly \(n(\tilde v_2)\) components \((G_i)_{i \in \mathcal{I}_2}\) going through \(\tilde v_2\), where \(\mathcal{I}_2 = \{1 \leq i \leq r+s \colon \tilde v_2 \in G_i\}\) and \(n(\tilde v_2) = |\mathcal{I}_2|\). More precisely, there are exactly \(n(\tilde v_2)-1\) new components going through \(\tilde v_2\) since one component is already counted as it also goes through \(\tilde v_1\). Since \(|V_{G_i, G_j}| = | \{\tilde v_2\} | = 1\) for every \(i,j \in \mathcal{I}_2\) by item (c), the \(n(\tilde v_2) -1\) new components attached to \(\tilde v_2\) do not share other vertices. Moreover, since \(|V_{G_i,G_j,G_k}| \leq 2\) for every \(i,j,k \in \mathcal{I}_1 \cup \mathcal{I}_2\) by item (c), each vertex \(\tilde v_k\) for \(3 \le k \le p\) belongs to at most one subgraph among \((G_i)_{i \in \mathcal{I}_1 \cup \mathcal{I}_2}\). Otherwise, there are indices \(i\in \mathcal{I}_1\) and \(j \in \mathcal{I}_2 \backslash \mathcal{I}_1\) such that at least one vertex among \(\tilde v_3, \ldots, \tilde v_p\) belongs to \(V_i \cap V_j\), say \(\tilde v_3 \in V_i \cap V_j\). If we denote by \(G_k\), \(k \in \mathcal{I}_1 \cap \mathcal{I}_2\), the component containing both \(\tilde v_1\) and \(\tilde v_2\), then this implies that \(|V_{G_i,G_j,G_k}| = |\{\tilde v_1, \tilde v_2, \tilde v_3\}| = 3\), which is a contradiction to item (c). We may assume without loss of generality that \(\tilde v_3\) belongs to exactly one component among \((G_i)_{i \in \mathcal{I}_1 \cup \mathcal{I}_2}\). Then, there are exactly \(n(\tilde v_3) - 1\) new components going through \(\tilde v_3\). Iterating the argument, we see that condition (c) requires that at each time there are exactly \(n(\tilde v_i)-1\) new components going through \(\tilde v_i\). This implies equality in~\eqref{eq: estimate mult v_i}. 

We now assume by contradiction that (c) is not verified for some \(k \in [2,r+s]\). By definition, this means that there are distinct indices \(\ell_1, \ldots, \ell_k \in [r+s]\) such that the subgraphs \(G_{\ell_1}, \ldots, G_{\ell_k}\) form a cycle, i.e., \(|V_{G_{\ell_1}, \ldots, G_{\ell_k}}| =k\). We may assume without loss of generality that \(G_{\ell_1}\) and \(G_{\ell_2}\) are connected through \(\tilde v_1\), \(G_{\ell_2}\) and \(G_{\ell_3}\) through \(\tilde v_2\), and so on, and finally \(G_{\ell_k}\) and \(G_{\ell_1}\) are connected through \(\tilde v_k\). We then proceed as before. We first fix the \(n(\tilde v_1)\) components going through \(\tilde v_1\) and consider the \(n(\tilde v_2)-1\) new components going through \(\tilde v_2\), and so on, and finally we have the \(n(\tilde v_k)-2\) new components going through \(\tilde v_k\) (there are exactly \(n(\tilde v_k)-2\) new components since \(G_{\ell_k}\) and \(G_{\ell_1}\) have already been counted). Moreover, there are at most \(n(\tilde v_i)-1\) new components going through \(\tilde v_i\) for every \(i \in [k+1,p]\). This shows that
\[
n(\tilde v_1) + \sum_{i=1}^{k-1} (n(\tilde v_i) - 1) + (n(\tilde v_k) - 2) + \sum_{i = k+1}^p (n(\tilde v_i) -1) \geq r + s ,
\]
that is, \(\sum_{i=1}^p (n(\tilde v_i)-1) \geq r+s > r+s-1\), thus yielding a strict inequality in~\eqref{eq: estimate mult v_i}.
\end{proof}

We now address the case where \(W_1, \ldots, W_r\) and \(W_{r+1}, \ldots, W_{r+s}\) are not necessarily disjoint. Recall that by definition, the families \(\{W_1, \ldots, W_r\}\) and \(\{W_{r+1}, \ldots, W_{r+s}\}\) are disjoint from each other. Analogously to the definition of the set of common vertices in \(V\), we define the set of common vertices in \(W\) among the subgraphs \(G_1, \ldots, G_r\) as follows. For every \(k \in [2,r]\) and every choice of indices \(1 \leq \ell_1 < \cdots < \ell_k \leq r\), define
\[
W_{G_{\ell_1}, \ldots, G_{\ell_k}} \coloneqq \left \{ w \in W \colon \exists \, 1 \le i < j \le k \enspace \textnormal{such that} \enspace w \in W_{\ell_i} \cap W_{\ell_j} \right \}.
\]
This definition extends naturally to the set of common vertices in \(W\) among \(G_{r+1}, \ldots, G_{r+s}\). We now describe a merging procedure that replaces non-disjoint subsets by disjoint ones while preserving connectivity.

\begin{defn} [Merging procedure] \label{def: merging operation}
Let \(W_1, \ldots, W_r \subseteq W\) be (not necessarily disjoint) subsets, and let \(\tilde{w}_1, \ldots, \tilde{w}_q\) be the vertices of \(W\) belonging to at least two of these subsets (with the convention that \(q=0\) if \(W_1,\ldots,W_r\) are disjoint). Define
\[
\mathcal{I} \coloneqq \{i \in [r] \colon \exists \, j \in [q] \enspace \textnormal{such that} \enspace \tilde{w}_j \in W_i\},
\]
so that \(\mathcal{I}\) indexes the subsets containing at least one shared vertex. By definition, the family \(\{ W_i \colon i \in \mathcal{I}^\textnormal{c} \}\) is pairwise disjoint and disjoint from \(\cup_{i \in \mathcal{I}} W_i\). For every \(j \in [q]\), set
\[
\mathcal{I}_j \coloneqq \{i \in [r] \colon \tilde{w}_j \in W_i\}
\]
the collection of subsets sharing the vertex \(\tilde{w}_j\). Note that \(\mathcal{I} = \cup_{j=1}^q \mathcal{I}_j\). 

Introduce the intersection graph \(H = (\mathcal{I}, E_H)\), whose vertices are the indices in \(\mathcal{I}\) and where
\[
(i,j) \in E_H \quad \Longleftrightarrow \quad W_i \cap W_j \ne \emptyset.
\]
Let \(\widetilde{\mathcal{I}}_1, \ldots, \widetilde{\mathcal{I}}_C\) be the connected components of \(H\). Then, for every \(\ell \in [C]\), define the merged subset
\[
\widetilde{W}_\ell \coloneqq \cup_{i \in \widetilde{\mathcal{I}}_\ell} W_i.
\]
By construction, \(\mathcal{I} = \sqcup_{\ell=1}^C \widetilde{\mathcal{I}}_\ell\) and
\[
\cup_{i \in \mathcal{I}} W_i = \sqcup_{\ell=1}^C \widetilde{W}_\ell.
\]
Together with \(\{W_i \colon i \in \mathcal{I}^{\mathrm{c}}\}\), this yields \(C + |\mathcal{I}^{\mathrm{c}}|\) disjoint subsets and the decomposition
\[
\cup_{i=1}^r W_i = \left( \sqcup_{\ell=1}^C \widetilde{W}_\ell \right) \sqcup \left( \sqcup_{i \in \mathcal{I}^{\mathrm{c}}} W_i \right).
\]
Finally, for every \(\ell \in [C]\), merge the subgraphs \(\{G(W_i) \colon i \in \widetilde{\mathcal{I}}_\ell\}\) by identifying repeated copies of shared vertices \(\tilde{w}_\ell\) (and their adjacent vertices in \(V\)) and removing any duplicated vertices or edges. Denote the resulting connected graphs by \(\widetilde{G}_\ell = G(\widetilde{W}_\ell)\).
\end{defn}

The following example refers to Figure~\ref{fig1} and aims to clarify the previous definition.

\begin{ex} \label{ex3}
Consider the connected bipartite graph of Figure~\ref{subfig1}, along with three different collections of subsets of \(W\) shown in Figures~\ref{subfig2}-\ref{subfig4}. We note the following. 
\begin{itemize}
\item Figure~\ref{subfig2}: The subsets \(W_1\) and \(W_2\) are disjoint and satisfy \(W = W_1 \sqcup W_2\), thus \(\mathcal{I} = \emptyset\). 
\item Figure~\ref{subfig3}: The subset \(W_1\) is disjoint from both \(W_3\) and \(W_4\), while \(W_3 \cap W_4 = \{w_5\}\). Thus, \(\mathcal{I} = \{3,4\}\). Applying the merging procedure from Definition~\ref{def: merging operation} to \(W_3\) and \(W_4\) results in the subset \(\widetilde{W}_1 = \{w_5, \ldots w_{12}\}\), which corresponds to \(W_2\) from Figure~\ref{subfig2}. Moreover, \(W_1\) and \(\widetilde{W}_1\) are disjoint, satisfying \(W = W_1 \sqcup \widetilde{W}_1\). Merging the subgraphs \(G_3 = G(W_3)\) and \(G_4 = G(W_4)\) removes the repeated copy of \(w_5\), as well as the repeated copy of the vertices \(v_2, v_3, v_4\), and \(v_5\) and of their associated edges. The resulting graph \(\widetilde{G}_1\) corresponds to \(G_2\) from Figure~\ref{subfig2}. 
\item Figure~\ref{subfig4}: The subsets \(W_5\) and \(W_6\) are disjoint from \(W_2\), while \(W_5 \cap W_6 = \{w_4\}\), leading to \(\mathcal{I} = \{5,6\}\). Applying the merging procedure to \(W_5\) and \(W_6\) gives \(\widetilde{W}_1 = \{w_1, \ldots, w_4\}\), which corresponds to \(W_1\) from Figure~\ref{subfig2}. Merging the subgraphs \(G_5\) and \(G_6\) removes the duplicate copy of \(w_4\) as well as the repeated copy of the vertices \(v_2, v_3, v_4\) and of their associated edges. The resulting graph \(\widetilde{G}_1\) thus corresponds to the graph \(G_1\) in Figure~\ref{subfig2}.
\end{itemize}
\end{ex}

For the following result, fix integers \(q,q' \ge 0\) and let \(\tilde{w}_1, \ldots, \tilde{w}_{q}\) denote the vertices of \(W\) that belong to two or more subsets among \(W_1,\ldots, W_r\). Similarly, let \(w_1', \ldots, w'_{q'}\) denote the vertices that belong to two or more subsets among \(W_{r+1},\ldots, W_{r+s}\). Let \(\mathcal{I} \subseteq [r]\) be as in Definition~\ref{def: merging operation}, and define
\[
\mathcal{J} = \left \{ i \in \{r+1, \ldots, r+s\} \colon  \exists \, j \in [q'] \enspace \textnormal{such that} \enspace w_j' \in  W_i \right \}.
\]
Applying the merging procedure of Definition~\ref{def: merging operation} to the subgraphs that share vertices in \(\cup_{i \in \mathcal{I}} W_i\) produces \(C\) disjoint connected subgraphs \(\{\widetilde{G}_1, \ldots, \widetilde{G}_C\} = \{ G(\widetilde{W}_i), 1 \leq i \leq C\}\). Similarly, applying the merging procedure to the subgraphs that share vertices in \(\cup_{i \in \mathcal{J}} W_i\) results in \(C'\) disjoint connected subgraphs \(\{G'_1, \ldots, G'_{C'}\} = \{G(W'_i) \colon 1 \leq i \leq C'\}\). Altogether, we obtain the family of disjoint connected subgraphs
\[
\{\overline{G}_i \colon 1\le i\le C + C' + |\mathcal{I}^\textnormal{c}| + |\mathcal{J}^\textnormal{c}|\} = \{\widetilde{G}_1,\ldots, \widetilde{G}_{C}, (G_i)_{i \in \mathcal{I}^\textnormal{c}}, G'_1,\ldots, G'_{C'} ,(G_i)_{i \in \mathcal{J}^\textnormal{c}}\}.
\]

\begin{lem} \label{lem: card not disjoint case}
Let \(G = (W \cup V, E)\) and let \(G_i = G(W_i)\), \(1 \le i \le r+s\), be connected bipartite graphs such that \(G = \cup_{i=1}^r G_i \sqcup \cup_{i=1}^s G_{r+i}\). Assume
\begin{itemize}
\item[(i)] \(W_1,\ldots, W_r\) satisfy conditions (A)-(C);
\item[(ii)] condition (D) holds, i.e., for every \(i \in [s]\) and every \(e \in E_{r+i}\), \(m(e)\) is even.
\end{itemize}
Then,
\begin{equation} \label{eq: estimate W+V not disjoint case}
|W| + |V| - \frac{|E|}{2} -1 \leq \sum_{i=1}^{r+s} (|W_i| -1) .
\end{equation}
Equality in~\eqref{eq: estimate W+V not disjoint case} holds if and only if 
\begin{itemize}
\item[(a)] for every \(2 \leq k \leq |\mathcal{I}|\) and distinct \(\ell_1, \ldots, \ell_k \in \mathcal{I}\), \(|W_{G_{\ell_1}, \ldots, G_{\ell_k}}| \leq k-1\);
\item[(b)] for every \(i \in [C]\) and \(v \in \widetilde{V}_i\), \(\deg_{\widetilde{G}_i}(v) = 2\), and for every \(i \in \mathcal{I}^\textnormal{c}\) and \(v \in V_i\), \(\deg_{G_i}(v) = 2\);
\item[(c)] for every \(i \in [C']\) and \(v \in V'_i\), \(\deg_{G'_i}(v) = 2\), and for every \(i \in \mathcal{J}^\textnormal{c}\) and \(v \in V_i\), \(\deg_{G_i}(v) = 2\);
\item[(d)] for every \(2 \leq k \leq C + C' + |\mathcal{I}^\textnormal{c}| +|\mathcal{J}^\textnormal{c}|\) and distinct \(\ell_1, \ldots, \ell_k\), \(|V_{\overline{G}_{\ell_1}, \ldots, \overline{G}_{\ell_k}}| \leq k-1\).
\end{itemize}
\end{lem}

\begin{proof} 
Consider first the subsets \((W_k)_{k \in \mathcal{I}}\). The merging procedure described by Definition~\ref{def: merging operation} yields \(C\) disjoint subsets \(\widetilde{W}_1,\ldots, \widetilde{W}_C\) such that \(\cup_{k \in \mathcal{I}} W_k = \sqcup_{i=1}^C \widetilde{W}_i\). Fix \(i \in [C]\) and recall that the subset \(\widetilde W_i\) is obtained by merging the subsets \(W_k\) for \(k \in \widetilde{\mathcal{I}}_i\) and by removing any repeated copy of \(\tilde{w}_j\) for every \(j \in [q]\) such that \(\mathcal{I}_j \cap \widetilde{\mathcal{I}}_i \neq \emptyset\). This implies that
\[
|\widetilde{W}_i| = \sum_{k \in \widetilde{\mathcal{I}}_i} |W_k| - \sum_{j=1}^q (|\mathcal{I}_j \cap \widetilde{\mathcal{I}}_i|-1)\mathbf{1}_{\{\mathcal{I}_j \cap \widetilde{\mathcal{I}}_i \neq \emptyset\}}.
\]
Furthermore, according to Definition~\ref{def: merging operation}, if there exists \(j \in [q]\) such that \(\mathcal{I}_j \cap \widetilde{\mathcal{I}}_i \neq \emptyset\), then \(\mathcal{I}_j \subseteq \widetilde{\mathcal{I}}_i\), yielding 
\begin{equation} \label{eq: inequality for W tilde}
|\widetilde{W}_i| = \sum_{k \in \widetilde{\mathcal{I}}_i} |W_k| - \sum_{j=1}^q (|\mathcal{I}_j|-1)\mathbf{1}_{\{\mathcal{I}_j \cap \widetilde{\mathcal{I}}_i \neq \emptyset\}}.
\end{equation}
We therefore need to estimate the sum \( \sum_{j=1}^q (|\mathcal{I}_j|-1)\mathbf{1}_{\{\mathcal{I}_j \cap \widetilde{\mathcal{I}}_i \neq \emptyset\}}\). We can proceed in a similar way as done in the proof of Lemma~\ref{lem: card disjoint case}. Let \(j_1, \ldots, j_m \in [q]\) denote the indices such that \(\mathcal{I}_{j_k} \cap \widetilde{\mathcal{I}}_i \neq \emptyset\). That is, \(\tilde{w}_{j_1}, \ldots, \tilde{w}_{j_m}\) are the vertices in \((\tilde{w}_j)_{j\in [q]}\) that belong to \(\widetilde{W}_i\). We first consider the \(|\mathcal{I}_{j_1}|\) subgraphs sharing the vertex \(\tilde{w}_{j_1}\). One of them contains at least one vertex among \(\tilde{w}_{j_2}, \ldots, \tilde{w}_{j_m}\), say \(\tilde{w}_{j_2}\). Then there are at most \(|\mathcal{I}_{j_2}|-1\) new components going through \(\tilde{w}_{j_2}\), since at least one component is also connected to \(\tilde{w}_{j_1}\). Similarly, there are at most \(|\mathcal{I}_{j_3}|-1\) new components going through \(\tilde{w}_{j_3}\), since at least one component is already connected to \(\tilde{w}_{j_1}\) or \(\tilde{w}_{j_2}\). Iterating this argument for each \(\tilde{w}_{j_k}\) for \(k \in [m]\), we obtain that 
\[
|\mathcal{I}_{j_1}| + \sum_{k=2}^m (|\mathcal{I}_{j_k}| -1) \ge |\widetilde{\mathcal{I}}_i|,
\]
leading to
\begin{equation}\label{eq: bound card I}
\sum_{k=1}^m (|\mathcal{I}_{j_k}| -1) \ge |\widetilde{\mathcal{I}}_i| -1,
\end{equation}
Combining~\eqref{eq: inequality for W tilde} and~\eqref{eq: bound card I} yields
\[
|\widetilde{W}_i| \le \sum_{k \in \widetilde{\mathcal{I}}_i} |W_k| - |\widetilde{\mathcal{I}}_i| + 1 =  \sum_{k \in \widetilde{\mathcal{I}}_i} (|W_k|-1) + 1,
\]
so that
\begin{equation} \label{eq: inequality for W tilde 2}
\sum_{i=1}^C |\widetilde{W}_i| \leq \sum_{i=1}^C \sum_{k \in \widetilde{\mathcal{I}}_i} (|W_k| -1) + C = \sum_{k \in \mathcal{I}} (|W_k|-1) + C.
\end{equation}
Now, consider the subsets \((W_k)_{k \in \mathcal{J}}\). By the merging procedure described in Definition~\ref{def: merging operation}, we obtain \(C'\) disjoint subsets \(W_1',\ldots, W'_{C'}\) such that \(\cup_{k \in  \mathcal{J}} W_k = \sqcup_{i=1}^{C'} W'_i\). Following the previous arguments, we deduce that 
\begin{equation}  \label{eq: inequality for W prime 2}
\sum_{i=1}^{C'} |W'_i|  \leq \sum_{k \in  \mathcal{J}} (|W_k| - 1) + C'.
\end{equation}
The subsets \(\widetilde{W}_1,\ldots, \widetilde{W}_C, (W_i)_{i \in \mathcal{I}^\textnormal{c}}\) are disjoint and satisfy (A)-(C). Similarly, \(W'_1,\ldots, W'_{C'}, (W_i)_{i \in \mathcal{J}^\textnormal{c}}\) are disjoint, and the multiplicity of each edge in \(E_i'\) for \(i \in [C']\) and in \(E_i\) for \(i \in \mathcal{J}^\textnormal{c}\) is even, as the merging operation preserves edge multiplicities. Applying Lemma~\ref{lem: card disjoint case} we therefore obtain 
\begin{equation} \label{eq: inequality for W+V}
\begin{split}
|W| + |V| & \leq \sum_{i=1}^C (|\widetilde{W}_i| -1) + \sum_{i \in \mathcal{I}^\text{c}} (|W_i| -1) +  \sum_{i=1}^{C'} (|W'_i| -1) + \sum_{i \in \mathcal{J}^\text{c}} (|W_i| -1) + \frac{|E|}{2} + 1\\
& \leq \sum_{i \in \mathcal{I}} (|W_i| -1) + \sum_{i \in \mathcal{I}^\textnormal{c}} (|W_i| -1) + \sum_{i \in \mathcal{J}} (|W_i| -1) + \sum_{i \in \mathcal{J}^\textnormal{c}} (|W_i| -1) + \frac{|E|}{2} + 1\\
& = \sum_{i=1}^{r+s} (|W_i| -1) + \frac{|E|}{2} + 1.
\end{split}
\end{equation}
Here, we used~\eqref{eq: inequality for W tilde 2} and~\eqref{eq: inequality for W prime 2} for the second inequality, along with the fact that \(\sum_{i=1}^r |W_i| = \sum_{i \in \mathcal{I}} |W_i| +  \sum_{i \in \mathcal{I}^\textnormal{c}} |W_i|\) and \(\sum_{i=1}^s |W_{r+i}| = \sum_{i \in \mathcal{J}} |W_i| +  \sum_{i \in \mathcal{J}^\textnormal{c}} |W_i|\) for the third equality. This shows the inequality~\eqref{eq: estimate W+V not disjoint case}.

Now, we study the case of equality in~\eqref{eq: estimate W+V not disjoint case}. According to Lemma~\ref{lem: card disjoint case}, we have equality in the first line of~\eqref{eq: inequality for W+V} if and only if the following conditions are satisfied:
\begin{enumerate}
\item[(b)] for every \(i \in \mathcal{I}^\text{c}\) and \(v \in V_i\), \(\deg_{G_i}(v) = 2\), and for every \(i \in [C]\) and \(v \in \widetilde{V}_i\), \(\deg_{\widetilde{G}_i}(v) = 2\);
\item[(c)] for every \(i \in \mathcal{J}^\text{c}\) and \(v \in V_i\), \(\deg_{G_i}(v) = 2\), and for every \(i \in [C']\) and \(v \in V_i'\), \(\deg_{G'_i}(v) = 2\);
\item[(d)] for every \(2 \leq k \leq C + C' + |\mathcal{I}^\text{c}| + |\mathcal{J}^\text{c}|\) and every distinct indices \(\ell_1, \ldots, \ell_k\), we have that \(|V_{\overline{G}_{\ell_1}, \ldots, \overline{G}_{\ell_k}}| \leq k-1\).
\end{enumerate}
Therefore, equality in~\eqref{eq: estimate W+V not disjoint case} holds if and only if equality also holds in the second line of~\eqref{eq: inequality for W+V}, which is equivalent to equality in~\eqref{eq: bound card I}. It thus remains to show that equality in~\eqref{eq: bound card I} holds if and only if condition (a) is satisfied. Note that, by Definition~\ref{def: merging operation}, the index set \(\mathcal{I}\) decomposes as \(\mathcal{I} = \sqcup_{i=1}^C \widetilde{\mathcal{I}}_i\). Condition (a) is therefore equivalent to requiring that, for every \(i \in [C]\), every \(2 \le k \le | \widetilde{\mathcal{I}}_i|\), and every collection of distinct indices \(\ell_1, \ldots, \ell_k \in \widetilde{\mathcal{I}}_i\), the inequality \(|W_{G_{\ell_1}, \ldots, G_{\ell_k}}| \le k-1\) holds. The argument then follows the same lines as in the proof of Lemma~\ref{lem: card disjoint case}.
\end{proof}

We now introduce the class \(\mathcal{W}_r\), consisting of \(r\)-tuples of subsets of \(W\), which will be used to state the key combinatorial identity proved in this section.

\begin{defn}[Class \(\mathcal{W}_r\) of subsets] \label{def: set W_K}
Let \(G = (W \cup V, E)\) be a finite, connected bipartite multigraph. Consider subsets \(W_1, \ldots, W_r \subseteq W\) with \(r \ge 1\) such that \(\cup_{i=1}^r W_i \subseteq W\). For every \(i \in [r]\), let \(G_i = G(W_i) = (W_i \cup V_i, E_i)\) denote the corresponding induced subgraph (see Definition~\ref{def: subgraphs}). Let \(s = |W| - |\cup_{i=1}^r W_i|\) and denote by \(w_1, \ldots, w_s\) the vertices in \(W \backslash \cup_{i=1}^r W_i\). Let \(\tilde{w}_1, \ldots, \tilde{w}_q\) be the vertices of \(W\) that belong to more than one subset \(W_i\), and define
\[
\mathcal{I} \coloneqq \{i \in [r] \colon \exists \, j \in [q] \enspace \textnormal{such that} \enspace \tilde{w}_j \in W_i\}.
\]
Applying the merging procedure of Definition~\ref{def: merging operation} to the subgraphs that share vertices in \(\cup_{i \in \mathcal{I}} W_i\), we obtain \(C\) disjoint connected subgraphs \(\{\widetilde{G}_i = G(\widetilde{W}_i) \colon 1 \leq i \leq C\}\). 

We say that \(\{W_1, \ldots, W_r\}\) belongs to the class \(\mathcal{W}_r\) if the following conditions hold:
\begin{enumerate}
\item every subgraph \(G_i\) is connected;
\item for every \(e \in E \backslash \cup_{i=1}^r E_i\), \(m(e) = 2\); 
\item for every \(i \in [r]\) and \(w \in W_i\), \(\deg_{G_i}(w) = \deg(w) \ge 2\);
\item for every \(v \in \cup_{i=1}^r V_i\), there exists at least one \(V_i\) such that \(\deg_{G_i}(v) \geq 2\);
\item for every \(i \in \mathcal{I}^\textnormal{c}\) and \(v \in V_i\), \(\deg_{G_i}(v) = 2\), and for every \(i \in [C]\) and \(v \in \widetilde{V}_i\), \(\deg_{\widetilde{G}_i}(v) = 2\);
\item for every \(2 \leq k \leq |\mathcal{I}|\) and distinct \(\ell_1, \ldots, \ell_k \in \mathcal{I}\), \(|W_{G_{\ell_1}, \ldots, G_{\ell_k}} | \leq k-1\);
\item for any \(2 \leq k \leq C + |\mathcal{I}^\textnormal{c}| + s\) and distinct \(\ell_1, \ldots, \ell_k \in \{1, \ldots, C +|\mathcal{I}^\textnormal{c}| + s\}\), \(|V_{\overline{G}_{\ell_1}, \ldots, \overline{G}_{\ell_k}}| \leq k -1\), where \(\{\overline{G}_i \colon 1 \leq i \leq C + |\mathcal{I}^\textnormal{c}| + s\}\) denotes the family of all disjoint connected subgraphs obtained after merging, namely \(\{\widetilde{G}_1,\ldots, \widetilde{G}_C, (G_i)_{i \in \mathcal{I}^\textnormal{c}}, G(\{w_1\}),\ldots, G(\{w_s\})\}\).
\end{enumerate}
\end{defn}

\begin{rmk}
If \(\{W_1, \ldots, W_r\} \in \mathcal{W}_r\), then any two subsets \(W_i\) and \(W_j\) with \(|W_i|, |W_j| \ge 2\) must be distinct, although they may have a nonempty intersection. Indeed, if \(W_i = W_j\), then
\[
|W_{G_i,G_j}| = |W_i| \ge 2,
\]
which contradicts Condition (6) of Definition~\ref{def: set W_K} for \(k=2\).
\end{rmk}

\begin{ex}[Example~\ref{ex3} continued]
Consider the connected bipartite graph shown in Figure~\ref{subfig1}, together with three different collections of subsets of $W$ illustrated in Figures~\ref{subfig2}-\ref{subfig4}. 
\begin{itemize}
\item Figure~\ref{subfig2}: The subsets \(W_1\) and \(W_2\) belong to \(\mathcal{W}_2\).
\item Figure~\ref{subfig3}: The subsets \(W_1, W_3\), and \(W_4\) belong to \(\mathcal{W}_3\). 
\item Figure~\ref{subfig4}: The subsets \(W_5, W_6\), and \(W_2\) do not belong to \(\mathcal{W}_3\). Indeed, condition (4) of Definition~\ref{def: set W_K} is not satisfied, since \(v_2 \in V_5 \cap V_6\) and \(\deg_{G_5}(v_2) = \deg_{G_6}(v_2) = 1\). 
\end{itemize}
\end{ex}

\begin{ex}
Figure~\ref{fig2} provides a second example of a connected bipartite graph \(G = (W \cup V, E)\) for which there exists at least one collection of subsets \(\{W_1, \ldots, W_K\} \in \mathcal{W}_K\) for some integer \(K \ge 1\). In particular, there are two such collections: the subset \(W_1 = \{w_2, \ldots, w_6\}\), which belongs to \(\mathcal{W}_1\), and the subsets \(W_2 = \{w_2, w_3, w_4, w_5\}\) and \(W_3 = \{w_3,w_6\}\), which belong to \(\mathcal{W}_2\). Note that, in this case, the graph obtained by merging the subgraphs \(G(W_2)\) and \(G(W_3)\) coincides with \(G(W_1)\). 
\end{ex}

The following result is a direct consequence of the results of this section, namely Lemmas~\ref{lem: necessary cond}, ~\ref{lem: card disjoint case}, and~\ref{lem: card not disjoint case}.

\begin{prop} \label{prop: combinatorics}
Let \(G = (W \cup V, E)\) be a finite, connected bipartite multigraph, and let \(W_1,\ldots, W_r\) be subsets of \(W\) such that the associated subgraphs \(G_1, \ldots, G_r\) are connected. 
\begin{itemize}
\item[(a)] If \(\{W_1, \ldots, W_r\} \notin \mathcal{W}_r\), then either \(P_{(W_i, \eta_i)_{i=1}^r} (\boldsymbol{\gamma})\) is invariant under \(T_{E_0}\) for some \(E_0 \subset E\) with odd cardinality, or 
\[
\sum_{i=1}^r (|W_i| - 1) > |W| + |V| - \frac{|E|}{2} - 1.
\]
\item[(b)] If \(\{W_1, \ldots, W_r\}  \in \mathcal{W}_r\), then
\[
\sum_{i=1}^r (|W_i| - 1) = |W| + |V| - \frac{|E|}{2} - 1.
\]
\end{itemize}
\end{prop}

The final result of this section characterizes the connected bipartite multigraphs \(G = (W \cup V, E)\) with all vertices of even degree for which there exists at least one collection of subsets \(W_1, \ldots, W_K \subseteq W\) belonging to \(\mathcal{W}_K\) for some \(K \ge 1\) and satisfying \(|W_k| \ge 2\) for every \(k \in [K]\). This characterization is used in Section~\ref{section: convergence traffic} to prove Lemma~\ref{lem: main estimate}. The following lemma shows that these graphs are admissible graphs with \(R \ge 1\) blocks (see Definition~\ref{def: admissible graph}), and that the corresponding subsets \(W_1, \ldots, W_K\) form admissible decompositions (see Definition~\ref{def: admissible decomposition}). 

\begin{lem} \label{lem: equiv def}
Let \(G= (W \cup V, E)\) be a connected bipartite multigraph in which all vertices have even degree. Then the following are equivalent:
\begin{itemize}
\item[(i)] There exists at least one collection of subsets \(W_1, \ldots, W_K \subseteq W\) belonging to \(\mathcal{W}_K\) for some integer \(K \ge 1\), with \(|W_k| \ge 2\) for all \(k \in [K]\);
\item[(ii)] \(G\) is an admissible graph with \(R = \frac{|E|}{2} + 1 - |V| - S \ge 1\) blocks \(B_i = G(W_{B_i})\), each satisfying \(|W_{B_i}| \ge 2\), and there exist disjoint index sets \(I_1, \ldots, I_R \subseteq [K]\) with \(K = \sum_{i=1}^R |I_i|\) such that, for every \(i \in [R]\), the family \(\{W_\ell \colon \ell \in I_i\}\) is an admissible decomposition of \(W_{B_i}\). 
\end{itemize}
\end{lem}

\begin{proof}
We first show that (i) implies (ii). Suppose there exists a collection of subsets \(W_1, \ldots, W_K \subseteq W\) with \(|W_k| \ge 2\) such that \(\{W_1, \ldots, W_K\} \in \mathcal{W}_K\) for some integer \(K \ge 1\). Applying the merging procedure of Definition~\ref{def: merging operation} to the subsets \(W_1,\ldots, W_K\), and using the notation of Definition~\ref{def: set W_K}, we obtain disjoint subsets \(\widetilde{W}_1, \ldots, \widetilde{W}_C, (W_i)_{i \in \mathcal{I}^{\text{c}}}\), each of cardinality at least two, such that \(\cup_{i=1}^K W_i = \left( \sqcup_{i=1}^C \widetilde{W}_i\right) \sqcup \left(\sqcup_{i \in \mathcal{I}^\textnormal{c}} W_i\right)\). Let \(R \coloneqq C + |\mathcal{I}^\text{c}|\), and let \(S\) denote the number of remaining vertices in \(W \backslash \cup_{i=1}^K W_i\). Thus \(G\) decomposes into \(R+S\) disjoint subgraphs \(\overline{G}_1, \ldots, \overline{G}_{R+S}\) connected through vertices in \(V\). By Lemma~\ref{lem: card disjoint case}, in particular~\eqref{eq: estimate V disjoint case}, we have 
\[
R + S = \frac{|E|}{2} + 1 - |V|. 
\]
Moreover, by condition (7) of Definition~\ref{def: set W_K}, the graph \(G\) is a block tree with \(R+S\) blocks. More precisely, there are \(R \ge 1\) blocks \(B_i = G(W_{B_i})\), where \(W_{B_i} = \overline{W}_i\) is given by
\[
\{\overline{W}_i \colon i \in [R]\} = \{\widetilde{W}_1, \ldots, \widetilde{W}_C\} \cup \{W_i \colon i \in \mathcal{I}^{\text{c}}\}, 
\]
each satisfying \(|W_{B_i}| \ge 2\). There are \(S\) further blocks, each consisting of exactly one vertex in \( W \backslash \cup_{i=1}^K W_i\). The degree condition in Definition~\ref{def: admissible graph} follows from items (2), (3), and (5) of Definition~\ref{def: set W_K}, together with the additional assumption that all vertices in \(W\cup V\) have even degree. Hence, \(G\) is an admissible graph. It remains to show that each block \(B_i\), \(i \in [R]\), admits an admissible decomposition of \(W_{B_i}\). If \(W_{B_i} = W_j\) for some \(j \in \mathcal{I}^\text{c}\), then \(\{W_j\} \in \mathcal{A}_1(W_{B_i})\) by definition; in this case, set \(I_i = \{j\}\). Otherwise, \(W_{B_i} = \widetilde{W}_\ell\) for some \(\ell \in [C]\). Let \(\widetilde{\mathcal{I}}_\ell \subseteq [K]\) be the index set such that \(\widetilde{W}_\ell = \cup_{k \in \widetilde{\mathcal{I}}_\ell} W_k\) (see Definition \ref{def: merging operation}). By construction, the index sets \(\widetilde{\mathcal{I}}_1, \ldots, \widetilde{\mathcal{I}}_C\) are disjoint and satisfy \(\mathcal{I} = \sqcup_{\ell=1}^C \widetilde{\mathcal{I}}_\ell\). The collection \((W_k)_{k \in \widetilde{\mathcal{I}}_\ell}\) is an admissible decomposition of \(W_{B_i} = \widetilde{W}_\ell\): items (a) and (b) of Definition~\ref{def: admissible decomposition} follow from the merging procedure, item (c) from assumption (i), item (d) from condition (1), item (e) from condition (4), and item (f) from condition (6) of Definition~\ref{def: set W_K}. In this case, set \(I_i = \widetilde{\mathcal{I}}_\ell\). We have therefore constructed disjoint index sets \(I_1, \ldots, I_R \subseteq [K]\) such that \(K=\sum_{i=1}^R |I_i| = |\mathcal{I}^\text{c}| + \sum_{\ell=1}^C |\widetilde{\mathcal{I}}_\ell|\), and for each \(i \in [R]\), \(\{W_\ell \colon \ell \in I_i\} \in \mathcal{A}_{|I_i|} (W_{B_i})\). This proves (ii).

We now show that (ii) implies (i). Assume that \(G\) is an admissible graph with \(R+S\) blocks, where \(R \ge 1\) blocks \(B_i = G(W_{B_i})\) satisfy \(|W_{B_i}| \ge 2\), and the remaining \(S\) blocks contain exactly one vertex of \(W\). Moreover, for every \(i\in [R]\), assume that there exists an admissible decomposition of \(W_{B_i}\), that is, \(\{W_1^i, \ldots, W_{|I_i|}^i \} \in \mathcal{A}_{|I_i|} (W_{B_i})\). Set \(K \coloneqq \sum_{i=1}^R |I_i|\), and consider the collection of subsets \(\{W_j^i \colon j \in [|I_i|], i \in [R]\}\). We claim that this family belongs to \(\mathcal{W}_K\) and satisfies \(|W_j^i| \ge 2\) for all \(i,j\). The size condition \(|W_j^i| \ge 2\) follows immediately from item (c) of Definition~\ref{def: admissible decomposition}. We now verify conditions (1)-(7) of Definition~\ref{def: set W_K}. Conditions (1), (4), and (6) follow directly from items (d), (e), and (f) of Definition~\ref{def: admissible decomposition}. Condition (3) holds since \(G\) is admissible in the sense of Definition~\ref{def: admissible graph}. Condition (2) is also satisfied by Definition~\ref{def: admissible graph}: in each of the \(S\) blocks consisting of a single vertex of \(W\), every vertex in \(V\) has degree \(2\), and hence all edges in such blocks have multiplicity \(2\). It remains to verify conditions (5) and (7). For condition (5), note that for every \(i \in [R]\) and \(j \in [|I_i|]\), \(W_j^i \cap \cup_{k \neq j} W_k^i \neq \emptyset \). Consequently, merging \(W_1^i, \ldots, W_{|I_i|}^i\) according to Definition~\ref{def: merging operation} yields the connected component \(W_{B_i}\). Condition (5) is therefore satisfied according to Definition~\ref{def: admissible graph}. Finally, condition (7) follows from the fact that \(G\) is a block tree. Therefore, \(\{W_j^i \colon j \in [|I_i|], i \in [R]\} \in \mathcal{W}_K\), which proves (i).
\end{proof}

\section{Convergence of matrix moments} \label{section: cycle}

This section is devoted to the proof of Theorem~\ref{main1}. We begin by proving the convergence in expectation of the moments of \(M=Y_m Y_m^\top\) and by computing the corresponding limit. As discussed in Subsection~\ref{subsection: outline proof}, the expected \(k\)th moment of \(M\) corresponds to the traffic trace associated to the test graph \(T_{\text{cycle}} = (W \cup V, E, Y_m)\), where \(G = (W \cup V, E)\) denotes the simple bipartite cycle of length \(2k\). Our approach thus relies on applying Proposition~\ref{main3}.

Throughout this section, let \(G = (W \cup V, E)\) denote the non-oriented simple bipartite cycle of length \(2k\), with edges \(e = (w, v)\) connected vertices from \(W\) to \(V\). We define \(W = \{w_1, \ldots, w_k\}\) and \(V = \{v_1, \ldots, v_k\}\) as the vertex sets and assume the vertices are labeled in cyclic order such that \(w_i < v_i < w_{i+1}\) for every \(i \in [k]\), with the convention that \(w_{k+1} = w_1\). The set of unordered edges \(E\) is then labeled by pairs of cyclically adjacent vertices, i.e., \(E = \cup_{i=1}^k \{(w_i,v_i),(w_{i+1}, v_i)\}\). See Figure~\ref{fig3} for an illustration. According to Proposition~\ref{main3}, we have 
\begin{equation} \label{eq: moments}
\lim_{m, n, p \to \infty} \E \left [ \frac{1}{p} \tr M^k \right ] = \lim_{m, n, p \to \infty} \tau_{p,m,n} \left [T_{\text{cycle}} \right ]  =  \sum_{\pi \in \mathcal{P}(V)}  \sum_{\mu \in \mathcal{P}(W)} \tau^0_{G^{\pi,\mu}},
\end{equation}
where \(G^{\pi,\mu}\) is the graph obtained from \(G\) by identifying vertices of \(V\) which belong to the same block of \(\pi\) and vertices of \(W\) which belong to the same block of \(\mu\). We do not identify edges, so \(G^{\pi,\mu}\) may have multiple edges. The limiting injective trace \(\tau^0_{G^{\pi,\mu}}\) vanishes unless \(G^{\pi,\mu}\) is an admissible graph. The goal of this section is thus to compute the right-hand side of~\eqref{eq: moments}. To do so, we first need to identify the partitions \(\pi \in \mathcal{P}(V)\) and \(\mu \in \mathcal{P}(W)\) for which \(G^{\pi,\mu}\) is an admissible graph.

\begin{figure}
\begin{tikzpicture}

\draw (0,0) circle [radius=2];
\def\nodesA{12}
\foreach \i in {1, ..., \nodesA} {
        
\pgfmathsetmacro\angle{360/\nodesA * \i}
\pgfmathsetmacro\xpos{0 + 2*cos(\angle)}
\pgfmathsetmacro\ypos{0 + 2*sin(\angle)}
     \coordinate (P\i) at (\xpos, \ypos); 
  \ifodd\i
           \fill[black] (\xpos, \ypos) circle [radius=2pt];  
        \else
           \draw[fill=white]  (\xpos, \ypos) circle [radius=2pt];  
        \fi
    }

\node[right] at (P1) {\tiny $v_1$}  ;
\node[above] at (P2) {\tiny $w_1$}  ;
\node[above] at (P3) {\tiny $v_6$}  ;
\node[above] at (P4) {\tiny $w_6$}  ;
\node[left] at (P5) {\tiny $v_5$}  ;
\node[left] at (P6) {\tiny $w_5$}  ;
\node[left] at (P7) {\tiny $v_4$}  ;
\node[below] at (P8) {\tiny $w_4$}  ;
\node[below] at (P9) {\tiny $v_3$}  ;
\node[below] at (P10) {\tiny $w_3$}  ;
\node[right] at (P11) {\tiny $v_2$}  ;
\node[right] at (P12) {\tiny $w_2$}  ;
\end{tikzpicture}
\caption{A simple bipartite cycle of length \(12\).}
\label{fig3}
\end{figure}

\subsection{Preliminaries on set partitions}

We first introduce some general definitions about partitions of sets that will apply to both sets \(W\) and \(V\). We start with some classical definitions.

\begin{defn} \label{def: partition}
Let \(X = \{x_1, \ldots, x_k\}\) be a set where elements are labeled in cyclic order, meaning \(x_i < x_{i+1}\), with the convention that \(x_{k+1}=x_1\).
\begin{itemize}
\item[(a)] A partition \(\pi\) of \(X\) is a decomposition \(\pi = \{B_1, \ldots, B_{|\pi|}\}\) into disjoint, non-empty subsets \(B_i\), called the \emph{blocks} of the partition. The number of blocks of \(\pi\) is denoted by \(|\pi|\). Given two elements \(x_i,x_j \in X\), we write \(x_i \sim_\pi x_j\) if \(x_i\) and \(x_j\) belong to the same block.  The set of all partitions of \(X\) is denoted by \(\mathcal{P}(X)\).
\item[(b)] The partition \(\pi = \{X\} \in \mathcal{P}(X)\) with only one block is called the \emph{singleton partition} (or the \emph{trivial partition}). The partition \(\pi = \{ \{x\} \colon x \in X\} \in \mathcal{P}(X)\) is called the \emph{partition of singletons}.
\item[(c)] A partition \(\pi\) of \(X\) is called \emph{crossing} if there exist indices \(i_1 < j_1 < i_2 < j_2\) such that \(x_{i_1} \sim_\pi x_{i_2} \not \sim_\pi x_{j_1} \sim_\pi x_{j_2}\). If no such indices exist, \(\pi \in \mathcal{P}(X)\) is \emph{noncrossing}. The set of noncrossing partitions of \(X\) is denoted by \(\mathcal{NC}(X)\). Note that \(\pi = \{X\} \in \mathcal{NC}(X)\) and \(\pi = \{\{x\} \colon x \in X\}  \in \mathcal{NC}(X)\).
\item[(d)] Let \(B_i = \{x_{i_1}, \ldots, x_{i_{|B_i|}}\}\) and \(B_j = \{x_{j_1}, \ldots, x_{j_{|B_j|}}\}\) be two blocks of a partition \(\pi \in \mathcal{P}(X)\), ordered as \(x_{i_1} < \cdots < x_{i_{|B_i|}}\) and \(x_{j_1} < \cdots < x_{j_{|B_j|}}\). We say \(B_i < B_j\) if there exist consecutive elements \(x_{j_\ell}, x_{j_\ell+1} \in B_j\) such that \(x_{j_\ell} < x_{i_1} < \cdots < x_{i_{|B_i|}} <  x_{j_\ell+1}\).
\end{itemize}
\end{defn}

We now introduce two additional definitions needed to describe the matrix moments.

\begin{defn} \label{def: partition 2}
Let \(X = \{x_1, \ldots, x_k\}\) be a cyclically ordered set.
\begin{itemize}
\item[(e)] For a partition \(\pi\) of \(X\), let \(b(\pi)\) denote the collection of \emph{nearest neighbor pairs} within a block of \(\pi\): 
\[
\qquad \quad b(\pi) \coloneqq  \cup_{i=1}^k \left \{(x_i,x_{i+1}) \colon x_i \sim_\pi x_{i+1}  \right \}.
\]
Since \(X\) is cyclic, \((x_k,x_1)\) is also included whenever \(x_k \sim_\pi x_1\). In particular, in the singleton partition \(\pi = \{ X\}\), every pair of neighbors belongs to the same block, i.e., \(b(\pi) = \{(x_i,x_{i+1}) \colon i \in [k]\}\), thus \(|b(\pi)| = k\).
\item[(f)] For a partition \(\pi\) of \(X\), let \(c(\pi)\) denote the collection of pairs of next elements within a block such that, for every other pair of elements in the same block, the two pairs do not intersect: 
\[
\begin{split}
\qquad \quad c(\pi)  \coloneqq \cup_{i=1}^{k-1} \cup_{j=i+1}^k \big \{(x_i,x_j) \colon x_j = \min \{x_\ell \colon & x_\ell \sim_\pi x_i \enspace \text{and} \enspace  \nexists \: x_p \sim_\pi x_q \enspace \text{such that} \enspace\\
&  x_i < x_p < x_\ell < x_q \enspace \text{or} \enspace  x_p < x_i < x_q < x_\ell\} \big \}.
\end{split}
\]
Here, pairs \((x_i,x_j)\) and \((x_p,x_q)\) are said to intersect (or cross) if they satisfy \(x_i < x_p < x_j < x_q\) or \(x_p < x_i < x_q < x_j\). For a noncrossing partition \(\pi \in \mathcal{NC}(X)\), \(c(\pi)\) simplifies to
\[
\begin{split}
c(\pi) & =   \cup_{i=1}^{k-1} \cup_{j=i+1}^k \left \{(x_i,x_j) \colon x_j = \min \{ x_\ell \colon x_\ell \sim_\pi x_i \} \right\}  \\
&= \cup_{i=1}^{|\pi|} \cup_{j=1}^{|B_i|-1} \left \{ (x_{i_j}, x_{i_{j+1}}) \right \},
\end{split}
\]
where each block \(B_i = \{x_{i_1} , \ldots, x_{i_{|B_i|}}\}\) is ordered as \(x_{i_1} < x_{i_2} < \cdots < x_{i_{|B_i|}}\). In this case, we have \(|c(\pi)| = \sum_{i=1}^{|\pi|} (|B_i|-1) = k - |\pi|\).
\end{itemize}
\end{defn}

\begin{rmk}\label{rmk: partition}
Let \(\pi\) be a partition of \(X=\{x_1<\cdots<x_k\}\). Then every pair \((x_i,x_{i+1})\in b(\pi)\), for \(i\in[k-1]\), belongs to \(c(\pi)\). In contrast, if \(x_1\sim_\pi x_k\), then \((x_k,x_1)\in b(\pi)\), but \((x_k,x_1)\) may not belong to \(c(\pi)\). More precisely, \((x_k,x_1)\in c(\pi)\) if and only if, for every \(x_\ell\) with \(1<\ell<k\) and \(x_\ell\sim_\pi x_1\), the pair \((x_1,x_\ell)\) intersects at least one pair \((x_p,x_q)\) belonging to a different block of \(\pi\). In particular, if \(\pi\) is noncrossing, then \((x_k,x_1)\in c(\pi)\) if and only if the block containing \(x_1\) and \(x_k\) has size \(2\), that is, if \(x_1\) and \(x_k\) are the only elements of that block.
\end{rmk}

We illustrate items (e) and (f) of Definition~\ref{def: partition 2} with the following example.

\begin{ex}
Consider the set \(X=\{x_1,\ldots,x_{10}\}\) labeled in cyclic order, and the following crossing partition:
\[
\pi = \{ \{x_1, x_3, x_5\}, \{x_2, x_4\}, \{x_6, x_7,x_9\}, \{x_8,x_{10}\}\} .
\]
In this case, there is exactly one pair of nearest neighbor elements within a block, namely \(b(\pi)=\{(x_6,x_7)\}\). Moreover, we have \(c(\pi) = \{(x_1,x_5), (x_6,x_7)\}\). To see this, consider first the block \(\{x_1,x_3,x_5\}\). This block has three possible pairs \((x_1,x_3)\), \((x_1,x_5)\), and \((x_3,x_5)\). Among these, only \((x_1,x_5)\) belongs to \(c(\pi)\), since it does not intersect any pair from another block. Indeed, \((x_1,x_3)\notin c(\pi)\) because there exist \(x_2\sim_\pi x_4\) such that \(x_1<x_2<x_3<x_4\), and similarly \((x_3,x_5)\notin c(\pi)\) since
\(x_2<x_3<x_4<x_5\). The block \(\{x_2,x_4\}\) contains only the pair \((x_2,x_4)\), which does not belong to \(c(\pi)\) because it intersects both \((x_1,x_3)\) and \((x_3,x_5)\) from the first block. For the block \(\{x_6,x_7,x_9\}\), the possible pairs are \((x_6,x_7)\), \((x_6,x_9)\), and \((x_7,x_9)\). The pair \((x_6,x_7)\) belongs to \(c(\pi)\), since it consists of nearest neighbors within the block and therefore cannot intersect any pair from another block. By contrast, \((x_6,x_9)\) and \((x_7,x_9)\) both intersect the pair \((x_8,x_{10})\), and hence do not belong to \(c(\pi)\).
\end{ex}

Throughout, let \(G = (W \cup V, E)\) denote the simple bipartite cycle of length \(2k\), as introduced at the beginning of this section. We next describe the set \(W^\pi\) associated to a partition \(\pi \in \mathcal{P}(W \cup V)\).

\begin{defn}[\(W^\pi\) for \(\pi \in \mathcal{NC}(V)\)] \label{def: NC(V)}
Let \(w_1 \neq \cdots \neq w_k\), and consider a noncrossing partition \(\pi \in \mathcal{NC}(V)\) with blocks \(B_1,\ldots, B_{|\pi|}\). The set \(W^\pi\) corresponds to the finest partition of \(W\) constructed as follows. Consider two blocks \(B_i = \{v_{i_1}, \ldots, v_{i_{|B_i|}}\}\) and \(B_j = \{v_{j_1}, \ldots, v_{j_{|B_j|}}\}\) of \(\pi\) with \(|B_i| \ge 2\) and \(|B_j| \ge 2\), and assume that \(B_i < B_j\), say \(v_{j_1} < v_{i_1} < \cdots < v_{i_{|B_i|}} < v_{j_2} < \cdots < v_{j_{|B_j|}}\). We then define disjoint subsets for the elements of \(W\) corresponding to the block \(B_i\), denoted by \(W_1^i,\ldots, W_{|B_i|-1}^i\), as follows:
\begin{equation} \label{eq: subsets block B_i}
W^i_\ell = \{w_{i_\ell+1}, \ldots, w_{i_{\ell +1}}\} \enspace \text{for} \enspace 1 \leq \ell \leq |B_j|-1.
\end{equation}
Similarly, for the block \(B_j\), the disjoint subsets \(W_1^j,\ldots, W_{|B_j|-1}^j\) are defined by
\[
\begin{split}
W^j_1 &= \{w_{j_1+1},\ldots w_{i_1}\} \cup \{w_{i_{|B_i|}+1},\ldots,w_{j_2}\},\\
W^j_\ell &= \{w_{i_\ell+1}, \ldots, w_{i_{\ell+1}}\} \enspace \text{for} \enspace 2 \leq \ell \leq |B_j|-1.
\end{split}
\]
If a block \(B_k\) contains no other blocks within it, we define the subsets \(W_1^k,\ldots, W_{|B_k|-1}^k\) as given in~\eqref{eq: subsets block B_i}.
We proceed in this way for every block of \(\pi\) with at least two elements. Since each block \(B_i\) defines \(|B_i|-1\) disjoint subsets, we obtain \(k - |\pi| = \sum_{i=1}^{|\pi|} (|B_i|-1)\) disjoint subsets, denoted hereafter by \(W_1^\pi, \ldots, W_{k - |\pi|}^\pi\). Additionally, we set \(W_{k - |\pi| +1}^\pi = W \backslash \sqcup_{i=1}^{k - |\pi|} W_i\), which corresponds to the remaining vertices in \(W\). The finest partition of \(W^\pi\) is given by
\[
W^\pi = \sqcup_{i=1}^{k - |\pi| + 1} W_i^\pi.
\]
Let \(S_\pi\) denote the number of subsets among \(W_1^\pi, \ldots, W_{k - |\pi|}^\pi, W_{k - |\pi| + 1}^\pi\) that have exactly one element. According to Definition~\ref{def: partition 2}, \(S_\pi\) corresponds to the number \(|b(\pi)|\) of nearest neighbor pairs within a same block of \(\pi\). Thus, the number of subsets with cardinality at least \(2\) is given by \(R_\pi = k - |\pi| + 1 - S_\pi\). We denote by \(W_1^\pi,\ldots, W_{R_\pi}^\pi\) the subsets with more than one element, and by \(W_{R_\pi+1}^\pi,\ldots, W_{R_\pi+S_\pi}^\pi\) those with exactly one element. 
\end{defn}

\begin{rmk}\label{rmk: NC(V)}
Let \(\pi\in\mathcal{NC}(V)\) be a noncrossing partition and \(w_1,\ldots,w_k\) pairwise distinct. The graph \(G^\pi=G(W^\pi)\) associated with \(\pi\) is a block tree. Its blocks are the connected subgraphs
\(G(W_1^\pi),\ldots,G(W_{R_\pi+S_\pi}^\pi)\), defined as in Definition~\ref{def: subgraphs}.
The separating vertices in \(V\) are precisely those obtained by merging vertices belonging to the same block of \(\pi\); in particular, there are
\(\sum_{i=1}^{|\pi|}\mathbf 1_{\{|B_i|\ge2\}}\)
such separating vertices. Each subgraph \(G(W_i^\pi)\) is a bipartite cycle of length \(2|W_i^\pi|\): the vertices are arranged cyclically, and between any two adjacent elements of \(W_i^\pi\) there is exactly one vertex of \(V_i^\pi\). Consequently, \(G^\pi\) is a cactus graph. See Figure~\ref{fig4} for an illustration.
\end{rmk}

\begin{figure}
\begin{tikzpicture}

\draw (0,0) circle [radius=1];
\def\nodesA{4}
\foreach \i in {1, ..., \nodesA} {
        
\pgfmathsetmacro\angle{360/\nodesA * \i}
\pgfmathsetmacro\xpos{0 + cos(\angle)}
\pgfmathsetmacro\ypos{0 + sin(\angle)}
     \coordinate (P\i) at (\xpos, \ypos); 
  \ifodd\i
            \draw[fill=white] (\xpos, \ypos) circle [radius=2pt];  
        \else
            \fill[black]  (\xpos, \ypos) circle [radius=2pt];  
        \fi
    }

\node[above] at (P1) {\tiny $w_1$}  ;
\node[right] at (P2) {\tiny $\tilde{v}_1$}  ;
\node[below] at (P3) {\tiny $w_5$}  ;
\node[right] at (P4) {\tiny $\tilde{v}_2$}  ;

\draw (-2,0) to[bend left = 50] (P2);
 \draw (-2,0) to[bend right = 50] (P2);
 \draw[fill=white] (-2,0) circle [radius=2pt];  
\node[left] at (-2,0) {\tiny $w_6$}  ; 

\draw (2,0) circle [radius=1];

\def\nodesB{4}

\foreach \i in {1, ..., \nodesB} {
        
        \pgfmathsetmacro\angle{360/\nodesB * \i}
        
        \pgfmathsetmacro\xpos{2 +  cos(\angle)}
        \pgfmathsetmacro\ypos{0 +  sin(\angle)}
        
        \coordinate (Q\i) at (\xpos, \ypos);
        
        \ifodd\i
           \draw[fill=white] (\xpos, \ypos) circle [radius=2pt];  
        \else
            \fill[black] (\xpos, \ypos) circle [radius=2pt];  
        \fi
    }
    
\node[above] at (Q1) {\tiny $w_3$};
\node[below] at (Q3) {\tiny $w_4$};
\node[right] at (Q4) {\tiny $v_3$};

\pgfmathsetmacro\xcoord{cos(30)} 
\pgfmathsetmacro\ycoord{sin(30)} 

\end{tikzpicture}
\caption{The graph \(G^\pi\) obtained from the simple cycle \(G\) in Figure~\ref{fig3}, associated with the noncrossing partition \(\pi=\{\{v_1,v_5,v_6\},\{v_2,v_4\},\{v_3\}\}\).
The vertices \(\tilde v_1\) and \(\tilde v_2\) are obtained by merging \(v_1,v_5,v_6\) and \(v_2,v_4\), respectively.
The subsets \(W_1^\pi=\{w_1,w_5\}\), \(W_2^\pi=\{w_3,w_4\}\), and \(W_3^\pi=\{w_6\}\) form the finest partition of \(W^\pi\).}
\label{fig4}
\end{figure}

As we will see later in Lemma~\ref{lem: contr cross part in V}, if \(\pi \in \mathcal{P}(V)\) is crossing then the parameter \(\tau^0_{G^\pi}\) vanishes. We now consider the partitions of \(W\) and begin with the noncrossing ones.

\begin{defn}[\(W^\pi\) for \(\pi \in \mathcal{NC}(W)\)]  \label{def: NC(W)}
Let \(v_1 \neq \cdots \neq v_k\), and consider a noncrossing partition \(\pi \in \mathcal{NC}(W)\) with blocks \(B_1,\ldots, B_{|\pi|}\). We define \(W^\pi\) recursively as follows. 
\begin{enumerate}
\item[1.] Consider a block \(B_i = \{w_{i_1}, \ldots, w_{i_{|B_i|}}\}\) of \(\pi\) with \(|B_i| \geq 2\) and \(w_{i_1} < \cdots < w_{i_{|B_i|}}\). Define the subsets \(W_1^i, \ldots, W_{|B_i|}^i\) by 
\[
\begin{split}
W_\ell^i &= \{w_{i_\ell +1},\ldots, w_{i_{\ell+1} -1}\} \cup \{\tilde{w}_{i}\} \enspace \text{for} \enspace 1 \leq \ell \leq |B_i|-1,\\
W_{|B_i|}^i &=  \{w_{i_{|B_i|} +1}, \ldots, w_{i_1-1}\} \cup \{\tilde{w}_{i}\},
\end{split}
\]
where \(\tilde{w}_{i}\) denotes the vertex obtained by merging \(w_{i_1}, \ldots, w_{i_{|B_i|}}\). Note that \(\tilde{w}_i\) is the unique common vertex shared by the subsets \(W_1^i, \ldots, W_{|B_i|}^i\). 
\item[2.] Consider another block \(B_j = \{w_{j_1}, \ldots, w_{j_{|B_j|}}\}\) with \(|B_j| \geq 2\) and \(w_{j_1} < \cdots < w_{j_{|B_j|}}\). Since \(\pi\) is noncrossing, the vertices \(w_{j_1}, \ldots, w_{j_{|B_j|}}\) must belong to exactly one subset among \(W_1^i,\ldots, W_{|B_i|}^i\), say \(W_1^i\). The block \(B_j\) then provides a decomposition of \(W_1^i\) into subsets \(W_1^j, \ldots, W_{|B_j|}^j\) defined by
\[
\begin{split}
W_\ell^j &= \{w_{j_\ell +1},\ldots, w_{j_{\ell+1} -1}\} \cup \{\tilde{w}_{j}\}  \enspace \text{for} \enspace 1 \leq \ell \leq |B_j|-1 ,\\
W_{|B_j|}^j &=  \{w_{j_{|B_j|} +1}, \ldots, w_{i_2-1}, w_{i_1+1}, \ldots, w_{j_1-1}\} \cup \{\tilde{w}_{i}, \tilde{w}_{j}\},
\end{split}
\]
where \(\tilde{w}_{j}\) stands for the vertex obtained by merging \(w_{j_1}, \ldots, w_{j_{|B_j|}}\). Note that \(\tilde{w}_j\) is the unique common vertex shared by \(W_1^j, \ldots, W_{|B_j|}^j\). 
\item[3.] Proceed in this way for every block of \(\pi\) with at least two elements.  
\end{enumerate}
The first block \(B_i\) defines \(|B_i|\) subsets. Every subsequent block \(B_j\) introduces \(|B_j|-1\) new subsets, since the block \(B_j\) provides a decomposition of an existing subset. Thus, at the end of the procedure, we obtain a total of \(\sum_{i=1}^{|\pi|} (|B_i|-1) + 1 = k - |\pi| +1 \) subsets of \(W\), which we denote by \(W_1^\pi, \ldots, W_{k - |\pi|+1}^\pi\). Moreover, the number of vertices \(\tilde{w}_1,\ldots,\tilde{w}_q\) which belong to several \(W_i^\pi\) is given by \(q = \sum_{i=1}^{|\pi|} \mathbf{1}_{\{B_i|\geq 2\}} \leq k - |\pi|\), with equality if every block of \(\pi\) has at most two elements.

The number of subsets among \(W_1^\pi, \ldots, W_{k - |\pi|+1}^\pi\) that consist of a single element, denoted by \(S_\pi\), corresponds to the number \(|b(\pi)|\) of pairs of nearest neighbors within a same block of \(\pi\). Thus, the number of subsets with cardinality at least two is given by \(R_\pi = k - |\pi| + 1 - S_\pi\). We write \(W_1^\pi,\ldots, W_{R_\pi}^\pi\) for the subsets with more than one element and \(W_{R_\pi+1}^\pi,\ldots, W_{R_\pi+S_\pi}^\pi\) for those with exactly one element. Note that each subset \(W_{R_\pi+i}^\pi\) with exactly one element contains a common vertex \(\tilde{w}_j\). By construction there exists at least one subset among \(W_1^\pi,\ldots, W_{R_\pi}^\pi\) that also contains \(\tilde{w}_j\). Consequently, merging the subsets \(W_1^\pi,\ldots, W_{R_\pi}^\pi\) as described in Definition~\ref{def: merging operation} results in \(W^\pi\).
\end{defn}

\begin{figure}
\begin{tikzpicture}

\node[right] at (-1,0) {\tiny $\tilde{w}_1$}  ;
\node[right] at (2,0) {\tiny $\tilde{w}_2$}  ;

\draw (-2,0) to[bend left = 50] (-1,0);
\draw (-2,0) to[bend right = 50] (-1,0);
\fill[black] (-2,0) circle [radius=2pt];   
\fill[black]  (0.5, -0.5) circle [radius=2pt];  
\fill[black]  (1.25, 0.4) circle [radius=2pt];  
\fill[black]  (-0.25, 0.4) circle [radius=2pt];   
\draw (-1,0) to[bend right = 15] (0.5,-0.5);
\draw (2,0) to[bend left = 15] (0.5,-0.5);
\draw (-1,0) to[bend left = 10] (-0.25,0.4);
\draw (-0.25,0.4) to[bend left = 7] (0.5,0.5);
\draw (0.5,0.5) to[bend left = 7] (1.25,0.4);
\draw (1.25,0.4) to[bend left = 10] (2,0);
\draw[fill=white]  (-1,0) circle [radius=2pt]; 
\node[left] at (-2,0) {\tiny $v_1$}  ; 
\node[below] at (0.5,-0.5) {\tiny $v_6$}  ; 
\node[above] at (-0.25,0.4) {\tiny $v_2$}  ; 
\node[above] at (0.5,0.5) {\tiny $w_3$}  ; 
\node[above] at (1.25,0.4) {\tiny $v_3$}  ; 
\draw[fill=white]  (0.5, 0.5) circle [radius=2pt];  

\draw (3,0) circle [radius=1];

\def\nodesB{4}

\foreach \i in {1, ..., \nodesB} {
        
        \pgfmathsetmacro\angle{360/\nodesB * \i}
        
        \pgfmathsetmacro\xpos{3 +  cos(\angle)}
        \pgfmathsetmacro\ypos{0 +  sin(\angle)}
        
        \coordinate (Q\i) at (\xpos, \ypos);
        
        \ifodd\i
           \fill[black]  (\xpos, \ypos) circle [radius=2pt];  
        \else
            \draw[fill=white] (\xpos, \ypos) circle [radius=2pt];  
        \fi
    }
    
\node[above] at (Q1) {\tiny $v_4$};
\node[below] at (Q3) {\tiny $v_5$};
\node[right] at (Q4) {\tiny $w_5$};
\end{tikzpicture}
\caption{The graph \(G^\pi\) obtained from the simple cycle \(G\) in Figure~\ref{fig3}, associated to the noncrossing partition \(\pi = \{\{w_1,w_2\}, \{w_3\}, \{w_4, w_6\}, \{w_5\}\}\). The vertices \(\tilde w_1\) and \(\tilde w_2\) are formed by merging \(w_1,w_2\) and \(w_4,w_6\), respectively.
Merging the subsets \(W_1^\pi=\{\tilde w_1,\tilde w_2,w_3\}\) and \(W_2^\pi=\{\tilde w_2,w_5\}\) yields \(W^\pi\).}
\label{fig5}
\end{figure}

It remains to consider the case of crossing partitions of \(W\).

\begin{defn}[\(W^\pi\) for \(\pi \in \mathcal{P}(W)\)]  \label{def: P(W)}
Let \(v_1 \neq \cdots \neq v_k\), and consider a partition \(\pi \in \mathcal{P}(W)\) with blocks \(B_1,\ldots, B_{|\pi|}\). Assume that \(\pi\) has a crossing. We describe the partitioning of the set \(W\) according to \(\pi \in \mathcal{P}(W)\).

Consider two blocks \(B_i={\{w_{i_1}, \ldots, w_{i_{|B_i|}}}\}\) and \(B_j={\{w_{j_1}, \ldots, w_{j_{|B_j|}}}\}\) with \(|B_i| \geq 2\) and \(|B_j| \geq 2\) and assume that \(w_{i_1} < w_{j_1} < w_{i_2} < w_{j_2} < \cdots < w_{j_{|B_j|}} < w_{i_3} < \cdots w_{i_{|B_i|}}\). As seen in Definition~\ref{def: NC(W)}, the block \(B_i\) provides a decomposition into subsets \(W_1^i, \ldots, W_{|B_i|}^i\) given by
\[
\begin{split}
W_\ell^i &= \{w_{i_\ell +1},\ldots, w_{i_{\ell+1} -1}\} \cup \{\tilde{w}_{i}\} \enspace \text{for} \enspace 1 \leq \ell \leq |B_i|-1,\\
W_{|B_i|}^i &=  \{w_{i_{|B_i|} +1}, \ldots, w_{i_1-1}\} \cup \{\tilde{w}_{i}\},
\end{split}
\]
where we recall that \(\tilde{w}_{i}\) denotes the vertex obtained by merging \(w_{i_1}, \ldots, w_{i_{|B_i|}}\). Since \(w_{j_1} \in W_1^i\) and \(w_{j_2}, \ldots, w_{j_{|B_j|}} \in W_2^i\), the block \(B_j\) defines a partition of the subset obtained by merging \(W_1^i\) and \(W_2^i\). That is, if \(\widetilde{W}_i\) denotes the merging of \(W_1^i\) and \(W_2^i\), i.e., 
\[
\widetilde{W}_i = \{w_{i_1+1}, \ldots, w_{i_2-1}, w_{i_2+1}, \ldots, w_{i_3-1}\} \cup \{\tilde{w}_i\},
\]
then the block \(B_j\) gives a decomposition of \(\widetilde{W}_i\) into subsets \(W_1^j, \ldots, W_{|B_j|-1}^j\) defined by
\[
\begin{split}
W_\ell^j &= \{w_{j_\ell +1},\ldots, w_{j_{\ell+1} -1}\} \cup \{\tilde{w}_{j}\} \enspace \text{for} \enspace 2 \leq \ell \leq |B_j|-1,\\
W_1^j & = \{w_{i_1+1}, \ldots, w_{i_2-1}, w_{i_2+1}, \ldots, w_{j_2 - 1}, w_{j_{|B_j|} +1}, \ldots, w_{i_3-1}\} \cup \{\tilde{w}_i, \tilde{w}_j\}.
\end{split}
\]
Given the subsets \(W_3^i, \ldots, W_{|B_i|}^i, W_1^j, \ldots, W_{|B_j|-1}^j\), we then proceed recursively for every block with at least two elements. In particular, we observe that given a partition \(\pi \in \mathcal{P}(W)\), the first block \(B_i\) defines \(|B_i|\) subsets. If the elements of the next block \(B_j\) belong to exactly one subset, then we get \(|B_j|-1\) new subsets. Otherwise, if the elements of \(B_j\) belongs to \(N(B_j)\) different subsets, then we merge these \(N(B_j)\) subsets into a new subset \(\widetilde{W}_i\) and \(B_j\) provides a decomposition of \(\widetilde{W}_i\) into \(|B_j| - 1 - (N(B_j)-1) = |B_j| - N(B_j)\) new subsets. Given now \(|B_i| - N(B_j) + 1 +  |B_j| - N(B_j)  = |B_i| + |B_j| - 2 N(B_j) +1\) subsets, we proceed in this way for every block of \(\pi\) with more than one element. 

The number of subsets that we obtain at the end of the procedure corresponds to \(|c(\pi)| + 1\). Indeed, the procedure described above is equivalent to start from \(G\) and first define subsets associated to every pair of elements \((w,w') \in c(\pi)\). Then, inside each subset we make the crossing identifications that are present in \(\pi\). We denote the \(|c(\pi)| + 1\) subsets of \(W\) that we obtain by \(W_1^\pi, \ldots, W_{|c(\pi)|+1}^\pi\). The number \(S_\pi\) of subsets \(W_i^\pi\) of cardinality equal to one corresponds to the number \(|b(\pi)|\), so that there are \(R_\pi = |c(\pi)| + 1 - S_\pi\) subsets of cardinality at least two. In particular, merging the subsets \(W_1^\pi, \ldots, W_{R_\pi}^\pi\) results in \(W^\pi\). 
\end{defn}

\begin{rmk}\label{rmk: P(W)}
Let \(\pi\in\mathcal P(W)\) be a partition with \(v_1,\ldots,v_k\) pairwise distinct. The graph \(G^\pi=G(W^\pi)\) associated with \(\pi\) is a block tree with exactly one block, since it has no separating vertices in \(V\). Moreover, \(G(W^\pi)\) is obtained by merging the subgraphs
\(G(W_1^\pi),\ldots,G(W_{R_\pi}^\pi)\) according to the merging procedure of Definition~\ref{def: merging operation}. If \(\pi\) is noncrossing, the resulting graph is a cactus graph, as it consists of simple cycles glued together at common vertices in \(W\). Figures~\ref{fig5} and~\ref{fig6} illustrate this construction.
\end{rmk}

\begin{figure}
\begin{tikzpicture}

\draw (0,0) circle [radius=1.5];
\def\nodesB{8}

\foreach \i in {1, ..., \nodesB} {
        
        \pgfmathsetmacro\angle{360/\nodesB * \i}
        
        \pgfmathsetmacro\xpos{0 + 1.5 * cos(\angle)}
        \pgfmathsetmacro\ypos{0 + 1.5 * sin(\angle)}
        
        \coordinate (Q\i) at (\xpos, \ypos);
        
        \ifodd\i
           \fill[black] (\xpos, \ypos) circle [radius=2pt];  
        \else
            \draw[fill=white] (\xpos, \ypos) circle [radius=2pt]; 
        \fi
    }

\node[right] at (Q1) {\tiny $v_1$};
\node[above] at (Q2) {\tiny $\tilde{w}_1$};
\node[left] at (Q3) {\tiny $v_6$};
\node[left] at (Q4) {\tiny $w_6$};
\node[left] at (Q5) {\tiny $v_5$};
\node[below] at (Q6) {\tiny $\tilde{w}_2$};
\node[right] at (Q7) {\tiny $v_2$};
\node[right] at (Q8) {\tiny $w_2$};

\fill[black] (0.5,0) circle [radius=2pt];
\node[right] at (0.5,0)  {\tiny $v_3$};
\fill[black] (-0.5,0) circle [radius=2pt];
\node[left] at (-0.5,0)  {\tiny $v_4$};
\draw (Q2) to[bend left = 20] (0.5,0);
\draw (Q2) to[bend right = 20 ] (-0.5,0);
\draw (Q6) to[bend right = 20] (0.5,0);
\draw (Q6) to[bend left = 20 ] (-0.5,0);

\draw[fill=white] (0,1.5) circle [radius=2pt]; 
\draw[fill=white](0,-1.5) circle [radius=2pt]; 

\end{tikzpicture}
\caption{The graph \(G^\pi\) obtained from the simple cycle \(G\) in Figure~\ref{fig3}, associated to the crossing partition \(\pi = \{\{w_1,w_4\}, \{w_2\}, \{w_3,w_5\}, \{w_6\}\}\). The vertices \(\tilde w_1\) and \(\tilde w_2\) are obtained by merging \(w_1,w_4\) and \(w_3,w_5\), respectively.}
\label{fig6}
\end{figure}

\subsection{Convergence of matrix moments in expectation} \label{subsection: matrix moments}

In this subsection, we compute the limiting tracial moments of \(M=Y_m Y_m^\top\) using the limiting injective trace of Proposition~\ref{main3}. We recall the parameters introduced in Definitions~\ref{def: C_deg (f)} and~\ref{def: C_W (f)}: \(C_d(f)\) is defined for even integers \(d\), and \(C_{(W_k)_{k=1}^K}(f)\) is defined for subsets \(W_1, \ldots, W_K \subseteq W\) with \(|W_i| \ge 2\) such that the associated subgraphs \(G(W_1), \ldots, G(W_K)\) are connected.

\begin{prop} \label{prop: main cycle}
Under Assumptions~\ref{hyp1}--\ref{hyp3}, for every \(k\ge1\), the \(k\)th moment \(\frac{1}{p}\tr M^k\) converges in expectation to a limit \(m_k=m_k(\phi,\psi,f,\Phi,\nu_x)\) given by
\[
m_k = \sum_{\pi \in \mathcal{NC}(V)} \frac{\phi^{k - |\pi|}}{\psi^{k-1}}  C_2(f)^{S_\pi}  \prod_{i=1}^{R_\pi} \left ( \sum_{\mu_i \in \mathcal{P}(W_i^\pi)} \psi^{|W_i^\pi|-|\mu_i|}  C_{G_i^{\pi, \mu_i}} (f)  \right ) ,
\]
where for every \(i \in \{1, \ldots, R_\pi\}\) and \(\mu_i  \in \mathcal{P}(W_i^\pi)\),
\[
C_{G_i^{\pi, \mu_i}} (f)  =  
\begin{cases}
C_{2 |W_i^\pi|}(f) & \text{if} \enspace \mu_i = \{W_i^\pi\}, \\
\sum_{P_i \in \mathcal{P}([R_{\mu_i}])}  C_{(\widetilde{W}_j^{P_i})_{j=1}^{|P_i|}}(f) \mathbf{1}_{\{\forall j, \,\widetilde{G}_j^{P_i} \: \textnormal{is connected}\}} & \text{otherwise}.
\end{cases}
\]
Here:
\begin{itemize}
\item for any partition \(P\) of \(W\) or \(V\), we write \(S_P = |b(P)|\) and \(R_P = | c( P)| + 1 - S_P\), where \(b(P)\) and \(c(P)\) are as in Definition~\ref{def: partition 2};
\item for \(\pi \in \mathcal{NC}(V)\), \(W_1^\pi, \ldots, W_{R_\pi}^\pi\) are the disjoint subsets of \(W^\pi\) with cardinality at least \(2\), as in Definition~\ref{def: NC(V)}, and \(G(W_1^\pi), \ldots, G(W_{R_\pi}^\pi)\) are their associated subgraphs;
\item for \(\mu_i \in \mathcal{P}(W_i^\pi)\), \(G_i^{\pi, \mu_i}\) is the graph obtained from \(G(W_i^\pi)\) by identifying vertices of \(W_i^\pi\) which belong to the same block of \(\mu_i\);
\item for \(\mu_i \in \mathcal{P}(W_i^\pi)\), the subsets \(W_1^{\mu_i}, \ldots, W_{R_{\mu_i}}^{\mu_i}\) denote the finest partition of \(W_i^{\pi,\mu_i}\), as described by Definitions~\ref{def: NC(W)} and~\ref{def: P(W)};
\item for \(P_i \in \mathcal{P}( [R_{\mu_i}])\) with blocks \(B_1,\ldots, B_{|P_i|}\), and for every \(1 \le j \le |P_i|\), the subset \(\widetilde{W}_j^{P_i}\) is obtained by merging \(\{W_\ell^{\mu_i} \colon \ell \in B_j\}\) according to Definition~\ref{def: merging operation}, and \(\widetilde{G}_j^{P_i} = G(\widetilde{W}_j^{P_i})\) denotes the corresponding subgraph.
\end{itemize}
\end{prop}

We now describe the previous result in words. We start with the simple bipartite cycle \(G= (W \cup V, E)\) and proceed as follows. Choose a noncrossing partition \(\pi \in \mathcal{NC}(V)\). According to Remark~\ref{rmk: NC(V)}, the resulting graph \(G^\pi\) is a cactus graph, which contains 
\begin{itemize}
\item \(S_\pi\) simple cycle graphs of length \(2\), each contributing a parameter \(C_2(f)\),
\item \(R_\pi\) simple cycle graphs of length \(4\) or more, denoted by \(G(W_1^\pi), \ldots G( W_{R_\pi}^\pi)\). 
\end{itemize}
Consider the simple cycle graphs \(G(W_1^\pi), \ldots G( W_{R_\pi}^\pi)\). Now, for every \(i \in [R_\pi]\), choose a partition \(\mu_i \in \mathcal{P}(W_i^\pi)\) and consider the resulting subgraph \(G(W_i^{\pi, \mu_i})\), obtained by identifying the vertices in \(W_i^\pi\) that belong to the same block \(\mu_i\). Each subgraph \(G(W_i^{\pi, \mu_i})\) contributes a sum of terms, detailed as follows: 
\begin{itemize}
\item For each subset \(W_i^{\pi, \mu_i}\), define the finest partition \(W_1^{\mu_i}, \ldots, W_{R_{\mu_i}}^{\mu_i}\) according to Definitions~\ref{def: NC(W)} and~\ref{def: P(W)}. This partition provides the contribution \(C_{(W_i^{\mu_i})_{i=1}^{R_{\mu_i}}}(f)\). 
\item Consider a partition \(P_i \in \mathcal{P}( [R_{\mu_i}])\) with blocks \(B_1,\ldots, B_{|P_i|}\). For each block \(B_j\), merge the subsets \((W_\ell^{\mu_i})_{\ell \in B_j}\) as per Definition~\ref{def: merging operation} and denote \(\widetilde{W}_j^{P_i}\) the new subset and \(\widetilde{G}_j^{P_i} = G(\widetilde{W}_j^{P_i})\) the corresponding subgraph. If \(\widetilde{G}_j^{P_i} \) is connected for every \(1 \le j \le |P_i|\), then this partition provides the contribution \(C_{(\widetilde{W}_j^{P_i})_{i=1}^{|P_i|}}(f)\), otherwise zero.
\end{itemize} 

The rest of the subsection is devoted to the proof of Proposition~\ref{prop: main cycle}. We first consider the partition of singletons for both sets \(V\) and \(W\).

\begin{lem} \label{lem: contr cycle}
If \(G = (W \cup V, E)\) is the simple bipartite cycle of length \(2k\), then
\[
\tau^0_G = 
\begin{cases}
C_2(f) & \text{if} \enspace k=1, \\
\frac{1}{\psi^{k-1}} C_{W}(f) & \text{if} \enspace k \ge 2.
\end{cases}
\]
\end{lem}
\begin{proof}
If \(k=1\), then \(G\) is a double tree. It follows from Proposition~\ref{main3} that
\[
\tau^0_G = C_2(f).
\]
Assume now that \(k \geq 2\). The simple bipartite cycle graph is a block tree with exactly one block, where \(W\) contains more than two vertices since \(k \ge 2\), and all vertices in \(W \cup V\) have degree \(2\). Thus, \(G\) is an admissible graph and by Proposition~\ref{main3}, the limiting injective trace \(\tau^0_G\) is given by
\[
\tau^0_G = \frac{1}{\psi^{k-1}} \sum_{K \ge 1} \sum_{\{W_1,\ldots, W_K\} \in \mathcal{A}_K(W)} C_{(W_k)_{k=1}^K}(f).
\]
By definition, \(\{W\} \in \mathcal{A}_1(W)\). We claim that there are no other admissible decomposition of \(W\). First, assume by contradiction that there exist subsets \(W_1,W_2\) such that \(\{W_1, W_2\} \in \mathcal{A}_2(W)\). Then, \(W_1 \cup W_2 = W\) and \(| W_1 \cap W_2 | \geq 1\). If \(| W_1 \cap W_2 | =1\), without loss of generality we may label \(W_1\) and \(W_2\) as \(W_1 = \{w_1, \ldots, w_{j_1}\}\) and \(W_2 = \{w_{j_1}, \ldots, w_k\}\) such that \(w_{j_1} \in W_1 \cap W_2\). By construction, there exists \(v \in V_1 \cap V_2\) such that \(w_n \sim v \sim w_1\) and \(\deg_{G_1}(v) = \deg_{G_2}(v) = 1\). This contradicts condition (e) of Definition~\ref{def: admissible decomposition}. If there are two or more common vertices between \(W_1\) and \(W_2\), then \(|W_{G_1, G_2}| \geq 2\) and this is a contradiction to (f) of Definition~\ref{def: admissible decomposition}. Now, assume by contradiction that there exist \(\{W_1,\ldots, W_K\} \in \mathcal{A}_K(W)\) for some \(K > 1\). Let \(\tilde{w}_1, \ldots, \tilde{w}_q \in W\) denote the vertices which belong to two or more subsets among \(W_1, \ldots, W_K\). If \(q \geq K\), then \(|W_{G_1, \ldots, G_K}| = q \geq K\), which contradicts item (f) of Definition~\ref{def: admissible decomposition}. If \(q = K - 1\), since \(W = \cup_{i=1}^K W_i\) and there is no subset which is disjoint from the others, without loss of generality we may label the subsets \(W_1,\ldots, W_K\) by \(W_1 = \{w_1, \ldots, w_{j_1}\}, W_2 = \{w_{j_1},w_{j_1+1}, \ldots, w_{j_2}\}, \ldots, W_K = \{w_{j_{K-1}}, w_{j_{K-1} + 1}, \ldots ,w_k\}\), where \(1 < j_1 < j_2 < \cdots < j_{K-1} < k\). In particular, \(|W_i \cap W_{i+1}| =1\) for any \(1 \leq i \leq K-1\) and \(|W_1 \cap W_K| =0\). By construction, there exists \(v \in V_1 \cap V_K\) such that \(w_n \sim v \sim w_1\) and \(\deg_{G_1}(v) = \deg_{G_K}(v) =1\). This is again a contradiction to item (e) of Definition~\ref{def: admissible decomposition}. Finally, if \(q \leq K-2\), any decomposition of \(W\) into subsets \(W_1, \ldots, W_K\) fails to satisfy condition (e) of Definition~\ref{def: admissible decomposition}. 

The only contribution in \(\tau^0_G\) comes therefore from \(\{W\} \in \mathcal{A}_1(W)\), i.e., 
\[
\tau^0_G = \frac{1}{\psi^{k-1}} C_{W}(f).
\]
The parameter \(C_{W}(f)\) is given by~\eqref{eq: C_{W_i}} and, in this case, takes the following form:
\begin{equation} \label{eq: C_W cycle}
C_W (f) = \frac{1}{(2 \pi)^{2k}} \int_{\R^{2k}} \prod_{i=1}^k \textnormal{d} \gamma_{(w_i, v_{i-1})}   \textnormal{d} \gamma_{(w_i, v_i)}  \hat{f} (\gamma_{(w_i, v_{i-1})} )   \hat{f} (\gamma_{(w_i, v_i)} )  e^{ \E_X \left [ Z_{w_i}(\boldsymbol{\gamma}) \right ]} \E_X \left [ \prod_{i=1}^k Z_{w_i} (\boldsymbol{\gamma}) \right ],
\end{equation}
where \(Z_{w_i}(\boldsymbol{\gamma}) = \Phi \left( \gamma_{(w_i, v_{i-1})} X_{v_{i-1}} + \gamma_{(w_i, v_i)}X_{v_i}  \right)\), and we used the convention that \(v_0 = v_k\).
\end{proof}

We now assume that \(w_1, \ldots, w_k\) are pairwise distinct and consider a noncrossing partition of \(V\). 

\begin{lem} \label{lem: contr non-cross part in V}
If \(\mu = \{ \{w_i\} \colon i \in [k]\}\), then for every \(\pi \in \mathcal{NC}(V)\),
\[
\tau^0_{G^{\pi,\mu}} = \frac{\phi^{k - |\pi|}}{\psi^{k-1}} C_2(f)^{S_\pi} \prod_{i=1}^{R_\pi} C_{W_i^\pi}(f), 
\]
where 
\begin{itemize}
\item \(S_\pi =|b(\pi)|\) and \(R_\pi = k - |\pi| +1 - S_\pi\) (see Definition~\ref{def: partition 2}); 
\item \(W^\pi\) denotes the finest partition of \(W\) according to \(\pi\) and \(W_1^\pi, \ldots, W_{R_\pi}^\pi\) are the subsets of \(W^\pi\) with at least two vertices, as described by Definition~\ref{def: NC(V)}.
\end{itemize}
\end{lem}

According to Remark~\ref{rmk: NC(V)}, each subgraph \(G(W_i^\pi)\) is a simple bipartite cycle of length \(2 |W_i^\pi|\). Therefore, the parameter \(C_{W_i^\pi}(f)\) takes the same form  as~\eqref{eq: C_W cycle} in the proof of the previous lemma. 

\begin{proof}
We first observe that, since \(\mu \in \mathcal{P}(W)\) is the partition of singletons, the graph \(G^{\pi,\mu}\) becomes a double tree if and only if \(\pi \in \mathcal{NC}(V)\) is the singleton partition. In this case, \(\deg(w_i) = 2\) for every \(i \in [k]\). By Proposition~\ref{main3}, we obtain 
\[
\tau^0_{G^{\pi,\mu}} = \frac{\phi^{k - 1}}{\psi^{k-1}} C_2(f)^k.
\]
Here, \(S_\pi = |b(\pi)| = k\) and \(R_\pi = k - |\pi| + 1 - k = 0\). 

Now consider a noncrossing partition \(\pi\) with \(\pi \neq \{V\}\). According to Definition~\ref{def: NC(V)} and Remark~\ref{rmk: NC(V)}, the set \(W^\pi\) decomposes into disjoint subsets \(W_1^\pi, \ldots, W_{k-|\pi|+1}^\pi\) such that the associated subgraphs \(G(W_1^\pi), \ldots, G(W_{k - |\pi| +1}^\pi)\) are simple bipartite cycles, which are connected through \(p = \sum_{i=1}^{|\pi|} \mathbf{1}_{\{|B_i| \ge 2\}} \le k - |\pi|\) vertices in \(V\). The resulting graph \(G^{\pi,\mu}\) is therefore a block tree with \(k-|\pi|+1\) blocks, given by the subgraphs \(G(W_1^\pi), \ldots, G(W_{k - |\pi| +1}^\pi)\). Since each block is a simple cycle, \(G^{\pi,\mu}\) is an admissible graph. Recall that \(S_\pi = |b(\pi)|\) denotes the number of subgraphs of length \(2\). Consequently, there are \(R_\pi = k - |\pi| + 1 - S_\pi\) subsets \(W_1^\pi,\ldots, W_{R_\pi}^\pi\) of cardinality at least \(2\). By definition, we have \(\{W_i^\pi\} \in \mathcal{A}_1(W_i^\pi)\) for every \(i \in [R_\pi]\). Therefore, by Proposition~\ref{main3}, the following term contributes to the limiting injective trace \(\tau^0_{G^{\pi,\mu}}\):
\[
\frac{\phi^{k - |\pi|}}{\psi^{k-1}} \prod_{i=1}^{R_\pi} C_{W_i^\pi}(f) \prod_{w \in W \backslash \cup_{i=1}^{R_\pi} W_i^\pi} C_{\deg(w)}(f) =\frac{\phi^{k - |\pi|}}{\psi^{k-1}} C_2(f)^{S_\pi} \prod_{i=1}^{R_\pi} C_{W_i^\pi}(f),
\]
where we used the fact that every \(w \in W \backslash \cup_{i=1}^{R_\pi} W_i\) has degree \(2\). Finally, since each subgraph \(G(W^\pi_i)\) is a simple bipartite cycle, there are no further admissible decomposition of \(W^\pi_i\), as shown in Lemma~\ref{lem: contr cycle}. Hence, no additional terms contribute to \(\tau^0_{G^{\pi,\mu}}\), which completes the proof.
\end{proof}

We next show that if \(\pi \in \mathcal{P}(V)\) is a crossing partition, then \(\tau^0_{G^{\pi,\mu}}\) vanishes for any partition \(\mu \in \mathcal{P}(W)\).

\begin{lem} \label{lem: contr cross part in V}
For every partition \(\mu \in \mathcal{P}(W)\) and every crossing partition \(\pi \in \mathcal{P}(V)\), \(\tau^0_{G^{\pi,\mu}}= 0\).
\end{lem}
\begin{proof}
It suffices to consider the case where \(\pi\) contains a single crossing, since any crossing partition contains such a configuration. Moreover, the partition \(\mu \in \mathcal P(W)\) plays no role in the argument, as the obstruction to admissibility comes entirely from the identifications induced by \(\pi\) on the vertices in \(V\). Let \(\pi \in \mathcal{P}(V)\) be a partition containing a crossing. Then there exist two blocks \(B_i\) and \(B_j\) of \(\pi\), and distinct vertices
\(v_{i_1}, v_{i_2} \in B_i\) and \(v_{j_1}, v_{j_2} \in B_j\), such that
\[
v_{i_1} < v_{j_1} < v_{i_2} < v_{j_2}.
\]
We perform identifications on all blocks of \(\pi\) except those involving the vertices in \(B_i\) and \(B_j\). By Remark~\ref{rmk: NC(V)}, the graph obtained after these identifications is a block tree.
We now analyze how the remaining crossing affects the resulting graph.

First, suppose that the crossing occurs within a single block, i.e., suppose that \(v_{i_1}, v_{i_2}, v_{j_1}\), and \(v_{j_2}\) belong to the same connected subgraph. By the cyclic ordering of the vertices in \(V\), performing the identifications within \(B_i\) and \(B_j\) produces a vertex of degree at least \(4\). Consequently, the degree condition in Definition~\ref{def: admissible graph} is violated, and the resulting graph \(G^\pi\) is not admissible.

Next, suppose that the crossing occurs between two distinct blocks. For instance, assume that
\(v_{i_1}, v_{j_1}\), and \(v_{i_2}\) belong to one subgraph, while \(v_{j_2}\) belongs to another.
Merging \(v_{i_1}\) with \(v_{i_2}\) creates a new block containing \(v_{j_1}\).
Subsequently, merging \(v_{j_1}\) with \(v_{j_2}\) produces a separating vertex that connects two previously distinct subgraphs.
This operation introduces a cycle in the block structure of \(G^\pi\), so the resulting graph is no longer a block tree.
Hence, \(G^\pi\) is not admissible in this case either.

In both situations, Proposition~\ref{main3} implies that the limiting injective trace vanishes, and therefore
\(\tau^0_{G^{\pi,\mu}} = 0\).
\end{proof}

We now focus on partitions among the vertices of \(W\). According to Definition~\ref{def: NC(V)}, given a partition \(\pi \in \mathcal{NC}(V)\), the set \(W^\pi\) is decomposed into \(k - |\pi| + 1\) disjoint subsets \(W_1^\pi, \ldots, W_{k - |\pi|+1}^\pi\). If we now consider a partition \(\mu\) of the set \(W\), the vertices within a same block of \(\mu\) belong either to exactly one component or to several components among \(W_1^\pi, \ldots, W_{k - |\pi|+1}^\pi\). We first show that the latter case gives a zero contribution in \(\tau^0_{G^{\pi, \mu}}\).

\begin{lem}\label{lem: contr between V and W}
Let \(\pi\in\mathcal{NC}(V)\), and let
\(W^\pi=\sqcup_{i=1}^{k-|\pi|+1} W_i^\pi\)
be the partition from Definition~\ref{def: NC(V)}.
Assume that for a partition \(\mu\in\mathcal P(W)\) there exist vertices
\(w_{j_1}\sim_\mu w_{j_2}\) with
\(w_{j_1}\in W_{i_1}^\pi\) and \(w_{j_2}\in W_{i_2}^\pi\) for some \(i_1\neq i_2\).
Then
\(\tau^0_{G^{\pi,\mu}}=0\).
\end{lem}

\begin{proof}
By Definition~\ref{def: NC(V)}, the subgraphs \(G (W_i^\pi), \ldots, G (W_{k - |\pi|+1}^\pi)\) are simple bipartite cycles connected through \(p \leq k - |\pi|\) vertices \(\tilde{v}_1, \ldots, \tilde{v}_p\in V\). As shown in the proof of Lemma~\ref{lem: contr non-cross part in V},
the graph \(G^\pi\) is an admissible graph whose blocks are precisely these subgraphs. Assume that there exists a partition \(\mu \in \mathcal{P}(W)\) such that \(w_{j_1} \sim_\mu w_{j_2}\), where \(w_{j_1} \in W_{i_1}^\pi\) and \(w_{j_2} \in W_{i_2}^\pi\) for \(i_1 \neq  i_2\). We claim that the resulting graph \(G^{\pi,\mu}\) is not admissible. First, suppose that \(G_{i_1}^\pi = G(W_{i_1}^\pi)\) and \(G_{i_2}^\pi =G(W_{i_2}^\pi)\) share a common vertex \(\tilde{v} \in V\). Since each block is a simple cycle, \(\deg_{G_{i_1}^\pi}(\tilde v)=\deg_{G_{i_2}^\pi}(\tilde v)=2\),
and hence \(\deg_{G^\pi}(\tilde v)\ge4\). After merging \(w_{j_1}\) and \(w_{j_2}\) according to Definition~\ref{def: merging operation}, the two subgraphs \(G_{i_1}^\pi\) and \(G_{i_2}^\pi\) merge into a single connected component \(\widetilde{G}^{\pi, \mu}\) with \(\deg_{\widetilde{G}^{\pi,\mu}}(\tilde{v}) = 4\). Definition~\ref{def: admissible graph} is therefore not satisfied. Next, suppose that \(G_{i_1}^\pi\) and \(G_{i_2}^\pi\) do not share a vertex in \(V\). Let  \(G_{i_3}^\pi, \ldots, G_{i_n}^\pi\) denote the minimal path of blocks connecting \(G_{i_1}^\pi\) to \(G_{i_2}^\pi\) through separating vertices. Merging \(w_{j_1}\) and \(w_{j_2}\) again produces a new connected component \(\widetilde{G}^{\pi, \mu}\), which merges the two blocks \(G_{i_1}^\pi\) and \(G_{i_2}^\pi\). As a consequence, the blocks \(\widetilde{G}^{\pi, \mu}, G_{i_3}^\pi, \ldots, G_{i_n}^\pi\) form a cycle in the block structure, since \(|V_{\widetilde{G}^{\pi, \mu}, G_{i_3}^\pi, \ldots, G_{i_n}^\pi}| = n-1\). Hence, \(G^{\pi, \mu}\) is no longer a block tree. In both cases, \(G^{\pi, \mu}\) fails to be admissible. Therefore, by Proposition~\ref{main3}, \(\tau^0_{G^{\pi,\mu}}=0\).
\end{proof}

As a consequence of Lemma~\ref{lem: contr between V and W}, the only partitions of \(W\) that gives a nonvanishing contribution are partitions of \(W_1^\pi, \ldots, W_{R_\pi}^\pi\), where we recall that \(R_\pi = k - |\pi| + 1 - |b(\pi)|\) and \(|b(\pi)|\) denotes the number of nearest neighbors in a same block of \(\pi\) (see Definition~\ref{def: partition 2}). Since each connected component \(G_i^\pi = G(W_i^\pi)\) is a simple bipartite cycle (see Remark~\ref{rmk: NC(V)}), in the following we may assume without loss of generality that \(v_1 \neq \cdots \neq v_k\) and \(\pi \in \mathcal{P}(W)\). 

\begin{lem} \label{lem: contr part in W}
If \(\pi = \{\{v_i\} \colon i \in [k]\}\), then for every \(\mu \in \mathcal{P}(W)\) it holds that
\[
\tau^0_{G^{\pi,\mu}} = 
\begin{cases}
C_{2k}(f) & \text{if} \enspace \mu = \{W\}, \\
\frac{1}{\psi^{|\mu|-1}} \sum_{P \in \mathcal{P}([R_\mu])} C_{(\widetilde{W}_i^P)_{i=1}^{|P|}}(f)  \mathbf 1_{\{\forall i,\; G(\widetilde W_i^P)\ \text{connected}\}} & \text{otherwise} ,
\end{cases}
\]
where 
\begin{itemize}
\item \(R_\mu = |c(\mu)| + 1 - S_\mu\), \(S_\mu = |b(\mu)|\), and \(c(\mu)\) are given by Definition~\ref{def: partition 2};
\item for \(\mu \in \mathcal{P}(W)\), the subsets \(W_1^\mu, \ldots, W_{R_\mu}^\mu\) denote the finest partition of \(W^\mu\) (see Definitions~\ref{def: NC(W)} and~\ref{def: P(W)});
\item for any partition \(P \in \mathcal{P}([R_\mu])\) of the set \([R_\mu]=\{1,\ldots, R_\mu\}\) with blocks \(B_1,\ldots, B_{|P|}\) we denote by \(\widetilde{W}_i^P\) the subset obtained by merging \(\{W_j^{\mu} \colon j \in B_i\}\) and similarly \(G(\widetilde{W}_i^P)\) the graph obtained by merging \(\{G (W_j^\mu) \colon j \in B_i\}\) (see Definition~\ref{def: merging operation}).
\end{itemize}
\end{lem}

\begin{proof}
We first observe that since \(\pi\) is the partition of singletons, the graph \(G^{\pi,\mu}\) is a double tree if and only if \(\mu\) is the singleton partition. In this case, \(W^\pi = \{\tilde{w}\}\) where \(\tilde{w}\) denotes the vertex obtained by merging \(w_1, \ldots, w_k\). According to Proposition~\ref{main3}, we obtain 
\[
\tau^0_{G^{\pi,\mu}}  = C_{\deg(\tilde{w})}(f) = C_{2k} (f).
\]
Now, consider a partition \(\mu\) such that \(\mu \neq \{W\}\). According to Remark~\ref{rmk: P(W)}, the graph \(G^\mu\) is a block tree with a single block. It is an admissible graph since items (a), (b), and (c) of Definition~\ref{def: admissible graph} are easily verified. Therefore, the limiting injective trace \(\tau^0_{G^{\pi,\mu}}\) is given by Proposition~\ref{main3}. We need to identify all the admissible decompositions of \(W^\mu\). By definition, \(\{W^\mu\} \in \mathcal{A}_1(W^\mu)\), so that the term
\[
\frac{1}{\psi^{|\mu| - 1}} C_{W^\mu}(f)
\]
contributes to \(\tau^0_{G^{\pi,\mu}}\). According to Definitions~\ref{def: NC(W)} and~\ref{def: P(W)}, to the partition \(\mu\) we define subsets \(W_1^\mu, \ldots, W_{R_\mu}^\mu\), which share \(q \leq k - |\pi|\) vertices \(\tilde{w}_1, \ldots, \tilde{w}_q \in W\). In particular, we have \(\{W_1^\mu, \ldots, W_{R_\mu}^\mu\} \in \mathcal{A}_{R_\mu}(W^\mu)\). This implies that the term
\[
\frac{1}{\psi^{|\mu| - 1}} C_{(W_i^\mu)_{i=1}^{R_\mu}}(f)
\]
contributes to \(\tau^0_{G^{\pi,\mu}}\). More generally, let \(P \in \mathcal{P}([R_\mu])\) denote a partition of the set \([R_\mu]=\{1, \ldots, R_\mu\}\) with blocks \(B_1, \ldots, B_{|P|}\) and let \(\widetilde{W}_i^P\) denote the subset obtained by merging \(\{W_j^\mu \colon j \in B_i\}\). We also define the corresponding subgraph \(G(\widetilde{W}_i^P)\) by merging the subgraphs \(\{G(W_j^\mu)\colon j \in B_i\}\) and removing the repeated copies of vertices and edges (see Definition~\ref{def: merging operation}). We easily note that if each component \(G(\widetilde{W}_i^P)\) is connected, then \(\{\widetilde{W}_1^P, \ldots, \widetilde{W}_{|P|}^P\} \in \mathcal{A}_{|P|}(W^\mu)\). In particular, we observe that if \(P = \{\{1, \ldots, R_\mu\}\}\) is the singleton partition, then \(\widetilde{W}^P = W^\mu\); while if \(P = \{\{i\} \colon 1 \leq i \leq R_\mu\}\) is the partition of singletons, then \(\widetilde{W}_i^P = W_i^\mu\) for all \(1 \le i \le R_\mu\). We therefore deduce that
\[
\tau^0_{G^{\pi,\mu}} = \frac{1}{\psi^{|\mu| - 1}} \sum_{P \in \mathcal{P}([R_\mu])} C_{(\widetilde{W}_i)_{i=1}^{|P|}} (f) \mathbf{1}_{\{\text{each} \: G(\widetilde{W}_i^P) \: \text{is connected}\}},
\]
as desired.
\end{proof}

We are now able to prove Proposition~\ref{prop: main cycle} by combining the previous results.

\begin{proof}[\textbf{Proof of Proposition~\ref{prop: main cycle}}]
From~\eqref{eq: moments}, we have
\[
\lim_{p,m,n \to \infty }  \frac{1}{p} \E \left [ \tr (Y_m Y_m^\top)^k \right ] = \sum_{\pi \in \mathcal{P}(V)}   \sum_{\mu \in \mathcal{P}(W)}  \tau^0_{G^{\pi,\mu}}.
\]
Combining Lemmas~\ref{lem: contr non-cross part in V},~\ref{lem: contr cross part in V}, and~\ref{lem: contr between V and W}, this sum reduces to
\[
\sum_{\pi \in \mathcal{P}(V)}   \sum_{\mu \in \mathcal{P}(W)}  \tau^0_{G^{\pi,\mu}} = \sum_{\pi \in \mathcal{NC}(V)} \phi^{k - |\pi|} C_2(f)^{S_\pi} \prod_{i=1}^{R_\pi} \left ( \sum_{\mu_i \in \mathcal{P}(W_i^\pi)} \tau^0_{G_i^{\pi, \mu_i} } \right),
\]
where \(G_i^{\pi, \mu_i}\) denotes the graph obtained from \(G_i^\pi = G(W_i^\pi)\) by identifying vertices in \(W\) which belong to the same block of \(\mu_i\). Finally, applying Lemma~\ref{lem: contr part in W} completes the proof of Proposition~\ref{prop: main cycle}.
\end{proof}

We now consider the case of symmetric \(\alpha\)-stable entries with \(\alpha=2\), where \(\E_W \left [ \exp(i \lambda W_{ij})\right ] = \exp(- n^{-1} \sigma_w^2 \lambda^2 / 2)\) for every \(\lambda \in \R\). Thus, \(\Phi ( \lambda ) = - \lambda^2 \sigma_w^2 /2 \). We show that Proposition~\ref{prop: main cycle} reduces to~\cite[Theorem 3.5]{benigni2021}. We introduce two parameters \(\theta_1(f)\) and \(\theta_2(f)\) which depends on the activation function \(f\):
\[
\begin{split}
\theta_1(f) & = \E_{Z \sim \mathcal{N}(0, \sigma_w^2 \sigma_x^2)} \left [ f^2(Z) \right ] =  \int_\R f^2(\sigma_w \sigma_x x) \frac{e^{-x^2/2}}{\sqrt{2\pi}} \textnormal{d}x,\\
\theta_2(f) & =  \left (\E_{Z \sim \mathcal{N}(0, \sigma_w^2 \sigma_x^2)} \left [ f'(Z) \right ] \right) ^2= \left (\sigma_w \sigma_x \int_\R f'(\sigma_w \sigma_x x) \frac{e^{-x^2/2}}{\sqrt{2 \pi}} \textnormal{d}x\right )^2.
\end{split} 
\]

\begin{lem} \label{lem: moments alpha=2}
Assume that \(\nu_w\) is the symmetric \(\alpha\)-stable law with \(\alpha=2\) (equivalently, \(\nu_w\) is centered Gaussian), and that \(\nu_x\) satisfies Assumption~\ref{hyp1}(b). Then, under Assumptions~\ref{hyp2}-\ref{hyp3}, for every \(k \in \N\), the \(k\)th moment \(\frac{1}{p} \tr M^k\) converges in expectation to
\[
m_k= \sum_{\pi \in \mathcal{NC}(V)} \frac{\phi^{k-|\pi|}}{\psi^{k-1}} \,\theta_1(f)^{S_\pi}  \prod_{i=1}^{R_\pi} \left ( \sum_{\mu_i \in \mathcal{NC}(W_i^\pi)} \psi^{|W_i^\pi|-|\mu_i|} \theta_1(f)^{S_{\mu_i}} \theta_2(f)^{|W_i^\pi| - S_{\mu_i}} \right ),
\]
where 
\begin{itemize} 
\item for any noncrossing partition \(P\) of \(W\) or \(V\), \(S_P = |b(P)|\) denotes the number of nearest neighbor pairs within a block of \(P\) (see Definition~\ref{def: partition 2}), and \(R_P = k - |P| + 1 - S_P\);
\item for \(\pi \in \mathcal{NC}(V)\), \(W_1^\pi, \ldots, W_{R_\pi}^\pi\) are the subsets of \(W^\pi\) with cardinality at least \(2\), as in Definition~\ref{def: NC(V)}. 
\end{itemize}
\end{lem}

In particular, Lemma~\ref{lem: moments alpha=2} shows that, for partitions
\(\pi \in \mathcal{P}(V)\) and \(\mu \in \mathcal{P}(W)\),
the limiting injective trace \(\tau^0_{G^{\pi,\mu}}\) is nonvanishing
only if the associated graph \(G^{\pi,\mu}\) is a bipartite cactus graph.

\begin{proof}
According to Definition~\ref{def: C_deg (f)} and Remark~\ref{rmk: C_deg(w) alpha=2}, if \(d\) is an even integer, then for \(\alpha=2\) the parameter \(C_d(f)\) simplifies to
\begin{equation} \label{eq: C_d(f,2)}
C_d(f) = \theta_1(f)^{d/2}.
\end{equation}
Furthermore, according to Definition~\ref{def: C_W (f)}, if \(G = (W \cup V, E)\) denotes the simple bipartite cycle, then for \(\alpha=2\) the parameter \(C_W(f)\) reduces to
\[
\begin{split}
C_W(f) & = \frac{(-\sigma_w^2)^{|W|}}{2^{|W|}(2\pi)^{2|W|}} \int_{\R^{2|W|}} \prod_{i=1}^{2|W|} \textnormal{d} \gamma_i \hat{f}(\gamma_i) e^{-\frac{\sigma_w^2 \sigma_x^2}{2}\gamma_i^2} \\
& \quad \times \E_X \left [ \prod_{i=2}^{|W|} (\gamma_{2(i-1)} X_{i-1} + \gamma_{2i-1}X_i)^2 (\gamma_1 X_1 + \gamma_{2 |W|} X_{|W|})^2\right ].
\end{split} 
\]
Note that \(C_W(f)\) is nonzero only for the term in the expectation that contains all parameters \(\gamma_i\). Indeed, if a term does not contain some parameter \(\gamma_i\), say \(\gamma_1\), then we can factorize the integral \(\int_{\R} \hat{f}(\gamma_1) e^{-\frac{\sigma_w^2\sigma_x^2}{2} \gamma_1^2} \textnormal{d} \gamma_1\) which vanishes since \(\hat{f}\) is odd by Assumption~\ref{hyp2}. As a consequence, \(C_W(f)\) is nonzero only for the term in the expectation given by 
\[
\E_X \left [2^{|W|} \prod_{i=1}^{2 |W|}  \gamma_i  X_i^2 \right ] = 2^{|W|} \sigma_x^{2|W|} \prod_{i=1}^{2 |W|} \gamma_i. 
\]
We therefore obtain 
\[
\begin{split}
C_W(f) & = \frac{(-\sigma_w^2 \sigma_x^2)^{|W|}}{(2\pi)^{2|W|}} \int_{\R^{2|W|}} \prod_{i=1}^{2|W|} \textnormal{d} \gamma_i \hat{f}(\gamma_i) \gamma_i e^{-\frac{\sigma_w^2 \sigma_x^2}{2}\gamma_i^2} = \left ( \frac{\sigma_w \sigma_x}{2 \pi} \int_{\R} i \gamma \hat{f}(\gamma)  e^{-\frac{\sigma_w^2 \sigma_x^2}{2}\gamma^2}  \textnormal{d} \gamma\right )^{2 |W|}.
\end{split} 
\]  
By the Fourier property \(\hat{f'}(\gamma) = i \gamma \hat{f}(\gamma)\), we have  
\begin{equation} \label{eq: C cycle (f,2)}
C_W(f)  = \left ( \frac{\sigma_w \sigma_x}{2 \pi} \int_\R \hat{f'}(\gamma) e^{-\frac{\sigma_w^2 \sigma_x^2}{2} \gamma^2} \textnormal{d}\gamma\right )^{2|W|} =  \left ( \frac{1}{\sqrt{2 \pi}}\int_\R f'(x) e^{-\frac{x^2}{2 \sigma_w^2 \sigma_x^2}} \textnormal{d} x \right )^{2|W|}  = \theta_2(f)^{|W|}.
\end{equation}
From Lemma~\ref{lem: contr non-cross part in V}, together with~\eqref{eq: C_d(f,2)} and~\eqref{eq: C cycle (f,2)}, it follows that if \(\mu = \{\{w_i\}, i \in [k]\}\) and \(\pi \in \mathcal{NC}(V)\), then
\begin{equation} \label{eq: alpha=2 V}
\tau^0_{G^{\pi,\mu}} =  \frac{\phi^{k - |\pi|}}{\psi^{k-1}} \theta_1(f)^{S_\pi}  \prod_{i=1}^{R_\pi}  \theta_2(f)^{|W_i^\pi|} =  \frac{\phi^{k - |\pi|}}{\psi^{k-1}} \theta_1(f)^{S_\pi}  \theta_2(f)^{k - S_\pi},
\end{equation}
where we used the identity \(\sum_{i=1}^{R_\pi}  |W_i^\pi| + S_\pi=k\). 

Now fix \(\pi = \{\{v_i\}, i \in [k]\}\) and let \(\mu \in \mathcal{P}(W)\). By Lemma~\ref{lem: contr part in W} and~\eqref{eq: C_d(f,2)}, if \(\mu=\{W\}\) then
\[
\tau^0_{G^{\pi,\mu}} = \theta_1(f)^k.
\]
For a general \(\mu\), recall from Lemma~\ref{lem: contr part in W} that \(\tau_{G^{\pi,\mu}}^0\) can be written as a sum over partitions
\(P\in\mathcal P([R_\mu])\),
\[
\tau^0_{G^{\pi,\mu}}
=
\frac{1}{\psi^{|\mu|-1}}
\sum_{P\in\mathcal P([R_\mu])}
C_{(\widetilde W_i^P)_{i=1}^{|P|}}(f)\,
\mathbf 1_{\{\forall i,\; G(\widetilde W_i^P)\ \text{connected}\}}.
\]
We now show that in the case \(\alpha=2\) only the contribution corresponding to the partition into singletons (i.e., \(P=\{\{1\},\ldots,\{R_\mu\}\}\)) is nonzero. For \(\alpha=2\), the coefficient \(C_{(\widetilde W_i^P)}(f)\) is given by
\[
C_{(\widetilde W_i^P)}(f) = \frac{1}{(2\pi)^{2k}} \int_{\R^{2k}} \prod_{e \in E^\mu} \prod_{i=1}^{m(e)} \textnormal{d} \gamma_e^i \hat{f}(\gamma_e^i) e^{\sum_{w \in W^\mu} \E_X \left [Z_w(\boldsymbol{\gamma})\right ]} \prod_{i=1}^{|P|} \E_X \left [ \prod_{w \in \widetilde{W}_i^P} Z_w(\boldsymbol{\gamma}) \right ],
\]
where 
\[
Z_w(\boldsymbol{\gamma}) = - \frac{\sigma_w^2}{2}  \left( \sum_{v \colon v \sim w} (\gamma_{(w,v)}^1 + \cdots + \gamma_{(w,v)}^{m((w,v))}) X_v\right )^2.
\]
Using \(\E_X \left [X_v X_{v'} \right] = \sigma_x^2 \mathbf{1}_{\{v=v'\}}\), we obtain
\[ 
\E_X \left [Z_w(\boldsymbol{\gamma})\right ] = - \frac{\sigma_w^2 \sigma_x^2}{2} \sum_{v \colon v \sim w} (\gamma_{(w,v)}^1 + \cdots + \gamma_{(w,v)}^{m((w,v))})^2.
\]
Furthermore, since \(v_1, \ldots, v_k\) are pairwise distinct, every edge in \(E^{\mu}\) has multiplicity at most two. Splitting the integration according to edge multiplicity yields
\[
\begin{split}
C_{(\widetilde W_i^P)_{i=1}^{|P|}}(f) =  \frac{1}{(2\pi)^{2k}} & \int_{\R^{2k}}  \prod_{e \in E^\mu \colon \atop m(e)=2}  \textnormal{d} \gamma_e^1\textnormal{d} \gamma_e^2 \hat{f}(\gamma_e^1) \hat{f}(\gamma_e^2) e^{-\frac{ \sigma_x^2 \sigma_w^2}{2}(\gamma_e^1+\gamma_e^2)^2} \\
&\quad \times\prod_{e \in E^\mu \colon \atop m(e)=1}  \textnormal{d} \gamma_e\hat{f}(\gamma_e) e^{-\frac{ \sigma_x^2 \sigma_w^2}{2}\gamma_e^2} \prod_{i=1}^{|P|} \E_X \left [ \prod_{w \in \widetilde{W}_i^P} Z_w(\boldsymbol{\gamma}) \right ].
\end{split}
\]
Because \(\hat f\) is odd (Assumption~\ref{hyp2}), any monomial that does not contain some
\(\gamma_{e_0}\) with \(m(e_0)=1\) yields a vanishing contribution, since
\[
\int_{\mathbb R}\hat f(\gamma_{e_0})e^{-\frac{\sigma_w^2\sigma_x^2}{2}\gamma_{e_0}^2} \mathrm{d} \gamma_{e_0}=0.
\]
If \(P\neq \{\{1\},\ldots,\{R_\mu\}\}\), then at least one block of \(P\) merges two distinct components \(W_j^\mu\) and \(W_{j'}^\mu\).
In this case, in the expansion of
\(\prod_{i=1}^{|P|}\E_X[\prod_{w\in\widetilde W_i^P} Z_w(\boldsymbol\gamma)]\),
there is no monomial containing all variables \(\gamma_e\) with \(m(e)=1\). Hence \(C_{(\widetilde W_i^P)}(f)=0\) for all \(P\neq \{\{1\},\ldots,\{R_\mu\}\}\). Consequently, only the finest partition contributes and we obtain
\[
\tau^0_{G^{\pi,\mu}}
=
\frac{1}{\psi^{|\mu|-1}}\,C_{(W_i^\mu)_{i=1}^{R_\mu}}(f),
\]
as claimed. Moreover, if \(\mu\) is crossing, then there is no term in the expansion that contains all \(\gamma_e\) for edges \(e\) with \(m(e)=1\), and therefore \(C_{(W_i^\mu)_{i=1}^{R_\mu}}(f)=0\). Hence it suffices to restrict to \(\mu \in \mathcal{NC}(W)\). For \(\mu \in \mathcal{NC}(W)\), the surviving contribution corresponds exactly to taking, in each factor \( \E_X \left [ \prod_{w \in W_i^\mu} Z_w(\boldsymbol{\gamma}) \right ]\), the monomial that contains all \(\gamma_e\) with \(m(e)=1\) once, yielding
\[
\begin{split}
C_{(W_i^\mu)_{i=1}^{R_\mu}}(f)  & =  \frac{1}{(2\pi)^{2S_\mu}} \int_{\R^{2S_\mu}} \prod_{e \in E^{\mu} \colon \atop m(e)=2}  \textnormal{d} \gamma_e^1\textnormal{d} \gamma_e^2 \hat{f}(\gamma_e^1) \hat{f}(\gamma_e^2) e^{-\frac{ \sigma_x^2 \sigma_w^2}{2}(\gamma_e^1+\gamma_e^2)^2} \\
& \quad \times \frac{1}{(2\pi)^{2k - 2S_\mu}} \int_{\R^{2 k - 2S_\mu}}  \prod_{i=1}^{R_\mu} (-\sigma_w^2 \sigma_x^2)^{|W_i^\pi|} \prod_{e \in E_i^{\mu} \colon \atop m(e)=1}  \textnormal{d} \gamma_e \gamma_e \hat{f}(\gamma_e) e^{-\frac{ \sigma_x^2 \sigma_w^2}{2}\gamma_e^2} \\
& =  C_{2S_\mu}(f) \left ( \frac{\sigma_w \sigma_x }{2\pi} \int_\R i \gamma \hat{f}(\gamma) e^{- \sigma_w^2 \sigma_x^2 \gamma^2 /2} \textnormal{d} \gamma \right )^{2(k -S_\mu)}\\
& =  \theta_1(f)^{S_\mu} \theta_2(f)^{k - S_\mu},
\end{split}
\]
where we used~\eqref{eq: C_d(f,2)} and~\eqref{eq: C cycle (f,2)}. Combining this with~\eqref{eq: alpha=2 V} as done in the proof of Proposition~\ref{prop: main cycle} yields the desired result.  
\end{proof}

\subsection{Convergence of matrix moments in probability}\label{sec:cov}

Having proved the convergence of the expected moments of the empirical eigenvalue distribution \(\hat{\mu}_M\) in the previous subsection, we now address the convergence in probability of these moments, thereby completing the proof of Theorem~\ref{main1}.

\begin{lem} 
The variances of the moments vanish asymptotically. Specifically, under Assumptions~\ref{hyp1}-\ref{hyp3}, for every integer \(k \ge 1\), 
\[
\lim_{m, p, n \to \infty} \var \left ( \frac{1}{p}  \tr  M^k \right) = 0.
\]
\end{lem}

The convergence in probability of the matrix moments follows directly from this result by applying Chebyshev's inequality. 

\begin{proof}
Fix \(k \ge 1\). Expanding the trace, we write
\[
\var \left ( \frac{1}{p}  \tr M^k \right) = \frac{1}{p^2} \sum_{1 \leq i_1, \ldots, i_k \leq p \atop 1 \leq i'_1, \ldots, i'_k \leq p} \sum_{1 \leq j_1, \ldots, j_k \leq m \atop 1 \leq j_1' , \ldots, j_k' \leq m} \left ( \E \left [ P(\boldsymbol{i},\boldsymbol{j})P(\boldsymbol{i}',\boldsymbol{j}') \right ] -\E \left [ P(\boldsymbol{i},\boldsymbol{j})\right ]  \E \left [ P(\boldsymbol{i}',\boldsymbol{j}') \right ]\right ),
\]
where for multi-indices \(\boldsymbol{i} = (i_1,\ldots, i_k)\) and \(\boldsymbol{j} = (j_1,\ldots, j_k)\),
\[
P(\boldsymbol{i},\boldsymbol{j}) = \prod_{\ell=1}^k Y_{i_\ell j_\ell} Y_{i_{\ell+1} j_\ell}, \quad \text{with} \enspace i_{k+1} = i_1,
\]
and \(P(\boldsymbol{i}',\boldsymbol{j}')\) is defined analogously. The expectation \(\E \left [ P(\boldsymbol{i},\boldsymbol{j})P(\boldsymbol{i}',\boldsymbol{j}') \right ]\) factorizes unless there are identifications among indices in \(\boldsymbol{i}, \boldsymbol{i}'\) or \(\boldsymbol{j}, \boldsymbol{j}'\). Thus, the variance vanishes unless there exists at least one pair \((\ell, \ell')\) such that \(i_\ell = i_{\ell'}'\) or \(j_\ell = j_{\ell'}'\). This corresponds to overlap between the two bipartite cycle graphs associated with the terms \(P(\boldsymbol{i},\boldsymbol{j})\) and \(P(\boldsymbol{i}',\boldsymbol{j}')\), where vertices in these graphs are identified. Consequently, the variance can be rewritten as
\begin{equation} \label{var}
\var \left ( \frac{1}{p}  \tr M^k \right) = \frac{1}{p^2} \sum_{\exists \ell, \ell' \colon (i_\ell = i'_{\ell'}) \vee (j_\ell = j'_{\ell'})} \left ( \E \left [ P(\boldsymbol{i},\boldsymbol{j})P(\boldsymbol{i}',\boldsymbol{j}') \right ] -\E \left [ P(\boldsymbol{i},\boldsymbol{j})\right ]  \E \left [ P(\boldsymbol{i}',\boldsymbol{j}') \right ]\right ).
\end{equation}
We focus on the first term in~\eqref{var}, which can be expanded as follows:
\[
\begin{split}
& \frac{1}{p^2} \sum_{\exists \ell, \ell' \colon (i_\ell = i'_{\ell'}) \vee (j_\ell = j'_{\ell'})} \E \left [ P(\boldsymbol{i},\boldsymbol{j})P(\boldsymbol{i}',\boldsymbol{j}') \right ] \\
& = \frac{1}{p} \sum_{\pi, \mu \in \mathcal{P}([2k]) \colon \atop \exists \ell, \ell' \in [k] \colon (\ell \sim_\pi k+\ell' ) \vee (\ell \sim_\mu k+\ell')} \frac{1}{p} \sum_{\boldsymbol{i} \cup \boldsymbol{i}' \in \mathcal{I}_\pi} \sum_{\boldsymbol{j} \cup \boldsymbol{j}' \in \mathcal{J}_\mu}  \E \left [ P(\boldsymbol{i},\boldsymbol{j})P(\boldsymbol{i}',\boldsymbol{j}') \right ] \\
& = \frac{1}{p} \sum_{\pi, \mu \in \mathcal{P}([2k]) \colon \atop \exists \ell, \ell' \, \text{s.t.} \, \ell \sim_\pi k+\ell' \, \text{or} \, \ell \sim_\mu k+\ell'}   \tau_{p,m,n}^0 \left [ T^{\pi,\mu} \right ],
\end{split}
\]
where \(\mathcal{P}([2k])\) is the set of partitions of \([2k]=\{1, \ldots, 2k\}\), and \(\mathcal{I}_\pi\) (respectively, \(\mathcal{J}_\mu\)) is the set of multi-indices \(\boldsymbol{i} \cup\boldsymbol{i}' \) in \(\{1, \ldots, p\}^{2k}\) (respectively, \(\boldsymbol{j} \cup\boldsymbol{j}' \) in \(\{1, \ldots, m\}^{2k}\) ) such that \(r \sim_\pi s\) if and only if \(x_r = x_s\), where \(x\) represents either \(i\) or \(i'\). Here, \(T^{\pi,\mu}\) is the connected bipartite test graph of length \(4k\) defined by identifications encoded in \(\pi\) and \(\mu\), and \(\tau_{p,m,n}^0 \left [ T^{\pi,\mu} \right ]\) denotes its mean injective trace (see Definition~\ref{def: traffic}). By Theorem~\ref{main3}, as \(p,m,n \to \infty\) such that \(n/m \to \phi, n/p \to \psi\), the mean injective trace \(\tau_{p,m,n}^0 \left [ T^{\pi,\mu} \right ]\) converges to a real number \(\tau_{G^{\pi,\mu}}^0\). As a result, each summand satisfies
\[
\frac{1}{p} \tau_{p,m,n}^0 \left [ T^{\pi,\mu} \right ] = \mathcal{O} \left ( \frac{1}{p} \right),
\]
and therefore the whole sum is \(\mathcal{O}(1/p)\). 
The second term in~\eqref{var} is given by
\[
\frac{1}{p^2} \sum_{\exists \ell, \ell' \colon (i_\ell = i'_{\ell'}) \vee (j_\ell = j'_{\ell'})} \E \left [ P(\boldsymbol{i},\boldsymbol{j})\right ]  \E \left [ P(\boldsymbol{i}',\boldsymbol{j}') \right ].
\]
We note that the expectation \( \E \left [ P(\boldsymbol{i},\boldsymbol{j})\right ] \) depends only on the identifications among indices within \(\boldsymbol{i}\) and \(\boldsymbol{j}\), so it is unaffected by overlaps with \(\boldsymbol{i}'\) or \(\boldsymbol{j}'\). 
This invariance allows the second term to be rewritten as
\[
\begin{split}
\frac{1}{p^2} \sum_{\exists \ell, \ell' \colon (i_\ell = i'_{\ell'}) \vee (j_\ell = j'_{\ell'})} \E \left [ P(\boldsymbol{i},\boldsymbol{j})\right ]  \E \left [ P(\boldsymbol{i}',\boldsymbol{j}') \right ] & = \mathcal{O} \left ( \frac{1}{p} + \frac{1}{m} \right) \left ( \frac{1}{p} \sum_{1 \le i_1, \ldots, i_k \le p} \sum_{1 \le j_1, \ldots, j_k \le m}\E \left [ P(\boldsymbol{i},\boldsymbol{j})\right ] \right)^2\\
& = \mathcal{O} \left ( \frac{1}{p} + \frac{1}{m} \right) \left(\tau_{p,m,n} \left [ T_{\text{cycle}}\right ] \right)^2 ,
\end{split}
\]
where the second equality follows from~\eqref{traffic cycle} and \(T_{\text{cycle}}\) denotes the simple bipartite cycle of length \(2k\). The scaling factor \(\mathcal{O} \left ( \frac{1}{p} + \frac{1}{m} \right)\) arises from the fact that identifications of the form \(i_\ell = i_{\ell'}'\) reduces the number of free \(\boldsymbol{i}'\)-indices by one and thus introduces a factor \(1/p\); similarly an identification of the form \(j_\ell=j'_{\ell'}\) introduces a factor \(1/m\).
Since the number of possible overlap positions \((\ell,\ell')\) is finite (depending only on \(k\)),
this yields the factor \(\mathcal{O} \left ( \frac{1}{p} + \frac{1}{m} \right)\). Finally, since the traffic trace \(\tau_{p,m,n} \left [ T_{\text{cycle}}\right ] \) converges by Theorem~\ref{main3}, we obtain 
\[
\frac{1}{p^2} \sum_{\exists \ell, \ell' \colon (i_\ell = i'_{\ell'}) \vee (j_\ell = j'_{\ell'})} \E \left [ P(\boldsymbol{i},\boldsymbol{j})\right ]  \E \left [ P(\boldsymbol{i}',\boldsymbol{j}') \right ] = \mathcal{O} \left ( \frac{1}{p} + \frac{1}{m} \right) .
\]
Combining the two bounds yields
\[
\var \left ( \frac{1}{p}  \tr  M^k \right) = \mathcal{O}  \left ( \frac{1}{p} \right).
\]
Letting \(p \to \infty\), the result follows.
\end{proof}

\section{Almost sure weak convergence of the empirical spectral measure} \label{sec:concen}

In this section, we prove Theorem~\ref{main1bis}, which establishes the almost sure weak convergence of the empirical spectral measure \(\hat{\mu}_M\). The key step in the proof is to show that the moments \(m_k\), given in Proposition~\ref{prop: main cycle}, do not grow too quickly, thereby ensuring they define a unique probability measure. This is achieved by verifying Carleman's condition, which states that a sequence of moments \((m_k)_{k \in \N}\) uniquely determines a probability measure if \(\sum_{k=1}^\infty |m_k|^{-1/k} = + \infty\). The moments \(m_k\) depend on certain graph-related parameters---specifically, \(C_d(f)\) and \(C_{(W_k)_{k=1}^K}(f)\)---defined in~\eqref{eq: C_deg (f)} and~\eqref{eq: C_{W_i}}, respectively. These parameters, in turn, depend on $\Phi$ from Assumption~\ref{hyp1}. We begin by deriving bounds for these parameters under the two cases of Assumption~\ref{hyp4}. In Lemma~\ref{lem: bound parameter} below, items (a) and (b) correspond exactly to conditions (a) and (b) in Assumption~\ref{hyp4}.

\begin{lem} \label{lem: bound parameter}
For every even integer \(d\), let \(C_d (f)\) denote the parameter defined in~\eqref{eq: C_deg (f)}. Let \(G = (W \cup V, E)\) be a finite, connected bipartite graph \(G = (W \cup V, E)\). For every collection of subsets \(W_1, \ldots, W_K \subseteq W\) with \(|W_k| \ge 2\) and with a nontrivial intersection (i.e., for every \(i \in [K]\), \(\cup_{j \neq i} W_j\cap W_i \neq  \emptyset \)), let \(C_{(W_k)_{k=1}^K} (f)\) denote the parameter defined in~\eqref{eq: C_{W_i}}. Then, the following hold.
\begin{enumerate}
\item[(a)] Assume there exists a constant \(a >0\) such that \(|\Phi(\lambda)| \le a\) for all \(\lambda \in \R\). Then, there exist universal constants \(c,C > 0\) such that 
\[
|C_d (f)| \le c \quad \textnormal{and} \quad | C_{(W_k)_{k=1}^K} (f) | \le C a^{\sum_{k=1}^K |W_k|}.
\]
\item[(b)] Assume instead that \(\Phi(\lambda) = - \sigma^\alpha |\lambda|^\alpha\) with \(\alpha \in ]0,2[\) and \(\sigma >0\), and that \(\nu_x\) is the centered normal distribution with variance \(\sigma_x^2\). Then, there exist universal constants \(c,C > 0\) such that 
\[
|C_d (f)| \le c \quad \textnormal{and} \quad | C_{(W_k)_{k=1}^K} (f) | \le  C^{\sum_{k=1}^K |W_k|} \prod_{k=1}^K |W_k|!.
\]
\end{enumerate}
\end{lem}

Note that item (a) is satisfied by sparse Wigner matrices~\eqref{hada}, whereas item (b) concerns L\'evy matrices~\eqref{levy}.

\begin{proof}
We begin by proving statement (a). According to~\eqref{eq: C_deg (f)} and~\eqref{thehypf}, we have 
\[
|C_d (f)| \le c \left ( \frac{1}{2 \pi} \int_\R |\hat{f}(t)| \textnormal{d}t \right)^d \le \tilde{c},
\]
for some universal constant \(\tilde c >0\). Similarly, from~\eqref{eq: C_{W_i}} and using the fact that the random variables \(Z_w\), as defined in~\eqref{eq: Z_w}, are also bounded by \(a\), there exists \(C >0\) such that
\[
| C_{(W_k)_{k =1}^K} (f) | \le C a^{\sum_{k=1}^K |W_k|},
\]
as desired. We now consider statement (b). The bound for \(C_d (f)\) follows again from~\eqref{eq: C_deg (f)} and~\eqref{thehypf} since we have \(e^{- \sigma^\alpha \E_X \left [ |\sum_{i=1}^{d/2} (\gamma_i^1 + \gamma_i^2) X_i|^\alpha \right ]}  \le 1\). For the parameter \(C_{(W_i)_{i \le K}} (f) \), substituting \(\Phi(\lambda) = - \sigma^\alpha |\lambda|^\alpha\) results in 
\[
C_{(W_k)_{k=1}^K}(f) = \frac{1}{(2\pi)^{|\cup_k^K E_k|}} \int_{\R^{|\cup_{k=1}^K E_k|}} \prod_{e \in \cup_{k=1}^K E_k} \prod_{i=1}^{m(e)} \textnormal{d} \gamma_e^i \hat{f}(\gamma_e^i) e^{\sum_{w \in \cup_{k=1}^K W_k} \E_X \left [Z_w (\boldsymbol{\gamma})\right ]} \prod_{k=1}^K \E_X \left [ \prod_{w \in W_k} Z_w (\boldsymbol{\gamma}) \right ],
\]
where 
\[
Z_w (\boldsymbol{\gamma})= -\sigma^\alpha \left | \sum_{v \in V \colon v \sim w} (\gamma_{(w,v)}^1 + \cdots + \gamma_{(w,v)}^{m((w,v))}) X_v \right |^\alpha.
\]
For each \(w\), we note that
\[
\sum_{v \in V \colon v \sim w} (\gamma_{(w,v)}^1 + \cdots + \gamma_{(w,v)}^{m((w,v))}) X_v \stackrel{d}{=} \sigma_x \sqrt{\sum_{v \in V\colon v \sim w} (\gamma_{(w,v)}^1 + \cdots + \gamma_{(w,v)}^{m((w,v))})^2} G_w,
\]
where the random variables \(G_w\sim \mathcal{N}(0,1)\) may be correlated. Consequently, 
\begin{equation} \label{exp1}
\E_X \left [Z_w (\boldsymbol{\gamma}) \right ] = - \sigma^\alpha \sigma_x^\alpha \beta_\alpha \left | \sum_{v \in V \colon v \sim w }  (\gamma_{(w,v)}^1 + \cdots + \gamma_{(w,v)}^{m((w,v))})^2\right |^{\alpha/2} ,
\end{equation}
where \(\beta_\alpha \coloneqq \E \left [ |G_w|^\alpha \right ]  =2^{\alpha/2} \frac{\Gamma \left ( \frac{\alpha +1}{2} \right)}{\sqrt{\pi}}\). Similarly, we have
\begin{equation} \label{exp2}
\begin{split}
\E_X \left [ \prod_{w \in W_k} Z_w (\boldsymbol{\gamma}) \right ] & = \prod_{w \in W_k} (-\sigma^\alpha \sigma_x^\alpha) \left | \sum_{v \in V \colon v \sim w }  (\gamma_{(w,v)}^1 + \cdots + \gamma_{(w,v)}^{m((w,v))})^2\right |^{\alpha/2}  \E \left [ \prod_{w \in W_k} |G_w|^\alpha \right ] \\
& \le  (- \sigma^\alpha \sigma_x^\alpha)^{|W_k|} \beta_{\alpha |W_k|} \prod_{w \in W_k} \left | \sum_{v \in V \colon v \sim w }  (\gamma_{(w,v)}^1 + \cdots + \gamma_{(w,v)}^{m((w,v))})^2\right |^{\alpha/2} ,
\end{split}
\end{equation}
where we applied H\"{o}lder's inequality. Combining~\eqref{exp1} and~\eqref{exp2}, we obtain
\[
\begin{split}
\left | C_{(W_k)_{k=1}^K}(f) \right | & \le  c^{|\cup_{k=1}^K E_k|} (-\sigma^\alpha \sigma_x^\alpha)^{\sum_{k=1}^K |W_k|} \prod_{k=1}^K \beta_{\alpha |W_k|} \\
& \quad \times \frac{1}{(2\pi)^{|\cup_k^K E_k|}}\int_{\R^{|\cup_{k=1}^K E_k|}} \prod_{e \in \cup_{k=1}^K E_k} \prod_{i=1}^{m(e)} \textnormal{d} \gamma_e^i 
\frac{C_{2}}{1+|\gamma_{e}^{i}|^{2}}\\
& \quad \times \prod_{k=1}^K  \prod_{w \in W_k} 
e^{- \sigma^\alpha \sigma_x^\alpha \beta_\alpha   \left | \sum_{v \in V \colon v \sim w }   \left ( \sum_{j=1}^{m((w,v))} \gamma_{(w,v)}^j \right )^2 \right |^{\alpha/2} }  
\left | \sum_{v \in V \colon v \sim w }  \left ( \sum_{j=1}^{m((w,v))} \gamma_{(w,v)}^j \right )^2\right |^{\alpha/2} ,
\end{split}
\]
where we used that \(|\hat{f}(t)| \le C_{2}/(1+|t|^2)\) according to~\eqref{thehypf}. Using that there exists a universal constant $C_0$ such that uniformly,
$$e^{- \sigma^\alpha \sigma_x^\alpha \beta_\alpha   \left | \sum_{v \in V \colon v \sim w }   \left ( \sum_{j=1}^{m((w,v))} \gamma_{(w,v)}^j \right )^2 \right |^{\alpha/2} }  \\
\left | \sum_{v \in V \colon v \sim w }  \left ( \sum_{j=1}^{m((w,v))} \gamma_{(w,v)}^j \right )^2\right |^{\alpha/2}\le C_0,$$
we deduce that there exists a finite constant $C_{1}$ so that 
\[
\begin{split}
& \left | C_{(W_k)_{k=1}^K}(f) \right |  \le C_1^{\sum_{k=1}^K |W_k|} \prod_{k=1}^K \beta_{\alpha |W_k|} .
\end{split}
\]
This completes the proof for statement (b).
\end{proof}

We now show that the limiting moments of the symmetrized random matrix \(H \in \R^{(p+m) \times (p+m)}\), defined by
\[
H \coloneqq 
\begin{pmatrix}
0 & Y_m^\top\\
Y_m & 0
\end{pmatrix},
\]
satisfy Carleman's condition. Let \(\{\lambda_i (H),1\le i\le m+p\}\) denote the eigenvalues of \(H\), and let \(\{\lambda_{i},1\le i\le p\}\) be the eigenvalues of \(M=Y_m Y_m^\top\). The \(2p\) nonzero eigenvalues of \(H\) correspond to \(\pm \sqrt{\lambda_1}, \ldots, \pm \sqrt{\lambda_p}\). The empirical spectral measure \(\hat{\mu}_{H}=\frac{1}{m+p}\sum_{i=1}^{m+p}\delta_{\lambda_{i}(H)}\) of \(H\) is therefore related to the empirical spectral measure \(\hat{\mu}_M=\frac{1}{p} \sum_{i=1}^{p} \delta_{\lambda_{i}}\) of \(M\) via
\begin{equation}\label{lin}
\int f(x^{2}) \text{d} \hat{\mu}_{H} (x)  = \frac{2p}{m+p}\int f(x) \text{d} \hat{\mu}_M (x) + \frac{m-p}{m+p} f(0),
\end{equation}
for any bounded and continuous function \(f\). In particular, we see that for every integer number $k\ge 1$, 
\[
\tilde{m}_k \coloneqq  \lim_{m,p,n\rightarrow\infty} \int_\R x^k \textnormal{d} \hat{\mu}_H (x) = 
\begin{cases}
\frac{2\phi}{\phi+\psi} m_{k/2} & \, \text{if} \: k \: \text{is even},\\
0 & \, \text{if} \: k \: \text{is odd},
\end{cases}
\]
where we recall that \(\phi = \lim_{m,p,n\to \infty} n/m\) and \(\psi = \lim_{m,p,n\to \infty} n/p\). The following result shows that the sequence of moments \((\tilde{m}_k)\) satisfies Carleman's condition, thereby defining a unique limiting measure \(\hat{\mu}\).

\begin{lem}
Under Assumption~\ref{hyp4}, there exists a finite constant \(C >0\) such that for every integer number \(k\),
\[
|\tilde m_{k}|\le (C k)^k.
\]
As a result, there exists a unique probability measure \(\tilde\mu\) with moments \(\tilde m_k\). Furthermore, there exists a unique probability measure $\mu$ with moments $m_k$. Finally, the empirical spectral measure \(\hat{\mu}_M\) of \(M\) converges weakly, both in expectation and in probability, to \(\mu\).
\end{lem}

\begin{proof}
Recall that the number of noncrossing partitions of a set of size \(k\) is given by the Catalan number, which has the explicit formula \(C_k = \frac{(2k)!}{(k+1)! k!}\) and asymptotically behaves as \(C_k \sim \frac{4^k k^{-3/2}}{\sqrt{\pi}}\). On the other hand, the number of possible partitions of a set of size \(k\) is given by the Bell number \(B_k\), which can be expressed as
\[
B_k = \sum_{\ell=1}^k S(k,\ell),
\]
where \(S(k,\ell)\) are the Stirling numbers of the second kind and count the possible partitions of a set of size \(k\) into \(\ell\) nonempty subsets. Similarly, the ordered Bell numbers (or Fubini numbers) can be computed from the Stirling numbers of the second kind via
\[
a(k) = \sum_{\ell=1}^k S(k,\ell) \ell!,
\]
and asymptotically behave as \(a(k) \sim \frac{1}{2} k! \log(2)^{-(k+1)}\).

From Proposition~\ref{prop: main cycle}, the moments \(m_k\) are expressed as a sum over noncrossing partitions \(\pi\) of \(V\), over partitions \(\mu\) of the subsets \(W_i^\pi\), and over partitions of the set \(\{1, \ldots,R_\mu\}\). Using the bounds derived in Lemma~\ref{lem: bound parameter}, we claim 
\begin{equation} \label{claim}
|m_k| \leq C^k k! C_k a(k),
\end{equation}
for some constant \(C >0\). Substituting the asymptotic expressions for \(C_k\) and \(a(k)\), we obtain
\[
|m_k | \sim (\tilde C k)^{2k},
\]
for some constant \(\tilde C >0\). This implies that the moments \((\tilde{m}_k)\) satisfies the bound
\[
|\tilde m_k| \leq  (\tilde C k)^k.
\]
This bound ensures that the sequence \((\tilde{m}_k)\) satisfies Carleman's condition. Thus, there exists a unique probability measure \(\tilde \mu\) such that, for every integer \(k\),
\[
\tilde m_k = \int x^k \text{d} \tilde\mu (x).
\]
By construction, \(\tilde \mu\) is symmetric. Furthermore, for every $k\ge 1$, the moments \(m_k\) satisfy
\[
m_k = \frac{\psi+\phi}{2\phi} \int x^{2k} \text{d} \tilde\mu(x) =\int x^{k} \text{d} \mu(x),
\]
where the probability measure \(\mu\) is given by 
\[
\mu=\frac{\phi+\psi}{2\phi}  x^2 \#\tilde\mu + \frac{\phi-\psi}{2\phi} \delta_0.
\]
Here, \(x^2 \#\tilde\mu\) denotes the pushforward measure of \(\tilde \mu\) under the mapping \(x \mapsto x^2\). This last point follows directly from Theorem~\ref{main1}, as convergence in moments is stronger than weak convergence. 

We now prove the claim~\eqref{claim}. By Proposition~\ref{prop: main cycle} and Lemma~\ref{lem: bound parameter}, there are constants \(c,\tilde{c}>0\) such that
\[
|m_k|  \le  c^k |\mathcal{NC}(V)| \sum_{\mu \in \mathcal{P}(W), \mu \neq \{W\}} \sum_{P \in \mathcal{P}([R_\mu])} \tilde{c}^{\sum_{i=1}^{|P|} |\widetilde{W}_i^P|} \prod_{i=1}^{|P|} |\widetilde{W}_i^P|!,
\]
where we used the upper bound in item (b) of Lemma~\ref{lem: bound parameter}, as it represents the least favorable case. Recall that for a partition \(\mu \in \mathcal{P}(W)\), we obtain subsets \(W_1^\mu, \ldots, W_{R_\mu}^\mu\), which form the finest partition of \(W^\mu\) (see Definitions~\ref{def: NC(W)} and~\ref{def: P(W)}). For every partition \(P \in \mathcal{P}([R_\mu])\), merging the corresponding subsets results in the subsets \(\widetilde{W}_1^P, \ldots, \widetilde{W}_{|P|}^P\). The number \(R_\mu\) satisfies \(R_\mu \le k - |\mu| + 1\). For further clarification, we refer the reader to Proposition~\ref{prop: main cycle}. According to item (b) of Proposition~\ref{prop: combinatorics}, for any partition \(P \in \mathcal{P}([R_\mu])\), \(\sum_{i=1}^{|P|} |\widetilde{W}_i^P| =|W^\mu| + |P| - 1  \le |\mu| + R_\mu - 1 \le k\). We bound \(\prod_{i=1}^{|P|} |\widetilde{W}_i^P|! \) above by \(k!\) as follows:
\[
k! \ge (|\widetilde{W}_1^P| +\cdots + |\widetilde{W}_{|P|}^P|)!  > \prod_{i=1}^{|P|} |\widetilde{W}_i^P|!,
\]
where we used the fact that \(|\widetilde{W}_i^P| > 1\). This leads to the inequality
\[
|m_k|  \le c^k C_k \tilde{c}^k k!  \sum_{\mu \in \mathcal{P}(W), \mu \neq \{W\}} \sum_{P \in \mathcal{P}([R_\mu])} 1,
\]
Our goal is to show that 
\[
\sum_{\mu \in \mathcal{P}(W), \mu \neq \{W\}} \sum_{P \in \mathcal{P}([R_\mu])} 1 \le C' a(k),
\]
for some constant \(C'>0\). Let \(\mu\) denote a partition of \(W\) into \(\ell\) blocks, where \(2 \le \ell \le |W|=k\) (the case \(\ell=1\) corresponds to \(\mu = \{W\}\)). The number of such partitions is given by \(S(k,\ell)\). Thus, the sum can be bounded as
\[
\sum_{\mu \in \mathcal{P}(W), \mu \neq \{W\}} \sum_{P \in \mathcal{P}([R_\mu])} 1 \le C' \sum_{\ell=2}^k S(k,\ell) R_\ell! ,
\]
where \(C'>0\) is a constant and \(R_\ell\) denotes the largest possible number \(R_\mu\) associated with a partition \(\mu\) having \(\ell\) blocks. We aim to show that \(R_\ell \le \ell\). Recall that for a partition \(\mu\), the number \(R_\mu\) is defined by 
\[
R_\mu = |c(\mu)| - |b(\mu)| +1, 
\]
where \(b(\mu)\) denotes the collection of pairs of nearest neighbor elements lying in a same block of \(\mu\), and \(c(\mu)\) the collection of pairs of next elements within a block such that, for every other pair of elements in a same block, the two pairs do not intersect (see Definition~\ref{def: partition 2}). First, consider \(\ell \ge  \lceil \frac{k+1}{2}\rceil\). In this case, we have
\[
R_\ell \le |c(\mu)| + 1 \le k - \ell + 1 \le k - \left \lceil \frac{k+1}{2} \right \rceil + 1 \le \ell,
\]
where we used the fact that \(|c(\mu)| \le k-\ell\), with equality when \(\mu\) is noncrossing. Now, let \(\ell \le \left \lceil \frac{k-1}{2} \right \rceil\). We want to show that \(|c(\mu)| - |b(\mu)| \le \ell-1\).
\begin{itemize}
\item We first consider noncrossing partitions. In this case, we have \(|c(\mu)| = k - \ell\), so proving \(R_\ell \le \ell\) is equivalent to showing that \(|b(\mu)| \ge k - 2 \ell +1\). We proceed by induction on \(\ell\), showing that \(|b(\mu)| \ge k - 2 \ell + 2\). If \(\ell=2\), then there are two blocks containing \(r\) and \(k-r\) elements, where \(r \ge 1\). Since \(\mu\) is noncrossing, the blocks contain \(r-1\) and \(k-r-1\) pairs of consecutive elements, respectively. This gives \(|b(\mu)| = k-2\), which satisfies the inequality. Assume the claim holds for some \(2 \le \ell \le  \left \lceil \frac{k-1}{2} \right \rceil -1\). To increase the number of blocks from \(\ell\) to \(\ell+1\), we split an existing block into two smaller blocks. This operation reduces the number of nearest neighbor pairs by at most \(2\). If a block contains exactly two consecutive elements, say \(x_i \sim x_{i+1}\), moving \(x_{i+1}\) to a new block results in a noncrossing partition with \(\ell+1\) blocks and the number of pairs of consecutive elements decreases by \(1\). If a block contains at least three consecutive elements, say \(x_i \sim x_{i+1} \sim x_{i+2}\), we can move \(x_{i+1}\) to a new block while keeping \(x_i \sim x_{i+2}\) in the original block. This results in a noncrossing partition with \(\ell+1\) blocks and decreases the number of pairs of consecutive elements by \(2\). Thus, after increasing \(\ell\) by \(1\), the number of nearest neighbor pairs can decrease at most by \(2\), leading to \(|b(\mu)|  \ge k - 2 \ell + 2 - 2 =  k - 2 (\ell+1) + 2\). This completes the induction and shows that for noncrossing partitions, \(|c(\mu)| - |b(\mu)| \le \ell-2\). 
\item To extend the inequality to crossing partitions, we consider the set \(c(\mu) \backslash b(\mu)\). Note that by Remark~\ref{rmk: partition}, $b(\mu)\subset c(\mu)\cup\{(x_{1},x_{k})\}$ so that 
\(|c(\mu)| - |b(\mu)|  \le  |c(\mu) \backslash b(\mu)|\le  |c(\mu)| - |b(\mu)| +1\). In particular, if \(\mu\) is noncrossing, by the previous item, we have
\begin{equation}\label{bc}
|c(\mu) \backslash b(\mu)| \le \ell-1.
\end{equation}
\item We now prove the inequality~\eqref{bc} for crossing partitions by induction on the number of crossings \(n_\text{c}\) in \(\mu\). Let \(\mu\) be a partition with \(\ell\) blocks and \(n_\text{c}\) crossings. This means that there exist \(x_p < x_q < x_r < x_s\) such that \(x_p \sim x_r \not \sim x_q \sim x_s\), i.e., \((x_p,x_r)\) and \((x_q,x_s)\) are intersecting pairs. We note that by moving \(x_q\) into the block containing $x_{q+1}$   while keeping \(x_s\) in its original block, results in a new  partition \(\mu'\) with the same number \(\ell\) of blocks. Moreover, \(|c(\mu') \backslash b(\mu')| \ge |c(\mu) \backslash b(\mu)|\), since if we create a new element of nearest neighbors in $b(\mu')$, it is also included in $c(\mu')$, while we have at most the same number of elements in $c(\mu')$. Proceeding inductively, we find a sequence of partitions $\mu=\mu_{1},\ldots,\mu_{n_{c}}$  of partitions, where $\mu_{n_{c}}$ is noncrossing and has $\ell$ blocks, satisfying 
$$ |c(\mu) \backslash b(\mu)|= |c(\mu_{1}) \backslash b(\mu_{1})|\le |c(\mu_{2}) \backslash b(\mu_{2})|\le\cdots \le  |c(\mu_{n_{c}}) \backslash b(\mu_{n_{c}})|\le \ell-1,$$
where the last step follows from~\eqref{bc}.
\end{itemize}
This completes the proof that \(R_\ell \le \ell\) for all \(\ell\), thereby proving the claim~\eqref{claim}.
\end{proof}

We conclude by showing that the empirical measure of the eigenvalues converges almost surely using concentration of measure estimates. 

\begin{lem} \label{concen} 
Let $h \colon \mathbb{R}\to \mathbb{R}$ be such that the function $g(x) \coloneqq h(x^{2})$ has finite total variation norm $\|g\|_{\textnormal{TV}} < \infty$. Then, for every $\epsilon >0$,
\begin{equation*}
\mathbb{P} \left ( \left | \int _{\mathbb{R}}h(x) \, \textnormal{d} \hat{\mu}_{M}(x) -\mathbb{E}\left [\int _{\mathbb{R}}h(x) \,
\textnormal{d} \hat{\mu}_{M}(x) \right ] \right |\ge \epsilon \|g\|_{
\textnormal{TV}} \right ) \le 4 \exp \left ( -\frac{\epsilon^{2}p}{2^{5}}  \right).
\end{equation*}
Here, the total variation norm is defined by
$ \lVert g \rVert _{\textnormal{TV}} = \sup \sum _{k \in \mathbb{Z}} |g(x_{k+1})
- g(x_{k})|$, where the supremum runs over all sequences
$(x_{k})_{k \in \mathbb{Z}}$ such that $x_{k+1} \ge x_{k}$ for every
$k \in \mathbb{Z}$.
\end{lem}

Applying the Borel-Cantelli Lemma, the above concentration inequality ensures that the convergence of the empirical spectral measure holds almost surely, thus proving Theorem~\ref{main1bis}.

\begin{proof} 
Hereafter, we assume that $g(x) = h(x^{2})$ has total variation norm bounded by one, which is sufficient by homogeneity. In light of~\eqref{lin}, it suffices to show concentration for $\int g \, \text{d} \hat{\mu}_{H}$ in order to obtain concentration for $\int h \, \text{d} \hat{\mu}_{M}$. We proceed using a martingale argument based on Azuma-Hoeffding’s inequality, following the approach of~\cite[Lemma C.2]{bordenave2011bis}.

Conditionally on $X$, the matrix $H$ has independent column vectors $\{(Y_{m}(i),0), i\le p\}$ and independent row vectors $\{(Y_{m}(i)^\top,0), i\le p\}$. Therefore, we can apply the same ideas of the concentration result from~\cite[Lemma C.2]{bordenave2011bis}. To this end, we successively apply Azuma-Hoeffding's inequality with respect to integration w.r.t.\ $W$ and $X$. Let $\mathcal F_{k}$ denote the filtration generated by $\{W_{i},i\le k\}\cup \{X_{\ell j}, j\le m, \ell\le n\}$ for \(1 \le k \le p\), with \(W_i = (W_{i1}, \ldots, W_{i m})^\top \in \R^m\), and let \(\sigma_X\) denote the sigma-algebra of the entries of the matrix \(X\), so that \(\mathcal{F}_0 = \sigma_X\). We denote by \(\E_k \left [\cdot \right ] \coloneqq \E_W \left [ \cdot | \mathcal{F}_k \right ]\) the conditional expectation w.r.t.\ \(\mathcal{F}_k\).  Then, consider the martingale
\[
D_k \coloneqq \E_k \left [ \frac{1}{p}\tr( g(H))\right],
\]
where \(D_p = \frac{1}{p} \tr( g(H)) \) and \(D_0 =  \E \left [ \frac{1}{p} \tr( g(H)) |\sigma_X \right ] \). By construction, we have 
\[
\sum_{k=1}^p (D_k - D_{k-1}) = \frac{1}{p} \tr( g(H)) - \E \left [ \frac{1}{p}\tr( g(H))|\sigma_X \right ] .
\]
We now  bound the martingale differences uniformly:
$$\Delta_{k} \coloneqq D_{k} - D_{k-1}=\mathbb E_{k} \left [ \frac{1}{p}\tr( g(H))- \frac{1}{p}\tr( g(H')) \right],$$
where $H$ and $H'$ are coupled such that they are constructed with the same $W_{i}$ for every $i\le k-1$ and every $i\ge k+1$. This implies that $H-H'$ has rank at most two, as they differ at most by one column vector and one row vector. By Weyl's interlacing property, it follows that 
$$\left|\tr( g(H))- \tr( g(H'))\right| \le \mathrm{rank}(H-H') \lVert g\rVert_{\mathrm{TV}}\le 2,$$
where we used the assumption that the total variation norm of $g$ is bounded by one. Therefore,
$$|\Delta_{k}|\le \frac{2}{p}.$$ 
Applying Azuma-Hoeffding inequality (see e.g.~\cite[Lemma 1.2]{mcdiarmid}), we obtain that uniformly with respect to the entries of $X$,
$$\mathbb P\left(\left|\frac{1}{p} \tr( g(H)) - \E \left [ \frac{1}{p}\tr( g(H))|\sigma_X \right ]\right|\ge \frac{\epsilon}{2} \right)\le 2 \exp \left ( - \frac{p\epsilon^{2}}{2^5} \right).$$
We can apply the same strategy to integrate $\mathbb E \left [\frac{1}{p}\tr( g(H))|\sigma_{X} \right]$ w.r.t.\ $X$  
and obtain concentration w.r.t.\ the $X_{i}$'s by considering the martingale $\mathbb E \left [\frac{1}{p}\tr( g(H))|\sigma(X_{i},i\le k) \right]$ to obtain 
$$\mathbb{P} \left(\left | \E \left [\frac{1}{p}\tr( g(H)) \right ] - \E \left [ \frac{1}{p}\tr( g(H))|\sigma_{X} \right ] \right|\ge \frac{\epsilon}{2}\right)\le 2 \exp \left ( - \frac{p\epsilon^{2}}{2^5} \right),.$$
Combining the two concentration estimates and applying the normalization $g(x)=h(x^{2})/\|h(\cdot^{2})\|_{\mathrm{TV}}$, we obtain the stated concentration inequality.
\end{proof}

\printbibliography
\end{document}